\documentclass[12pt]{article}

\usepackage[margin=2cm, top=2cm, bottom=2cm]{geometry}

\usepackage[utf8]{inputenc} 
\usepackage[T1]{fontenc}    
\usepackage{url}            
\usepackage{booktabs}       
\usepackage{amsfonts}       
\usepackage{nicefrac}       
\usepackage{microtype}      
\usepackage{mathrsfs}  

\usepackage{amsmath,amssymb,amsthm,xcolor}
\usepackage{array}
\usepackage{float}
\newtheorem{theorem}{Theorem}[section]

\newtheorem{remark}[theorem]{Remark}
\newtheorem{lemma}[theorem]{Lemma}
\newtheorem{corollary}[theorem]{Corollary}

\newtheorem{proposition}[theorem]{Proposition}
\newtheorem{definition}[theorem]{Definition}
\newtheorem{fact}[theorem]{Fact}
\newtheorem{assumption}[theorem]{Assumption}
\numberwithin{equation}{section}
\usepackage{graphicx}
\usepackage{subcaption}
\usepackage{mdframed}
\usepackage{lipsum}
\usepackage{wrapfig}
\usepackage[hidelinks,hypertexnames=false]{hyperref}
\usepackage{tikz}
\usepackage{enumitem}
\allowdisplaybreaks
\usepackage{epstopdf}
\usepackage{algorithmic}
\usepackage{algorithm}
\usepackage{booktabs}
\usepackage{changepage}
\usepackage{enumitem}
\usepackage{indentfirst}
\usepackage{stmaryrd} 
\usepackage{soul}

\newcommand{\co}{\operatorname{co}}

\newcommand{\runinheading}[1]{\par\smallskip\noindent\emph{\underline{#1.}}\space}
\newcommand{\tocgroup}[1]{%
  \addtocontents{toc}{\protect\addvspace{0.8em}%
    \protect\noindent\protect\textbf{#1}\protect\par
    \protect\addvspace{0.25em}}}

\title{Convergence of difference inclusions: a diameter criterion and step-size conditions}
\author{\large  Lexiao Lai\thanks{\url{lai.lexiao@hku.hk}, Department of Mathematics, The University of Hong Kong, Hong Kong.} \and Mingzhi Song\thanks{\url{songmingzhi123@gmail.com}, Department of Mathematics, The University of Hong Kong, Hong Kong.}}
\date{}

\begin{document}

\maketitle
\vspace*{-5mm}
\begin{center}
    \textbf{Abstract}
    \end{center}
    \vspace*{-4mm}
 \begin{adjustwidth}{0.2in}{0.2in}
~~~~We study bounded realizations of discrete difference inclusions with
set-valued increments and additive errors.  Our results have two parts.
First, we give a general convergence criterion in which changes in a
convergent scalar quantity control the diameters of local sequence segments
near an accumulation point.  We also develop a stratified descent framework
for verifying this control.  For convergent realizations, the framework
yields a stationarity condition for the limit from outer limits of the scaled
update map.  When applied to first-order methods for minimizing
locally Lipschitz objectives definable in polynomially bounded o-minimal
structures, the framework yields convergence to critical points for bounded
sequences generated by the inexact subgradient, momentum, and stochastic
subgradient methods with step sizes of order $k^{-1}$, under the corresponding
error and noise conditions.


Second, we study first-order methods with polynomial step sizes
of order $k^{-a}$.  At the square-summability boundary
\(a=1/2\), we construct a locally Lipschitz semialgebraic objective for which
the exact subgradient method generates a bounded, nonconvergent sequence whose
accumulation points satisfy our active-geometry assumptions.  Under these
assumptions, bounded sequences generated by the momentum method converge for
\(1/2<a\leq1\), while bounded sequences generated by the stochastic
subgradient method converge almost surely for \(2/3<a\leq1\).  The momentum
range is sharp for this method class, while the stochastic range matches
recent full-sequence convergence results for smooth nonconvex objectives.


\end{adjustwidth}

\clearpage
\tableofcontents
\clearpage

\tocgroup{Main text}

\section{Introduction}
We consider discrete dynamics governed by a difference inclusion:
\begin{equation}\label{eq:DI_with_noise}
    x_{k+1} \in x_k + F(x_k,\alpha_k) + e_k,\quad \forall k\in \mathbb{N},
\end{equation}
where $F\colon\mathbb{R}^n\times \mathbb{R}_+
\rightrightarrows\mathbb{R}^n$ is a set-valued update map that collects the
admissible increments at the current iterate $x_k$ and step size $\alpha_k$,
while $e_k$ is an additive error representing stochasticity or inexact
computation.  This template encompasses a wide range of iterative schemes,
including discretizations of differential inclusions
\cite{aubin1984differential,dontchev1989error,dontchev1992difference,clarke1990}
and stochastic approximations
\cite{benaim2005stochastic,borkar2008stochastic}.  In nonsmooth optimization,
such dynamics arise naturally because many algorithms involve discontinuous
or set-valued update maps defined through subdifferentials
\cite{clarke1990,rockafellar1998variational}; Section~\ref{sec:first-order}
details several examples from first-order methods.

Our aim is to provide sufficient conditions under which bounded realizations
$(x_k)$ of \eqref{eq:DI_with_noise} converge.  Most existing approaches rely
on a (quasi-)monotone quantity along the dynamics.  If the update map is a
contraction, then $(x_k)$ converges to the unique fixed point by Banach's fixed
point theorem.  In the same vein, Fejér monotonicity \cite{fejer1922lage} and
its variants \cite{bauschke2001weak,combettes2008fejer} require the existence
of $x^*\in \mathbb{R}^n$ with
$|x_{k+1}-x^*|\leq |x_k-x^*|$ for all $k\in\mathbb{N}$ up to summable
errors; if $x^*$ is also a limit point, then $(x_k)$ must converge to it.
This principle applies, for example, to the subgradient method for convex
functions \cite{alber1998projected}.  Other routes invoke a Lyapunov function
$L$ not directly tied to distance.  Under the usual compactness and invariance
hypotheses, strict decrease of $L$ away from a unique invariant equilibrium
$x^*$ implies $x_k\to x^*$ \cite{liapounoff1907probleme}; with multiple
minimizers, convergence can still be deduced when the decrease in $L$ is
controlled relative to (sub)gradient norms and a
Kurdyka--\L{}ojasiewicz inequality of $L$ holds
\cite{absil2005convergence,attouch2009convergence,attouch2013convergence}.
These arguments bound the total length of the sequence by variations of $L$;
convergence of $(L(x_k))_{k\in\mathbb{N}}$ then yields convergence of
$(x_k)$.  Recent work also establishes convergence under approximately
decreasing (nonmonotone) potentials \cite{li2023convergence}.  Such approaches,
however, rarely apply to the general difference inclusions in
\eqref{eq:DI_with_noise}.

To address this limitation, a substantial literature studies continuous-time
limits of \eqref{eq:DI_with_noise}, whose solutions can sometimes be governed
by potential functions
\cite{benaim2005stochastic,drusvyatskiy2015curves,bolte2020conservative}.
This line of thought dates back to ODE methods for stochastic approximations
\cite{ljung1977analysis,kushner1977general}.  When a potential satisfies
appropriate Sard-type properties, one can show convergence of its values and
that limit points, which may not be unique, are critical relative to $F$
\cite{duchi2018stochastic,davis2020stochastic,bolte2022long,xiao2023stochastic,bolte2025inexact}.
Nevertheless, such arguments do not guarantee convergence of $(x_k)$ itself.
For the subgradient method, $F(x,\alpha)=-\alpha\,\partial f(x)$ and $e_k=0$,
where $\partial f\colon\mathbb{R}^n\rightrightarrows\mathbb{R}^n$ is the
Clarke subdifferential \cite{clarke1975} of a locally Lipschitz objective
$f\colon\mathbb{R}^n\to\mathbb{R}$.  Even when $f$ is monotone along the
continuous-time limits, divergent dynamics have been constructed in
\cite{rios2022examples}.  Under stronger geometric assumptions on $f$
(definability in a polynomially bounded o-minimal structure; see
Section~\ref{sec:first-order}), convergence of sequences generated by the
subgradient method was recently established by the authors
\cite{lai2025diameter} for decreasing step sizes of order $k^{-1}$.  The proof
bounds the sequence diameter by function-value changes plus higher-order
accumulations of step sizes.  Since function values converge by prior results
and the step-size contributions eventually vanish, the sequence must converge.
Two questions remained: whether this diameter argument extends to more general
update rules and perturbations, and whether full-sequence convergence persists
beyond step sizes of order $k^{-1}$.

The second question is delicate even for the vanilla subgradient method.
Beyond the evidently necessary requirements that the step sizes diminish and
are nonsummable, related results often impose additional conditions.
Subsequential convergence of the vanilla subgradient method holds for any
sequence of vanishing, nonsummable step sizes
\cite{bolte2022long}, whereas the stochastic counterpart additionally requires
square-summability \cite{davis2020stochastic}.  For convex objectives,
square-summability, together with nonsummability, is sufficient for convergence
of the subgradient method \cite{alber1998projected}.  However, whether
square-summability is necessary or sufficient for full-sequence convergence in
the locally Lipschitz definable setting remains elusive.

Our contributions are organized in two parts.  The first develops a general
diameter-based framework for proving convergence of the difference inclusion
\eqref{eq:DI_with_noise} and characterizing its limit points.  We then apply
this framework to first-order optimization methods.
Theorem~\ref{thm:criterion} gives a local convergence criterion: a sequence
converges to an accumulation point once there exists a scalar quantity whose variation controls the diameters of local
segments starting near that point.  We then verify this control for
realizations of difference inclusions through a stratified descent framework
(Definition~\ref{def:stratified_descent}), obtaining the global diameter bound
in Theorem~\ref{thm:diameter}.  Crucially, the framework requires an
approximately monotone quantity only along projections of the iterates onto
smooth manifolds, rather than along the iterates themselves.  This viewpoint
synthesizes and generalizes the techniques underlying our recent work
\cite{lai2025diameter}.  Appendix~\ref{sec:technical_comparison} gives a
result-level correspondence, compares the common proof mechanisms, and
identifies the additional arguments needed here.
Propositions~\ref{prop:converge->critical} and~\ref{prop:G-in-coHx}
characterize the limit points of convergent realizations through outer limits
of the scaled update map $F$.  Applying the framework to first-order minimization of locally Lipschitz
objectives definable in polynomially bounded o-minimal structures yields
convergence of inexact subgradient, momentum, and stochastic subgradient
methods with step sizes of order $k^{-1}$; see Corollary~\ref{cor: converge} and
Theorems~\ref{theorem:momentum-converge} and~\ref{theorem:1overk}.  As discussed
in Remark~\ref{remark:discrete}, the arguments in Sections~\ref{sec:criterion}
and~\ref{sec:first-order} are entirely discrete and do not rely on
continuous-time methods.

The second part asks how far full-sequence convergence extends for first-order
methods with polynomial step sizes of order $k^{-a}$.  At the square-summability boundary
\(a=1/2\), we construct a locally Lipschitz semialgebraic objective for which
the exact subgradient method generates a bounded, nonconvergent sequence
(Theorem~\ref{thm:counterexample-half}).  Every accumulation point of this
sequence satisfies the local active-geometry conditions used in our convergence
results.  Under these conditions at a prescribed accumulation point, bounded
sequences generated by the momentum and stochastic subgradient methods
converge.
Table~\ref{tab:convergence-regimes} summarizes the precise step-size ranges.
The momentum range is sharp at its lower endpoint for this method class
(Sections~\ref{sec:square_summability_obstruction} and~\ref{sec:active_momentum}), while the stochastic range matches recent full-sequence
convergence results for smooth nonconvex objectives
\cite[Corollary~5.1]{qiu2024convergence},
\cite[Theorem~4.2 and Remark~4.3]{milzarek2023convergence}, and
\cite[Theorem~1.7]{dereich2024convergence}
(Remark~\ref{remark:smooth_case}).  

\begin{table}[H]
\centering
\small
\caption{Convergence ranges for bounded sequences generated by the momentum
and stochastic subgradient methods for locally Lipschitz definable objectives with
$\alpha_k=\Theta((k+1)^{-a})$.  The active setting assumes
that the prescribed accumulation point satisfies
Assumption~\ref{assumption:active_aiming_verdier};
the conclusion for the stochastic subgradient method holds almost surely.}
\label{tab:convergence-regimes}
\setlength{\tabcolsep}{4pt}
\renewcommand{\arraystretch}{1.15}
\begin{tabular}{@{}
>{\centering\arraybackslash}m{0.27\textwidth}
>{\centering\arraybackslash}m{0.16\textwidth}
>{\centering\arraybackslash}m{0.20\textwidth}
>{\centering\arraybackslash}m{0.29\textwidth}@{}}
\toprule
\textbf{Method} & \textbf{Update} & \textbf{General setting} & \textbf{Active setting} \\
\midrule
\begin{tabular}[c]{@{}c@{}}Momentum method\end{tabular}
& \begin{tabular}[c]{@{}c@{}}\eqref{eq:momentum-inexact}\end{tabular}
& \begin{tabular}[c]{@{}c@{}}$a=1$\\
  {\footnotesize (Theorem~\ref{theorem:momentum-converge})}\end{tabular}
& \begin{tabular}[c]{@{}c@{}}$\tfrac12<a\leq1$\\
  {\footnotesize (Theorem~\ref{theorem:active_momentum})}\end{tabular} \\
\addlinespace[0.35em]
\begin{tabular}[c]{@{}c@{}}Stochastic subgradient\\method\end{tabular}
& \begin{tabular}[c]{@{}c@{}}\eqref{eq:inexact stoc}\end{tabular}
& \begin{tabular}[c]{@{}c@{}}$a=1$\\
  {\footnotesize (Theorem~\ref{theorem:1overk})}\end{tabular}
& \begin{tabular}[c]{@{}c@{}}$\tfrac23<a\leq1$\\
  {\footnotesize (Theorem~\ref{theorem:active_stochastic})}\end{tabular} \\
\bottomrule
\end{tabular}
\end{table}

\subsection{Notation}
We use $\mathbb{N}:=\{0,1,2,\dots\}$ and $\mathbb{R}_+:=[0,\infty)$ throughout.
On $\mathbb{R}^n$, the symbol $|\cdot|$ denotes the Euclidean norm and, for linear maps, its induced operator norm; the associated inner product is $\langle\cdot,\cdot\rangle$.
For $a\in\mathbb{R}^n$ and $r>0$, the sets $B(a,r)$ and $\mathring B(a,r)$ are the closed and open balls centered at $a$ with radius $r$, respectively; we set $B(a,0):=\{a\}$.
For $A\subset\mathbb{R}^n$, we write $\overline{A}$, $\mathring{A}$, and $\co A$ for its closure, interior, and convex hull. For a stratum $M$, its frontier is denoted by $\partial M:=\overline{M}\setminus M$.
We use the Minkowski enlargement $B(A,r):=A+B(0,r)$, with $B(A,0):=A$.
For nonempty sets $A,B\subset\mathbb{R}^n$ and $x\in\mathbb{R}^n$, set $d(x,A):=\inf_{y\in A}|x-y|$ and $d(A,B):=\inf_{a\in A,\,b\in B}|a-b|$.
We adopt the convention $d(x,\emptyset):=1$.
The metric projection is $P_A(x):=\operatorname*{argmin}_{y\in A}|x-y|$; whenever this set is a singleton, we identify it with its unique element.
We also write $|A|:=\sup_{a\in A}|a|$ for the supremal norm of $A$ and $\operatorname{diam}(A):=\sup_{a,b\in A}|a-b|$ for its diameter.
If $m,p\in\mathbb{N}$ and $f\colon A\to\mathbb{R}^m$, with $A$ open, then $f$ is $C^p$ when its $p$th Fréchet derivative exists and is continuous on $A$; we write $Df$ for the Jacobian of $f$.
If $M\subset\mathbb{R}^n$ is a smooth embedded manifold and $x\in M$, then $T_M(x)$ and $N_M(x)$ denote its tangent and normal spaces at $x$. If $f\colon\mathbb{R}^n\to\mathbb{R}$ is $C^2$ on $M$, then $\nabla_M f$ denotes its Riemannian gradient on $M$.
For integers $a\leq b$, set $\llbracket a,b\rrbracket:=\{a,\dots,b\}$; for a sequence $(x_k)$, write $x_{\llbracket a,b\rrbracket}:=\{x_a,\dots,x_b\}$.
Given $g\colon\mathbb{R}^n\to\mathbb{R}$, $v\in\mathbb{R}$, and $\star\in\{\leq,\geq,<,>,=\}$, let $[g\star v]:=\{x\in\mathbb{R}^n:g(x)\star v\}$.
Finally, $\operatorname{sgn}\colon\mathbb{R}\to\{-1,0,1\}$ denotes the sign function, and, for $\theta\in(0,1)$, we write
\(\varphi_\theta(t):=\operatorname{sgn}(t)|t|^{1-\theta}\).

\section{A diameter framework for convergence}\label{sec:criterion}
We begin with the convergence principle used throughout the paper.  It only
requires diameter control in one fixed neighborhood of an accumulation point.

\begin{theorem}[Diameter criterion]
\label{thm:criterion}
Let \((x_k)\subset\mathbb R^n\), let \(x^*\) be an accumulation point, and
let \(L\) be a real-valued function on a set containing
\(\{x_k:k\in\mathbb N\}\).  Suppose that \(L(x_k)\) converges and
\(|x_{k+1}-x_k|\to0\).  If there is an open neighborhood \(U\) of \(x^*\)
such that
\begin{equation}
 \lim_{\delta\downarrow0}\ 
 \limsup_{\substack{k_1,k_2\to\infty,\ k_1<k_2\\
 x_{k_1}\in B(x^*,\delta),\ 
 x_{\llbracket k_1,k_2\rrbracket}\subset U}}
 \left(
 \operatorname{diam}(x_{\llbracket k_1,k_2\rrbracket})
 +L(x_{k_2})-L(x_{k_1})\right)\leq0,
 \label{eq:diam_cri}
\end{equation}
Then \(x_k\to x^*\).
\end{theorem}

\begin{proof}
Suppose otherwise.  Choose \(r>0\) so that \(B(x^*,2r)\subset U\) and
\((x_k)\) leaves \(\mathring B(x^*,r)\) infinitely often.  Fix
\(\delta\in(0,r)\).  Since \(x^*\) is an accumulation point, there are
arbitrarily late indices \(p_j\) with \(x_{p_j}\in B(x^*,\delta)\).  Let
\(q_j>p_j\) be the first subsequent index outside
\(\mathring B(x^*,r)\).  The increments vanish, so
\(x_{\llbracket p_j,q_j\rrbracket}\subset B(x^*,2r)\subset U\) for all
large \(j\).  Moreover,
\[
 \operatorname{diam}(x_{\llbracket p_j,q_j\rrbracket})
 \geq |x_{q_j}-x_{p_j}|\geq r-\delta.
\]
Convergence of \(L(x_k)\) makes the endpoint difference tend to zero.
Therefore the inner upper limit in \eqref{eq:diam_cri} is at least
\(r-\delta\).  Letting \(\delta\downarrow0\) contradicts
\eqref{eq:diam_cri}.
\end{proof}

For optimization methods, convergence of \(L(x_k)\) is often known before
full-sequence convergence, while diminishing updates give
\(|x_{k+1}-x_k|\to0\).  A Kurdyka--\L{}ojasiewicz estimate then bounds the
diameter in \eqref{eq:diam_cri}, typically with \(L=\psi\circ f\)
\cite{absil2005convergence,attouch2009convergence,attouch2013convergence,
josz2023global,li2023convergence}. To handle cases in which no monotone quantity is available, we develop a framework to produce diameter bounds for sequences generated by the difference inclusion \eqref{eq:DI_with_noise}. In contrast to the approaches outlined above, we only ask for an approximate descent quantity along projected sequences, onto smoothly stratified pieces of the domain. If, in addition, the projected length admits an upper bound expressed via a power function composed with the objective (see \eqref{eq:y_k^i}), then a diameter bound for the original sequence follows (Theorem~\ref{thm:diameter}). This framework extends the diameter estimate for sequences generated by the subgradient method in \cite{lai2025diameter} and applies to common first-order methods for locally Lipschitz objectives definable in polynomially bounded o-minimal structures (see Section~\ref{sec:first-order}). To formalize the assumptions, we recall the definition of stratifications \cite{van1996geometric,trotman2020stratification}.

\begin{definition}[\(C^p\) stratification]\label{def:stratification}
Let $M\subset\mathbb{R}^n$ and $p\in\mathbb{N}$. A \emph{$C^p$ stratification} of $M$ is a finite partition $\mathcal{M}=\{M_i\}_{i\in I}$ of $M$ into connected $C^p$ manifolds $M_i\subset\mathbb{R}^n$ (the strata) that satisfy the frontier condition:
whenever $i\neq j$ and $\overline{M_i}\cap M_j\neq\emptyset$, one has $M_j\subset \partial M_i$ and $\dim M_j<\dim M_i$.
\end{definition}
For definable strata, the dimension inequality follows from
$\dim(\overline A\setminus A)<\dim A$ for every nonempty definable set
$A$ \cite[Result~4.7(1)]{van1996geometric}. Finite definable $C^p$
stratifications satisfying this condition can be chosen compatible with any
prescribed finite family of definable sets
\cite[Result~4.8(1)]{van1996geometric}; see also
\cite{trotman2020stratification}.
We now state a condition under which diameter bounds can be obtained.

\begin{definition}[Stratified descent] \label{def:stratified_descent}
We say that a difference inclusion \eqref{eq:DI_with_noise} exhibits
stratified descent in $X\subset \mathbb{R}^n$ if there exist constants $\hat{\alpha},\iota>0$ and $\tau\in(0,1]$, together with:
\begin{enumerate}
    \item\label{item:f} a Lipschitz continuous function $f\colon\widetilde X\to\mathbb{R}$, for some $\widetilde X \supset X$;
    \item\label{item:psi} a power function \(\psi\colon\mathbb{R}\to\mathbb{R}\)
    defined by \(\psi(t):=c\operatorname{sgn}(t)|t|^{1-\theta}\), where
    \(c>0\) and \(\theta\in(0,1)\);
    \item\label{item:g} a pair of error functions
    \[
        (g_1,g_2)\colon
        (\mathbb{R}_+)^{\mathbb{N}}\times(\mathbb{R}^n)^{\mathbb{N}}
        \to(\mathbb{R}_+\cup \{\infty\})^2;
    \]
    \item\label{item:strata} a $C^3$ stratification $\{M_i\}_{i=1}^{\overline{T}}$ of $\widetilde X$ and constants
    $c_i,\beta_i,\gamma_i>0$ satisfying, for all $i,j\in\llbracket1,\overline{T}\rrbracket$,
    \(0<\beta_i<\gamma_i(1-\theta)<\tau(1-\theta)\), and
    \(M_j\subset\partial M_i \Longrightarrow \beta_i>\gamma_j\).
\end{enumerate}
such that the following holds:
\begin{enumerate}\item For any $i,j\in\llbracket1,\overline{T}\rrbracket$,
    \[ \overline{M_i}\cap M_j=M_i\cap\overline{M_j}=\emptyset\quad \implies\quad
    \bigcup_{\alpha\in(0,\min\{1,\hat\alpha\}]}\mathcal{N}(i,\alpha)
    \cap
    \bigcup_{\alpha\in(0,\min\{1,\hat\alpha\}]}\mathcal{N}(j,\alpha)
    =\emptyset,\]
where \[
\mathcal{N}(i,\alpha)
:=
X\cap
\left(
B(M_i,c_i\alpha^{\beta_i})
\setminus
\bigcup_{\ell:\,M_\ell\subset \partial M_i}
B(M_\ell,c_\ell\alpha^{\gamma_\ell})
\right).
\]
\item For every realization $(x_k)$ of \eqref{eq:DI_with_noise} with
step sizes $(\alpha_k)_{k\in\mathbb{N}}$ satisfying
$\sup_k\alpha_k\leq\hat{\alpha}$, set
\(g_{\ell,k}:=g_\ell((\alpha_m)_{m\geq k},(e_m)_{m\geq k})\) for
\(\ell\in\{1,2\}\) and \(k\in\mathbb{N}\).
Then the following hold:
\begin{enumerate}[label=(\roman*),leftmargin=*]
    \item\label{item:small_step}
    For every \(k\in\mathbb N\),
    \(|g_{1,k+1}-g_{1,k}|\leq\iota\alpha_k^\tau\).  If \(x_k\in X\),
    then \(|x_{k+1}-x_k|\leq\iota\alpha_k^\tau\).

    \item\label{item:stratified_descent}
    For every \(i\in\llbracket1,\overline T\rrbracket\) and
    \(k\in\mathbb N\) satisfying \(x_k\in\mathcal N(i,\alpha_k)\), the
    projection \(y_k^i:=P_{M_i}(x_k)\) is single-valued.  Moreover, for every
    \(i\in\llbracket1,\overline T\rrbracket\) and \(K\geq1\) such that
    \(x_k\in\mathcal N(i,\alpha_k)\) for all
    \(k\in\llbracket0,K\rrbracket\), the projected block satisfies
\begin{equation}
\label{eq:y_k^i}
        \sum_{k=0}^{K-1}|y_k^i-y_{k+1}^i|
        \leq
        \psi(f(y_0^i)+g_{1,0})
        -\psi(f(y_K^i)+g_{1,K})
        +\sum_{k=0}^{K-1}g_{2,k}
        +\iota\max_{k\in\llbracket0,K-1\rrbracket}\alpha_k^\tau .
\end{equation}
\end{enumerate}
\end{enumerate}
\end{definition}
Let us unpack Definition~\ref{def:stratified_descent}. The neighborhoods
\(\mathcal{N}(i,\alpha)\) are chosen so that strata whose closures are unrelated
remain separated uniformly over all admissible scales. This separation is what makes it
possible to identify which stratum controls the motion at a given scale and,
later, to assemble the local estimates across different strata.

The dynamical part of the definition has two components. Item~\ref{item:small_step}
requires small-step behavior: both the iterate increments and the variation of
the auxiliary error quantity \(g_{1,k}\) are controlled by a power of the step
size. Item~\ref{item:stratified_descent} is the main descent requirement. It
does not ask for a descent inequality along the original sequence itself.
Instead, whenever a block of iterates remains in the neighborhood
\(\mathcal{N}(i,\alpha_k)\) of a fixed stratum \(M_i\), it controls the length of
the projected sequence \(P_{M_i}(x_k)\) by the drop of the desingularized potential
\(\psi\circ f\), evaluated with the auxiliary correction \(g_{1,k}\), up to the
accumulated error \(g_{2,k}\) and a small step-size remainder. As in the usual
Kurdyka--\L{}ojasiewicz length estimates
\cite{absil2005convergence,attouch2009convergence,li2023convergence}, the
function \(\psi\) converts decrease in the potential into length control. The
additional geometric conditions on the exponents \(\beta_i,\gamma_i\) ensure
that these projected-length estimates can be patched together across strata,
yielding a diameter bound for the original sequence.

The following theorem shows that, for difference inclusions exhibiting stratified descent, the sequence diameter is controlled by the functions appearing in the condition.

\begin{theorem}[Stratified diameter bound]\label{thm:diameter}
Consider a difference inclusion \eqref{eq:DI_with_noise} that exhibits stratified descent in $X\subset \mathbb{R}^n$. Then for any $c_d\geq 1$, there exist $\bar{\alpha},C>0$ such that the following holds. For any realization $(x_k)$ of \eqref{eq:DI_with_noise} whose step sizes satisfy $(\alpha_k)\subset (0,\bar{\alpha}]$ and $\sup_{k\geq i}\alpha_k/\alpha_i\leq c_d$ for every \(i\in\mathbb N\), and every \(K\in\mathbb N\) satisfying $x_{\llbracket 0,K\rrbracket}\subset X$, we have \begin{align}\label{eq:diam_formula}
        \operatorname{diam}(x_{\llbracket 0,K\rrbracket}) \leq 2(\psi\circ f(x_0)-\psi\circ f(x_K)) + 4(\psi(g_{1,0})+ \psi(g_{1,K})) + 2\sum_{k=0}^{K-1}g_{2,k} + C\alpha_0^{\underline\beta(1-\theta)}.
    \end{align}
    where  $\underline{\beta}=\min_{i\in \llbracket 1,\overline{T}\rrbracket}\beta_i$ and $g_{\ell,k}:= g_{\ell}\left((\alpha_i)_{i\geq k},(e_i)_{i\geq k}\right)$ for $\ell\in \{1,2\}$ and $k\in \mathbb{N}$.
\end{theorem}
After shifting the initial index, Theorem~\ref{thm:diameter} bounds every
segment contained in \(X\) by the drop of \(2\psi\circ f\) plus the tail error
terms.  When these terms vanish, this yields the diameter estimate
\eqref{eq:diam_cri} in Theorem~\ref{thm:criterion}, with
\(L:=2\psi\circ f\).  The proof of
Theorem~\ref{thm:diameter}
parallels the diameter estimate for subgradient sequences in
\cite[Theorem~1.1]{lai2025diameter} and is given in
Appendix~\ref{sec:proof_diameter}. Theorem~\ref{thm:diameter} allows the
general stratified descent condition
and nonmonotone step sizes.  This flexibility is used in
Section~\ref{sec:first-order}.


Given a convergent realization of \eqref{eq:DI_with_noise}, our next goal is to identify properties of its limit point. Assuming the cumulative noise $(e_k)_{k\in \mathbb{N}}$ is dominated by the summation of step sizes (see \eqref{eq:noise_step}), Proposition~\ref{prop:converge->critical} shows that $0$ belongs to the convex hull of an outer limit of scaled update maps. We recall the relevant notion of outer limits for sets and set-valued maps \cite[Chapter~4]{rockafellar1998variational}. For a sequence $(C_k)$ of subsets of $\mathbb{R}^n$, its outer limit is
\begin{equation*}
    \limsup_{k\to\infty}C_k:=\{x\in\mathbb R^n:\exists k_j\uparrow\infty,\ \exists x_{k_j}\in C_{k_j}\text{ with }x_{k_j}\to x\}.
\end{equation*}
Let $X\subset \mathbb{R}^m$, $G\colon X\rightrightarrows \mathbb{R}^n$, and $\bar{x}\in\overline X$. Then the outer limit of $G$ at $\bar{x}$ relative to \(X\) is given by
\begin{equation*}
    \limsup_{\substack{x\to \bar{x}\\x\in X}}G(x)
    := \bigcup_{\substack{(x_k)\subset X\\x_k \to \bar{x}}}
    \limsup_{k\to \infty} G(x_k).
\end{equation*}
When \(\bar x\in X\), the map $G$ is called outer semicontinuous at $\bar{x}$ if
$\limsup_{x\to\bar{x},\,x\in X}G(x)\subset G(\bar{x})$; we omit the domain
constraint when it is clear. The scaled map
\((y,\alpha)\mapsto F(y,\alpha)/\alpha\) below has domain
\(\mathbb R^n\times(0,\infty)\). All its outer limits and local boundedness
conditions are relative to this domain; when the limiting step size is zero,
we write \(\alpha\downarrow0\).
\begin{proposition}[Limit stationarity]\label{prop:converge->critical}
Let $(x_k)$ be a realization of the difference inclusion \eqref{eq:DI_with_noise} with $\alpha_k>0$, $\alpha_k \to \underline{\alpha}\in \mathbb{R}_+$, $\sum_{k = 0}^\infty\alpha_k = \infty$, and 
\begin{equation}\label{eq:noise_step}
    \lim_{N\to \infty}\frac{\sum_{k = 0}^N e_k}{\sum_{k = 0}^N\alpha_k}  = 0.
\end{equation}
If $x_k\to x^*$ and the set-valued map $(y,\alpha) \mapsto F(y,\alpha)/\alpha$ is locally bounded near $(x^*,\underline{\alpha})$, then
\[
0\in \co\,\limsup_{\substack{(y,\alpha)\to(x^*,\underline{\alpha})\\\alpha>0}}
\frac{1}{\alpha}\,F(y,\alpha).
\]
\end{proposition}
\begin{proof}
\runinheading{Step~1: Separate the limiting directions}
Let \(G\) denote the outer limit in the conclusion.  Local boundedness of
\((y,\alpha)\mapsto F(y,\alpha)/\alpha\) makes \(G\), and hence its convex
hull \(\co G\), compact.  If \(0\notin\co G\), the Hahn--Banach separation
theorem gives \(u\in\mathbb R^n\) and \(\eta>0\) such that
\(\langle u,z\rangle\geq\eta\) for every \(z\in\co G\).

Pick $w_k\in F(x_k,\alpha_k)$ with $x_{k+1}=x_k+w_k+e_k$ and set $v_k:=w_k/\alpha_k$ for all $k\in \mathbb{N}$.
Local boundedness and $(x_k,\alpha_k)\to(x^*, \underline{\alpha})$ imply that $(v_k)_{k\in \mathbb{N}}$ is bounded. We next prove that there is $K\in \mathbb{N}$ such that $\langle u,v_k\rangle \geq \eta/2$ for all $k\geq K$. If not, there exists a subsequence $k_j\to\infty$ with $\langle u,v_{k_j}\rangle \leq \eta/2$.
Because $(v_k)_{k\in \mathbb{N}}$ is bounded and $(x_{k_j},\alpha_{k_j})\to(x^*,\underline{\alpha})$, we may pass to a further subsequence (not relabeled) with
$v_{k_j}\to v$. By the definition of the outer limit, $v\in G \subset \co G$.
As the linear map $\langle u,\cdot\rangle$ is continuous, $\langle u,v\rangle \leq \eta/2$, which is a contradiction.

\runinheading{Step~2: Average the updates}
For any $N\geq K$, let \(A_N:=\sum_{k=0}^{N-1}\alpha_k\) and
\(\lambda_k^{(N)}:=\alpha_k/A_N\). Then
\(\lambda_k^{(N)}\geq0\) and
\(\sum_{k=0}^{N-1}\lambda_k^{(N)}=1\).
Telescoping gives the averaging identity
\begin{equation}\label{eq:average}
    \frac{x_N-x_0}{A_N}=\sum_{k=0}^{N-1}\lambda_k^{(N)}\,v_k+\frac{\sum_{k=0}^{N-1}e_k}{A_N}. 
\end{equation}
Taking the inner products of $u$ and the average in \eqref{eq:average} yields
\[
\begin{aligned}
\left\langle u,\frac{x_N-x_0}{A_N}\right\rangle
&=
\sum_{k=0}^{N-1}\lambda_k^{(N)}\langle u,v_k\rangle
+\left\langle u,\frac{\sum_{k=0}^{N-1}e_k}{A_N}\right\rangle \\
&\geq
\sum_{k=0}^{K-1}\lambda_k^{(N)}\langle u,v_k\rangle
+\sum_{k=K}^{N-1}\lambda_k^{(N)}\eta/2
+\left\langle u,\frac{\sum_{k=0}^{N-1}e_k}{A_N}\right\rangle .
\end{aligned}
\]
As $N\to\infty$, the left-hand side converges to $0$ since $x_N \to x^*$ and $A_N\to \infty$. For the right-hand side, since  $\sum_{k=0}^{K-1}\lambda_k^{(N)} = \sum_{k = 0}^{K-1}\alpha_k/A_N\to 0$ and $\sum_{k=0}^{N-1}e_k/A_N \to 0$, the first and the third terms both vanish, while the middle term converges to $\eta/2>0$. This is a contradiction. Thus, $0\in \co G$.
\end{proof}
To interpret Proposition~\ref{prop:converge->critical}, consider update maps of the form $F(x,\alpha):=\alpha H(x)$ for some set-valued map $H\colon\mathbb{R}^n \rightrightarrows \mathbb{R}^n$, as in (stochastic) subgradient methods \cite{shor1962application,robbins1951stochastic}, where one seeks a point $x^*$ with $0\in H(x^*)$. If a realization of \eqref{eq:DI_with_noise} converges to $x^*$ and $H$ is outer semicontinuous and locally bounded at $x^*$, then
\begin{equation*}
    0\in \co\,\limsup_{\substack{(y,\alpha)\to(x^*,\underline{\alpha})\\\alpha>0}}\ \frac{1}{\alpha}\,F(y,\alpha) = \co\,\limsup_{y\to x^*} H(y) \subset \co H(x^*).
\end{equation*}
Hence $x^*$ solves the target inclusion whenever $H$ is convex-valued. The next proposition shows that the same conclusion holds for the slightly enlarged update map $F(x,\alpha):=\alpha\,\co H(B(x,\alpha))$, a form that will be useful in Section~\ref{sec:first-order} for the study of inexact subgradient methods.


\begin{proposition}[Enlargement limits]\label{prop:G-in-coHx}
Let $H\colon\mathbb{R}^n \rightrightarrows \mathbb{R}^n$ be an outer semicontinuous set-valued map that is locally bounded at a point $x\in\mathbb{R}^n$. Then
\begin{equation*}
    \limsup_{\substack{y\to x\\\alpha\downarrow0}} \co H(B(y,\alpha))\subset \co H(x).
\end{equation*}
\end{proposition}
\begin{proof}
Denote by \(G(x):= \limsup_{y\to x,\,\alpha\downarrow0} \co H(B(y,\alpha))\).
Thus $v\in G(x)$ iff there exist $y_k\to x$ and positive numbers
$\alpha_k\to0$ together with $u_k\in \co H(B(y_k,\alpha_k))$ such that
$u_k\to v$. We will show that any cluster point of such a sequence $(u_k)$ lies in $\co H(x)$.

Fix $v\in G(x)$ and corresponding sequences $(y_k,\alpha_k)$ and $u_k\in \co H(B(y_k,\alpha_k))$ with $u_k\to v$. By Carath\'eodory's theorem, for each $k$ there exist $z_{k,i}\in B(y_k,\alpha_k)$, $h_{k,i}\in H(z_{k,i})$, and $\lambda_{k,i}\geq0$ for $i=0,\dots,n$, with $\sum_{i=0}^{n}\lambda_{k,i}=1$ such that \(u_k=\sum_{i=0}^{n}\lambda_{k,i}\, h_{k,i}\).

For each fixed \(i\), since \(z_{k,i}\in B(y_k,\alpha_k)\), \(y_k\to x\),
and \(\alpha_k\to0\), we have
\(|z_{k,i}-x|\leq |z_{k,i}-y_k|+|y_k-x|
\leq \alpha_k+|y_k-x|\to0\).
Thus all the points \(z_{k,i}\), \(i=0,\ldots,n\), lie in a fixed
neighborhood of \(x\) for all sufficiently large \(k\). By local
boundedness of \(H\) at \(x\), there exists \(M>0\) such that
\(|h_{k,i}|\leq M\) for all large \(k\) and all \(i=0,\ldots,n\).
Moreover, the coefficient vectors
\((\lambda_{k,0},\ldots,\lambda_{k,n})\) belong to the probability simplex,
which is compact.  After passing to a subsequence, we may therefore assume
\(\lambda_{k,i}\to \lambda_i\) for \(i=0,\ldots,n\).
Since the sequences \((h_{k,i})_k\) are bounded for each fixed \(i\), a
diagonal extraction gives \(h_{k,i}\to h_i\) for \(i=0,\ldots,n\).
 Since $H$ is outer semicontinuous at $x$, we have $h_i\in H(x)$ for each $i$. Taking limits in the convex representation of $u_k$ gives
\[
v\ =\ \lim_{k\to\infty}u_k\ =\ \lim_{k\to\infty}\sum_{i=0}^{n}\lambda_{k,i}h_{k,i}\ =\ \sum_{i=0}^{n}\lambda_i\,h_i.
\]
Since $\lambda_i\geq 0$, $\sum_{i=0}^{n}\lambda_i=1$, and $h_i\in H(x)$ for all $i$, we conclude $v\in \co H(x)$. As $v\in G(x)$ was arbitrary, the inclusion $G(x)\subset \co H(x)$ follows.
\end{proof}

\section{Convergence of first-order methods for definable functions}\label{sec:first-order}
This section applies the framework from Section~\ref{sec:criterion} to inexact,
momentum, and stochastic first-order methods for nonsmooth optimization. We
consider the unconstrained problem \(\inf_{x\in \mathbb{R}^n} f(x)\), where
$f\colon\mathbb{R}^n \to \mathbb{R}$ is locally Lipschitz. Our focus is on
first-order methods that access only its Clarke subdifferential
$\partial f\colon\mathbb{R}^n \rightrightarrows \mathbb{R}^n$
\cite{clarke1975,clarke1990}. Recall that $x^*\in \mathbb{R}^n$ is called a
critical point of $f$ if $0\in \partial f(x^*)$. The goal of this section is
to show that bounded sequences generated by several first-order methods
converge to critical points.




Because first-order methods may behave erratically on merely locally Lipschitz objectives \cite{wang2005subdifferentiability,daniilidis2020pathological}, we assume additionally that $f$ is definable in a polynomially bounded o-minimal structure \cite{van1998tame,coste2000introduction}. We next recall the definition of o-minimal structures. A structure on $(\mathbb{R},+,\cdot)$ is a family $\mathcal{D}=(\mathcal{D}_n)_{n \geq 1}$ where, for each $n\geq1$, $\mathcal{D}_n$ is a Boolean algebra of subsets of $\mathbb{R}^n$ satisfying:
\begin{enumerate}
  \item If $A\in\mathcal{D}_n$, then $\mathbb{R}\times A$ and $A\times\mathbb{R}$ lie in $\mathcal{D}_{n+1}$.
  \item For every polynomial $P\in\mathbb{R}[X_1,\dots,X_n]$, the zero set $\{x\in\mathbb{R}^n:P(x)=0\}$ lies in $\mathcal{D}_n$.
  \item If $A\in\mathcal{D}_n$ and $\pi\colon\mathbb{R}^n\to\mathbb{R}^{n-1}$ is the projection onto the first $n-1$ coordinates, then $\pi(A)\in\mathcal{D}_{n-1}$.
\end{enumerate}
The structure is o-minimal (short for order-minimal) if every set in $\mathcal{D}_1$ is a finite union of points and open intervals.
A subset of $\mathbb{R}^n$ that lies in $\mathcal{D}_n$ is called definable, and a function is definable if its graph is definable. The structure is polynomially bounded if for every definable $f\colon\mathbb{R}\to\mathbb{R}$ there exist $a>0$ and $N\in\mathbb{N}$ such that $|f(x)|<x^N$ for all $x>a$ \cite[p.~510]{van1996geometric}.
Examples of polynomially bounded o-minimal structures include semialgebraic sets \cite{tarski1951decision,seidenberg1954new} and globally subanalytic sets \cite{van1996geometric}. From now on, we fix a polynomially bounded o-minimal structure and say a set/function is definable if it is definable in this structure.

Definable sets and functions enjoy tame features such as the \L{}ojasiewicz inequality \cite{lojasiewicz1958,bierstone1988semianalytic} and its consequences \cite{lojasiewicz1963propriete,van1996geometric}. They also admit stratifications with strong geometric control \cite{parusinski1994lipschitz,kurdyka1994wf,fischer2007minimal,nguyen2016lipschitz}, which are well suited to verifying the stratified descent condition (Definition \ref{def:stratified_descent}).
Many objectives arising in machine learning, including neural-network models
built from tame elementary operations, are definable; this viewpoint is used
in the variational analysis of automatic differentiation and stochastic
methods for learning models \cite{bolte2020conservative,davis2020stochastic}.  Semialgebraic and globally
subanalytic models, in particular, fall within polynomially bounded
o-minimal structures \cite{bierstone1988semianalytic,van1996geometric}.
We use standard asymptotic notation: for positive sequences
\((a_k)_{k\in\mathbb N}\) and \((b_k)_{k\in\mathbb N}\), write
\(a_k=O(b_k)\) if \(a_k\leq Cb_k\) for some \(C>0\) and all
sufficiently large \(k\), \(a_k=o(b_k)\) if \(a_k/b_k\to0\), and
\(a_k=\Theta(b_k)\) if both \(a_k=O(b_k)\) and \(b_k=O(a_k)\).
We use \(o(\cdot)\) analogously when another limiting variable is indicated.

\subsection{Inexact subgradient method}\label{sec:inexact}
We first consider an inexact subgradient method with updates given by the difference inclusion:
\begin{align}\label{eq:inexact}
        x_{k+1}\in x_k-\alpha_k\partial f(x_k;c_b\alpha_k^\xi)+e_k,\quad \forall k\in \mathbb{N},
    \end{align}
where
\(\partial f(x;r):=\co\bigcup_{z\in B(x,r)}\partial f(z)\)
is a Goldstein-type enlargement of the Clarke subdifferential
\cite{goldstein1977optimization}. Here $c_b\geq0$ and $\xi>0$ govern the radius of the
enlargement relative to the step size. The rule \eqref{eq:inexact} generalizes
the basic (or exact) subgradient method (the case $c_b=0$ and $e_k=0$) by
allowing directions drawn from the convex hull of nearby subgradients and by
admitting an additive error. This model covers Goldstein’s conceptual descent
method \cite{goldstein1977optimization} and its approximations such as gradient
sampling \cite{burke2005robust}, with properly chosen search radius. Moreover,
we show below that the momentum and stochastic subgradient methods reduce to
\eqref{eq:inexact} through a shadow transformation \cite{deng2021minibatch} and the windowing argument
of \cite{tadic2015convergence,qiu2024convergence}, respectively.

We next derive a finite-horizon diameter estimate for bounded realizations of
\eqref{eq:inexact} that remain near a level set of \(f\).
The estimate has the same form as the abstract bound in
Theorem~\ref{thm:diameter}, but the error tails are now written explicitly in
terms of the step sizes and the additive perturbations.
\begin{theorem}[Inexact diameter bound]\label{thm:inexact_diameter}
    Let $f\colon\mathbb{R}^n\to \mathbb{R}$ be locally Lipschitz and definable. Fix $\xi,\tau,c_b>0$ and $c_d\geq1$. For any bounded $X\subset \mathbb{R}^n$, there exists $\bar{\theta}\in (0,1)$ such that for any $\theta\in(\bar{\theta},1)$, there exist $\hat{\alpha},\beta,\varepsilon,\varsigma_1,\varsigma_2>0$ such that the following holds. For any realization $(x_k)$ of \eqref{eq:inexact} satisfying $\alpha_k\in(0,\hat{\alpha}]$, $\sup_{i>k} \alpha_{i}/\alpha_k\leq c_d$, and $|e_k|\leq \alpha_k^\tau$ for every \(k\in\mathbb N\), and every $K\in \mathbb{N}$ satisfying $x_{\llbracket 0,K \rrbracket} \subset X\cap [|f|\leq\varepsilon]$, we have
        \begin{align*}
        \operatorname{diam}(x_{\llbracket 0,K \rrbracket}) \leq~&\varsigma_1 \left(\varphi_\theta(f(x_0))-\varphi_\theta(f(x_K))\right)\\
        &+\varsigma_2\left(\alpha_0^\beta+ \sum_{k=0}^{K-1}\alpha_k^{1+\beta} +\left(\sum_{k=0}^{K-1}\alpha_k^{1+\beta}\right)^{1-\theta}+  \sum_{k=0}^{K-1} \alpha_k \left(\sum_{j=k}^{K-1} \alpha_j^{1+\beta}\right)^{\theta}\right)\\
        &+\varsigma_2\left(\sum_{k=0}^{K-1}\alpha_k\left(\sum_{j=k}^{K-1}|e_j|\right)^{\theta}+\sum_{k=0}^{K-1}|e_k|+\left(\sum_{k=0}^{K-1}|e_k|\right)^{1-\theta}\right).
    \end{align*}
\end{theorem}
The proof of Theorem~\ref{thm:inexact_diameter}, given in Appendix~\ref{sec:proof_inexact}, follows the same scheme
as \cite[Corollary~4.2]{lai2025diameter}.  As in
\cite[Section~4.1]{lai2025diameter}, we choose a signed desingularizing
function uniformly across the finitely many strata, and combine the
scale-dependent neighborhoods of \cite[Proposition~3.5]{lai2025diameter}
with a projected-length estimate modeled on
\cite[Lemma~4.1]{lai2025diameter}.  The Goldstein enlargement and the additive
perturbations require, respectively, Verdier-type tangential control and
explicit tail-error terms.  After choosing compatible scale exponents, these
estimates verify stratified descent (Definition~\ref{def:stratified_descent}),
and Theorem~\ref{thm:diameter} gives the stated bound.

Theorem~\ref{thm:criterion} cannot yet be applied directly: the estimate in
Theorem~\ref{thm:inexact_diameter} holds only inside a fixed level band,
whereas convergence of the potential and vanishing spatial increments are not
yet known.  The next proposition shows that the same localized diameter
estimate first traps the sequence in the band and then yields both missing
properties.  In particular, it does not assume convergence of the potential.

\begin{proposition}[Local diameter trapping]
\label{proposition:local_trapping_diameter}
Let \(f\colon\mathbb{R}^n\to\mathbb{R}\) be locally Lipschitz, let
\((x_k)\) be a bounded sequence, and let \(x^*\) be an
accumulation point of \((x_k)\).  Suppose that
\(|f(x_{k+1})-f(x_k)|\to0\).
Suppose moreover that there exist \(\varepsilon>0\) and a continuous,
strictly increasing function \(\chi\colon\mathbb R\to\mathbb R\) such that
\begin{align}\label{eq:local_diameter_trapping_condition}
    \limsup_{\substack{k_1,k_2\to\infty,\ k_1<k_2\\
x_{\llbracket k_1,k_2\rrbracket}\subset [|f-f(x^*)|\leq 2\varepsilon]}}
\operatorname{diam}(x_{\llbracket k_1,k_2\rrbracket})
+\chi(f(x_{k_2}))-\chi(f(x_{k_1}))
\leq 0 .
\end{align}
Then \(x_k\to x^*\).
\end{proposition}

\begin{proof}
Set \(f_*:=f(x^*)\).
The set \(X:=\overline{\{x_k:k\in\mathbb{N}\}}\) is compact.  By local
Lipschitz continuity, the restriction of \(f\) to \(X\) is
\(C_f\)-Lipschitz for some \(C_f>0\).  Since \(x^*\) is an accumulation
point, the sequence visits
\([|f-f_*|<\varepsilon]\) infinitely often.

\runinheading{Step~1: Trap the sequence in the level band}
We prove that the sequence eventually remains in
\([|f-f_*|\leq 2\varepsilon]\).  Suppose not.  Then we may choose \(k\) arbitrarily large
such that \(|f(x_k)-f_*|<\varepsilon\), and such that the sequence leaves
\([|f-f_*|\leq 2\varepsilon]\) later.  Let \(k'+1\) be the first index after
\(k\) such that
\(
    x_{k'+1}\notin [|f-f_*|\leq 2\varepsilon].
\)
Since \(|f(x_{m+1})-f(x_m)|<\varepsilon/2\) for all large \(m\), we have
\(
    3\varepsilon/2
    \leq |f(x_{k'})-f_*|
    \leq 2\varepsilon .
\)
Consequently,
\[
    \operatorname{diam}(x_{\llbracket k,k'\rrbracket})
    \geq |x_{k'}-x_k|
    \geq \frac{|f(x_{k'})-f(x_k)|}{C_f}
    \geq \frac{\varepsilon}{2C_f}.
\]
Applying \eqref{eq:local_diameter_trapping_condition} to the interval
\(\llbracket k,k'\rrbracket\), for \(k\) sufficiently large, yields
\(
    \chi(f(x_{k'}))<\chi(f(x_k)).
\)
Since \(\chi\) is strictly increasing, we get
\(
    f(x_{k'})<f(x_k)<f_*+\varepsilon .
\)
Together with \(|f(x_{k'})-f_*|\geq 3\varepsilon/2\), this gives
\(
    f(x_{k'})\leq f_*-3\varepsilon/2.
\)
The upper exit \(f(x_{k'+1})>f_*+2\varepsilon\) is then impossible, because
\(
    f(x_{k'+1})
    <
    f(x_{k'})+\varepsilon/2
    \leq
    f_*-\varepsilon .
\)
Thus the first exit is through the lower side:
\(
    f(x_{k'+1})<f_*-2\varepsilon .
\)

Since \(x^*\) is an accumulation point, the sequence later returns to
\([|f-f_*|\leq\varepsilon]\).  Let \(k^*>k'+1\) be the first return index.
Before this return, the sequence cannot jump from the lower side to the
upper side, because the value increments are smaller than \(\varepsilon/2\).
Hence \(f(x_m)<f_*-\varepsilon\) for all
\(m\in\llbracket k'+1,k^*-1\rrbracket\).
Let \(r\) be the last index in \(\llbracket k'+1,k^*-1\rrbracket\) such
that \(|f(x_r)-f_*|>2\varepsilon\), and set \(k^\dagger:=r+1\).  Then
\(
    x_{\llbracket k^\dagger,k^*\rrbracket}
    \subset [|f-f_*|\leq 2\varepsilon],
\)
and the small-increment condition gives
\(
    f(x_{k^\dagger})<f_*-3\varepsilon/2.
\)
Since \(x_{k^*}\in[|f-f_*|\leq\varepsilon]\), we have
\(
    f(x_{k^*})-f(x_{k^\dagger})
    \geq \varepsilon/2.
\)
Therefore,
\[
    \operatorname{diam}(x_{\llbracket k^\dagger,k^*\rrbracket})
    +
    \chi(f(x_{k^*}))-\chi(f(x_{k^\dagger})) \geq
    \operatorname{diam}(x_{\llbracket k^\dagger,k^*\rrbracket})
    \geq
    |x_{k^*}-x_{k^\dagger}|
    \geq
    \frac{\varepsilon}{2C_f},
\]
where we used again the monotonicity of \(\chi\).  This contradicts
\eqref{eq:local_diameter_trapping_condition}, because the pair
\(k^\dagger<k^*\) can be chosen arbitrarily far along the sequence.
Thus the sequence is eventually contained in \([|f-f_*|\leq 2\varepsilon]\).

\runinheading{Step~2: Prove potential convergence and small increments}
On this eventual tail, \eqref{eq:local_diameter_trapping_condition} applies to
every late interval.  We show that \(\chi(f(x_k))\) converges.  Let
\(\ell:=\liminf_k\chi(f(x_k))\), which is finite because \((x_k)\) is bounded
and \(\chi\circ f\) is continuous.  For any \(\eta>0\), choose an arbitrarily late
\(k_1\) with \(\chi(f(x_{k_1}))\leq\ell+\eta\).  Since the diameter is
nonnegative, \eqref{eq:local_diameter_trapping_condition} gives
\(\chi(f(x_{k_2}))\leq\ell+2\eta\) for every sufficiently late \(k_2>k_1\).
Thus \(\limsup_k\chi(f(x_k))\leq\ell\), proving convergence.  Applying the same
estimate with \(k_1=k\) and \(k_2=k+1\) now gives
\(|x_{k+1}-x_k|\to0\).

\runinheading{Step~3: Apply the diameter criterion}
The set \(U:=\{x:|f(x)-f_*|<2\varepsilon\}\) is an open neighborhood of
\(x^*\), and \eqref{eq:local_diameter_trapping_condition} implies
\eqref{eq:diam_cri} with \(L=\chi\circ f\) for blocks contained in \(U\).
Hence Theorem~\ref{thm:criterion} gives \(x_k\to x^*\).
\end{proof}

Combining Theorem~\ref{thm:inexact_diameter} with
Proposition~\ref{proposition:local_trapping_diameter} now gives a sufficient
condition for convergence of \eqref{eq:inexact}.  When the step sizes are of
order $1/k$, the error tails in the theorem vanish.
\begin{corollary}[Inexact subgradient convergence]\label{cor: converge}
    Let $f\colon\mathbb{R}^n\to \mathbb{R}$ be locally Lipschitz and definable. If $(x_k)$ is a bounded realization of \eqref{eq:inexact} with $c_b,\xi>0$, step sizes $\alpha_k=\Theta((k+1)^{-1})$, errors satisfying $|e_k|\leq \alpha_k^\tau$ for some $\tau>0$, and
    \[
        \sum_{k=0}^{\infty}\alpha_k\!\left(\sum_{j=k}^\infty |e_j|\right)^{\theta}<\infty
    \]
    for some $\theta\in (0,1)$, then $(x_k)$ converges to a critical point of $f$.
\end{corollary}
\begin{proof}
Let \((x_k)\) be a bounded realization satisfying the stated assumptions,
and let \(X\) be a compact set containing the sequence.  Let \(x^*\) be an
accumulation point and set \(f_*:=f(x^*)\).

\runinheading{Step~1: Establish the local diameter condition}
We first verify the hypotheses of
Proposition~\ref{proposition:local_trapping_diameter}.  Apply
Theorem~\ref{thm:inexact_diameter} to the shifted function \(f-f_*\).  Choose
\(\theta_0\in(\theta,1)\) sufficiently close to \(1\), and larger than the
threshold in Theorem~\ref{thm:inexact_diameter}.  Since
\(\alpha_k=\Theta((k+1)^{-1})\) and
\(\sum_{k=0}^\infty\alpha_k(\sum_{j=k}^\infty |e_j|)^\theta<\infty\), all
error tails in Theorem~\ref{thm:inexact_diameter} vanish on intervals whose
left endpoint tends to infinity.  Let \(\varsigma_1>0\) denote the coefficient
of the signed-power difference supplied by the theorem, and set
\(\chi(s):=\varsigma_1\varphi_{\theta_0}(s-f_*)\).  This function is strictly
increasing.  Taking the localization radius below to be half the one supplied
by the theorem, we obtain some \(\varepsilon>0\) such that
\[
    \limsup_{\substack{k_1,k_2\to\infty,\ k_1<k_2\\
    x_{\llbracket k_1,k_2\rrbracket}\subset[|f-f_*|\leq2\varepsilon]}}
    \operatorname{diam}(x_{\llbracket k_1,k_2\rrbracket})
    +\chi(f(x_{k_2}))-\chi(f(x_{k_1}))\leq 0.
\]

\runinheading{Step~2: Verify small increments and convergence}
It remains to check the small-increment condition.  Since the sequence is
bounded and \(f\) is locally Lipschitz, there exists \(L>0\) such that
\(|\partial f(x_k;c_b\alpha_k^\xi)|\leq L\) for all large \(k\).  Hence
\(|x_{k+1}-x_k|\leq L\alpha_k+|e_k|\to0\),
and therefore \(|f(x_{k+1})-f(x_k)|\to0\).  Proposition
\ref{proposition:local_trapping_diameter} now gives \(x_k\to x^*\).

\runinheading{Step~3: Establish criticality}
Let \(\bar x=x^*\).  Since \(E_0=\sum_k |e_k|<\infty\), and since
\(\sum_k\alpha_k=\infty\), we have
\(\frac{\sum_{k=0}^{N}e_k}{\sum_{k=0}^{N}\alpha_k}\to0\).
Applying Proposition~\ref{prop:converge->critical} to
\(F(x,\alpha):=-\alpha\,\partial f(x;c_b\alpha^\xi)\), and then using
Proposition~\ref{prop:G-in-coHx} with \(r=c_b\alpha^\xi\downarrow0\) and
\(H=-\partial f\), gives
\(
    0\in \partial f(\bar x).
\)
Thus the sequence converges to a critical point of \(f\).
\end{proof}
\begin{remark}[Discrete argument]\label{remark:discrete}
An alternative proof of Corollary~\ref{cor: converge} can be obtained by
combining Theorem~\ref{thm:inexact_diameter} with existing results on function
value convergence and subsequential convergence for stochastic subgradient
methods \cite{davis2020stochastic,bolte2022long}, which are proved through the
analysis of continuous-time limits \cite{benaim2005stochastic}. Indeed, one may
verify \cite[Assumptions~A \& B]{davis2020stochastic} for the inexact
subgradient method \eqref{eq:inexact}; see
Corollary~\ref{lemma:perturbed_subgradient_value_convergence}. The proof above
does not rely on the differential inclusion, or ODE, method
\cite{ljung1977analysis,kushner1977general}. Thus it gives an entirely discrete
argument in the present setting. We view this perspective as complementary to
the ODE method, and it may lead to finer analyses. For example, a convergence
rate for the subgradient method with constant step size was recently established
in \cite{chzhen2026convergence}.
\end{remark}

\subsection{Momentum method}\label{sec:mom}
Having established convergence for the inexact subgradient method, we now
turn to the subgradient method with momentum.  Fix \(\kappa\in[0,1)\) and
\(\iota\in\mathbb{R}\), set \(x_{-1}=x_0\), and compute
\begin{equation}\label{eq:momentum-inexact}
\begin{aligned}
    \vartheta_k
    &:=x_k+\iota(x_k-x_{k-1}),
    \qquad v_k\in\partial f(\vartheta_k),\\
    x_{k+1}
    &:=x_k+\kappa(x_k-x_{k-1})-\alpha_kv_k,
    \qquad \forall k\in\mathbb{N},
\end{aligned}
\end{equation}
where \(\alpha_k>0\).  When \(\kappa=\iota=0\), this is the vanilla
subgradient method.  The choice \(\iota=0\)
reduces to Polyak's heavy-ball method
\cite{polyak1964some,zavriev1993heavy}, whereas \(\iota=\kappa\) gives a
fixed-parameter Nesterov-type look-ahead update
\cite{nesterov1983method}.
\begin{remark}[Momentum regimes]
We keep \(\kappa\) and \(\iota\) fixed while allowing the step sizes
\((\alpha_k)\) to diminish.  For nonsmooth, nonconvex stochastic
objectives, existing momentum analyses establish stationarity guarantees,
criticality of accumulation points, or convergence of objective values,
but not convergence of the full sequence to a single critical
point \cite{mai2020convergence,le2024nonsmooth}.  Earlier full-sequence
results for composite nonsmooth objectives are available for
constant-inertia proximal methods such as iPiano, where the nonsmooth term
is handled by a proximal step \cite{ochs2014ipiano,Ochs2016}.  For the
direct recurrence \eqref{eq:momentum-inexact}, the corresponding results
discussed here require smoothness: \cite{josz2023convergence} treat the deterministic
case, while \cite{qiu2024convergence} treat the stochastic setting.

Time-dependent momentum parameters arise in both practice and theory.
Momentum may be increased to a prescribed cap in practice or allowed to
approach one in accelerated and stochastic heavy-ball methods
\cite{sutskever2013importance,nesterov1983method,beck2009fast,
sebbouh2021almost}.  Our reduction instead relies on the fixed bound
\(\kappa<1\), which gives uniform geometric decay of past subgradients;
the regime \(\kappa_k\to1\) is left for future work.
\end{remark}

We obtain the following convergence result for
step sizes of order \(k^{-1}\).
\begin{theorem}[Momentum convergence]\label{theorem:momentum-converge}
Let \(f\colon\mathbb{R}^n\to\mathbb{R}\) be locally Lipschitz and definable.
For any fixed \(\kappa\in[0,1)\) and \(\iota\in\mathbb R\), every bounded
realization \((x_k)\) of \eqref{eq:momentum-inexact} with
\(\alpha_k=\Theta((k+1)^{-1})\) converges to a critical point of \(f\).
\end{theorem}

The proof introduces a nearby sequence whose recurrence has the inexact form
\eqref{eq:inexact}.  We record this relation for polynomial step sizes for
later use.

\begin{lemma}[Momentum shadow]
\label{lemma:momentum_shadow}
Let \(f\colon\mathbb R^n\to\mathbb R\) be locally Lipschitz, fix
\(\kappa\in[0,1)\) and \(\iota\in\mathbb R\), let
\(\alpha_k=\Theta((k+1)^{-a})\) for some \(a>0\), and suppose that a sequence
generated by \eqref{eq:momentum-inexact}, with \(x_{-1}=x_0\), is bounded.
Define
\[
 c_\kappa:=\frac{\kappa}{1-\kappa},\qquad
 \widehat x_k:=x_k+c_\kappa(x_k-x_{k-1}),\qquad
 \widehat\alpha_k:=\frac{\alpha_k}{1-\kappa}.
\]
Then \(\widehat\alpha_k=\Theta((k+1)^{-a})\),
\((\widehat x_k)\) is bounded,
\(|\widehat x_k-x_k|=O(\widehat\alpha_k)\), and there is a constant
\(c_b>0\) such that
\begin{equation}
 \widehat x_{k+1}=\widehat x_k-\widehat\alpha_kv_k,
 \qquad
 v_k\in\partial f(\widehat x_k;c_b\widehat\alpha_k).
\label{eq:momentum_matched_shadow}
\end{equation}
In particular, \((\widehat x_k)\) and \((x_k)\) have the same accumulation
points.
\end{lemma}

\begin{proof}
The proof has three parts.  We first control the inertial displacement at the
scale of the current step size, then use the shadow transformation to cancel
the inertial term, and finally compare the two sequences.

\runinheading{Step~1: Bound the inertial displacement}
Set \(\delta_k:=x_k-x_{k-1}\).  Since \((x_k)\) is bounded, so is
\((\vartheta_k)\).  The closure of \(\{\vartheta_k:k\in\mathbb N\}\) is
compact, and local boundedness of the Clarke subdifferential therefore gives a
constant \(L>0\) such that \(|v_k|\leq L\) for every \(k\)
\cite[Propositions~2.1.2 and~2.1.5]{clarke1990}.

Since \(x_{-1}=x_0\), one has \(\delta_0=0\), and
\eqref{eq:momentum-inexact} gives
\[
 \delta_{k+1}=\kappa\delta_k-\alpha_kv_k.
\]
Unrolling this recursion shows that, for every \(k\geq1\),
\begin{equation}
 \delta_k
 =-\sum_{j=0}^{k-1}\kappa^{k-1-j}\alpha_jv_j
 =-\sum_{r=0}^{k-1}\kappa^r\alpha_{k-1-r}v_{k-1-r}.
 \label{eq:momentum_velocity_convolution}
\end{equation}
Here \(\kappa^0:=1\), including when \(\kappa=0\).
Choose constants \(0<c_-\leq c_+\) such that
\[
 c_-(j+1)^{-a}\leq\alpha_j\leq c_+(j+1)^{-a}
 \qquad\text{for every }j\in\mathbb N,
\]
and set \(m_k:=\lfloor(k-1)/2\rfloor\).  If \(0\leq r\leq m_k\), then
\(k-r\geq(k+1)/2\), and hence
\(\alpha_{k-1-r}\leq2^ac_+(k+1)^{-a}\).  On the remaining indices we use
only \(\alpha_{k-1-r}\leq c_+\).  Splitting
\eqref{eq:momentum_velocity_convolution} at \(m_k\) gives
\begin{align}
 |\delta_k|
 &\leq L\sum_{r=0}^{m_k}\kappa^r\alpha_{k-1-r}
      +L\sum_{r=m_k+1}^{k-1}\kappa^r\alpha_{k-1-r} \notag\\
 &\leq \frac{Lc_+}{1-\kappa}
      \left(2^a(k+1)^{-a}+\kappa^{m_k+1}\right).
 \label{eq:momentum_split_convolution}
\end{align}
Here \(m_k+1=\lceil k/2\rceil\).  If \(\kappa=0\), the second term in
parentheses vanishes.  If \(\kappa\in(0,1)\), set
\(\lambda:=-\tfrac12\log\kappa>0\).  Then
\(\kappa^{m_k+1}\leq e^{-\lambda k}\), and
\(\sup_{t\geq1}(t+1)^ae^{-\lambda t}<\infty\).  Thus in either case
\(\kappa^{m_k+1}=O((k+1)^{-a})\).  Using also the lower bound on
\(\alpha_k\) and the identity \(\alpha_k=(1-\kappa)\widehat\alpha_k\), we
obtain a constant \(C_\delta>0\) such that
\begin{equation}
 |\delta_k|\leq C_\delta\widehat\alpha_k
 \qquad\text{for every }k\in\mathbb N.
 \label{eq:momentum_velocity_step_scale}
\end{equation}
The case \(k=0\) follows from \(\delta_0=0\).

\runinheading{Step~2: Derive the shadow recursion}
The auxiliary point is the constant-momentum transform used in
\cite[Eq.~(17)]{deng2021minibatch}.  Since
\(1+c_\kappa=(1-\kappa)^{-1}\) and
\((1+c_\kappa)\kappa=c_\kappa\), substitution in the momentum update gives
\[
 \widehat x_{k+1}=x_k+c_\kappa\delta_k-\widehat\alpha_kv_k
 =\widehat x_k-\widehat\alpha_kv_k.
\]
Moreover, \eqref{eq:momentum_velocity_step_scale} gives
\[
 |\vartheta_k-\widehat x_k|
 =|\iota-c_\kappa|\,|\delta_k|
 \leq |\iota-c_\kappa|C_\delta\widehat\alpha_k.
\]
Set \(c_b:=\max\{1,|\iota-c_\kappa|C_\delta\}\).  Then
\(\vartheta_k\in B(\widehat x_k,c_b\widehat\alpha_k)\), and the definition
of the Goldstein enlargement yields
\[
 v_k\in\partial f(\vartheta_k)
 \subset\bigcup_{z\in B(\widehat x_k,c_b\widehat\alpha_k)}\partial f(z)
 \subset\partial f(\widehat x_k;c_b\widehat\alpha_k).
\]
This proves \eqref{eq:momentum_matched_shadow}.

\runinheading{Step~3: Compare the two sequences}
Multiplication by the positive constant \((1-\kappa)^{-1}\) gives
\(\widehat\alpha_k=\Theta((k+1)^{-a})\), and hence
\(\widehat\alpha_k\to0\).  By
\eqref{eq:momentum_velocity_step_scale},
\[
 |\widehat x_k-x_k|
 =c_\kappa|\delta_k|
 \leq c_\kappa C_\delta\widehat\alpha_k\longrightarrow0.
\]
It follows that \((\widehat x_k)\) is bounded.  If
\(x_{k_j}\to\bar x\), then \(\widehat x_{k_j}\to\bar x\); conversely, if
\(\widehat x_{k_j}\to\bar x\), then \(x_{k_j}\to\bar x\).  Therefore the
two sequences have exactly the same accumulation points.
\end{proof}
With the shadow recurrence established, we can now invoke convergence of the
inexact subgradient method.

\begin{proof}[Proof of Theorem~\ref{theorem:momentum-converge}]
Lemma~\ref{lemma:momentum_shadow} shows that the bounded shadow sequence is a
realization of \eqref{eq:inexact} with step sizes
\(\widehat\alpha_k=\Theta((k+1)^{-1})\), enlargement exponent \(\xi=1\), and
zero additive error.  Corollary~\ref{cor: converge} therefore gives
\(\widehat x_k\to x^*\) for some critical point \(x^*\) of \(f\).  Since
\(|\widehat x_k-x_k|\to0\), we also have \(x_k\to x^*\).
\end{proof}

\subsection{Stochastic subgradient method}\label{sec:stochastic}

We now consider the stochastic subgradient method \cite{robbins1951stochastic}
\begin{equation}\label{eq:inexact stoc}
    x_{k+1}=x_k-\alpha_k(v_k+\varepsilon_k),
    \qquad
    v_k\in\partial f(x_k),
    \qquad \forall k\in\mathbb{N}.
\end{equation}
Here \(\alpha_k>0\).  We use the following stochastic setup.

\begin{assumption}[Stochastic noise]
\label{assumption:stochastic_noise}
The \(\mathbb R^n\)-valued stochastic process
\((x_k,v_k,\varepsilon_k)_{k\in\mathbb N}\) is defined on a filtered
probability space
\((\Omega,\mathcal F,(\mathcal F_k)_{k\in\mathbb N},\mathbb P)\).  For every
\(k\in\mathbb N\), the random variables \(x_k\) and \(v_k\) are
\(\mathcal F_k\)-measurable, whereas \(\varepsilon_k\) is
\(\mathcal F_{k+1}\)-measurable.  Moreover, there exists a deterministic
constant \(\sigma>0\) such that
\begin{equation}\label{eq:stochastic_noise_moments}
    \mathbb{E}[\varepsilon_k\mid\mathcal{F}_k]=0,
    \qquad
    \mathbb{E}[|\varepsilon_k|^2\mid\mathcal{F}_k]\leq \sigma^2
    \quad\text{a.s. for every }k\in\mathbb{N}.
\end{equation}
\end{assumption}

\begin{remark}[Noise models]\label{remark:sto_assump}
The martingale-difference and uniform conditional second-moment conditions in
Assumption~\ref{assumption:stochastic_noise} are standard in
stochastic approximation and stochastic optimization; see, e.g.,
\cite{robbins1951stochastic,borkar2008stochastic,davis2020stochastic,qiu2024convergence}.
More general analyses replace the uniform bound with state-dependent growth
conditions, expected smoothness, or expected residual; see
\cite{BottouCurtisNocedal2018,GowerLoizouQianEtAl2019,GowerSebbouhLoizou2021}.
Although globally weaker than a uniform variance bound, these conditions
reduce to the second-moment bound in \eqref{eq:stochastic_noise_moments} after
localization to bounded iterates.  We impose this bound directly to keep the
presentation simple.  Heavy-tailed noise is an important exception: in
practice, stochastic gradients may fail to have bounded variance \cite{simsekli2019tail,garg2021proximal}, motivating
genuinely weaker moment assumptions.  The models studied in
\cite[Assumption~1]{zhang2020adaptive} and
\cite[Assumption~3]{hubler2024gradient} assume only a bounded \(p\)th central
moment for some \(p\in(1,2]\).  When \(p<2\), these models permit noise
with an infinite second moment and hence are not covered by
Assumption~\ref{assumption:stochastic_noise}, even when the iterates
remain bounded.  Extending the present analysis to such heavy-tailed noise
models is left for future work.
\end{remark}

We obtain convergence for step sizes of order \(1/k\).
\begin{theorem}[Stochastic convergence]\label{theorem:1overk}
Let \(f\colon\mathbb{R}^n\to\mathbb{R}\) be locally Lipschitz and definable.
Consider \eqref{eq:inexact stoc} under
Assumption~\ref{assumption:stochastic_noise}, with deterministic step
sizes \(\alpha_k=\Theta((k+1)^{-1})\).  Then, almost surely, every bounded
realization \((x_k)\) converges to a critical point of \(f\).
\end{theorem}
The proof of Theorem~\ref{theorem:1overk} is given at the end of the section. It relies on time-window constructions
similar to those of \cite{tadic2015convergence,qiu2024convergence}.  Fix
\(\zeta\in(0,1)\), and define the window indices by
\begin{equation}\label{eq:def_st}
    s_0:=0,
    \qquad
    s_{t+1}:=
    \inf\left\{k>s_t:\ \sum_{i=s_t}^{k}\alpha_i>\alpha_{s_t}^{\zeta}\right\}
    \qquad \forall t\in\mathbb N.
\end{equation}
Set
\begin{equation}\label{eq:def_at}
        a_t:=\sum_{k=s_t}^{s_{t+1}-1}\alpha_k,
    \qquad
    u_t:=\frac1{a_t}\sum_{k=s_t}^{s_{t+1}-1}\alpha_kv_k,
    \qquad
    e_t:=-\sum_{k=s_t}^{s_{t+1}-1}\alpha_k\varepsilon_k .
\end{equation}
Then \(x_{s_{t+1}}=x_{s_t}-a_tu_t+e_t\).

The construction is a shrinking-window variant of the fixed-mass time scales
in \cite[Section~2]{tadic2015convergence},
\cite[Definition~3.1 and Lemma~3.2(b)]{milzarek2023convergence}, and
\cite[Lemma~3.1]{qiu2024convergence}.  Since those results do not cover the
shrinking threshold \(\alpha_{s_t}^{\zeta}\), we record the required endpoint
estimates.
\begin{fact}[Window scales]\label{fact:st}
Assume \(\alpha_k=\Theta((k+1)^{-a})\) with \(a\in(0,1]\), and let
\(\zeta\in(0,1)\). Define
\(\lambda:=1-a(1-\zeta)>0\) and \(b:=a\zeta/\lambda\). Then the
windows defined by \eqref{eq:def_st} satisfy
\(s_t=\Theta((t+1)^{1/\lambda})\) and
\(a_t=\Theta(\alpha_{s_t}^{\zeta})=\Theta((t+1)^{-b})\).
In particular, when \(a=1\), one has \(a_t=\Theta((t+1)^{-1})\).
\end{fact}

The proof of Fact~\ref{fact:st} is given in
Appendix~\ref{sec:proof_fact_st}. We now show that the
windowed stochastic sequence eventually satisfies the inexact subgradient
update \eqref{eq:inexact}.
\begin{lemma}[Stochastic window reduction]
\label{lemma:stoc_error}
Suppose that Assumption~\ref{assumption:stochastic_noise} holds and
that the step sizes are deterministic with
\(\alpha_k=\Theta((k+1)^{-a})\) for some
\(a\in(1/2,1]\), and fix \(0<\zeta<\gamma<1-1/(2a)\).  Then, almost surely,
if the realization \((x_k)\) is bounded, there exists \(c_b>0\) such that, for all sufficiently large \(t\),
\[
    x_{s_{t+1}}=x_{s_t}-a_tu_t+e_t,
    \qquad
    u_t\in\partial f(x_{s_t};c_ba_t),
\]
and
\[
    |e_t|
    \leq
    \max_{s_t\leq \ell\leq s_{t+1}}
    \left|
        \sum_{k=s_t}^{\ell-1}\alpha_k\varepsilon_k
    \right|
    \leq \alpha_{s_t}^{\gamma}
    =O\bigl(a_t^{\gamma/\zeta}\bigr).
\]
\end{lemma}

\begin{proof}
\runinheading{Step~1: Maximal noise on the windows}
Set \(\beta_k:=(k+1)^{a\gamma}\).  This sequence is nondecreasing, and
\[
 \sum_{k=0}^{\infty}\alpha_k^2\beta_k^2\sigma^2<\infty
 \qquad\text{because}\qquad 2a(1-\gamma)>1.
\]
The weighted conclusion of \cite[Lemma~2.2]{milzarek2023convergence},
applied to the deterministic endpoints \(s_t\), therefore gives, almost
surely,
\[
 \sum_{t=0}^{\infty}\beta_{s_t}^2
 \max_{s_t\leq\ell\leq s_{t+1}}
 \left|\sum_{k=s_t}^{\ell-1}\alpha_k\varepsilon_k\right|^2<\infty.
\]
Thus the maximum in this display is
\(o(\beta_{s_t}^{-1})=o(\alpha_{s_t}^{\gamma})\), and hence is at most
\(\alpha_{s_t}^{\gamma}\) for all sufficiently large \(t\).  Let
\(\Omega_0\) be a probability-one event on which this holds.  Since
\(e_t=-\sum_{k=s_t}^{s_{t+1}-1}\alpha_k\varepsilon_k\), it also bounds
\(|e_t|\), while Fact~\ref{fact:st} gives
\(\alpha_{s_t}^{\gamma}=O(a_t^{\gamma/\zeta})\).

\runinheading{Step~2: Goldstein enlargement on a bounded realization}
Fix \(\omega\in\Omega_0\) for which \((x_k(\omega))\) is bounded, and suppress
\(\omega\) from the notation.  The set
\(K:=\overline{\{x_k:k\in\mathbb{N}\}}\) is compact.
Thus, there exists a constant \(L=L(\omega)>0\) such that
\(|v_k|\leq L\) for every \(k\).  For all sufficiently large \(t\) and every
\(s_t\leq k<s_{t+1}\), the recursion and the maximal-noise estimate give
\[
\begin{aligned}
    |x_k-x_{s_t}|
    &\leq
    \sum_{i=s_t}^{k-1}\alpha_i|v_i|
    +
    \left|
        \sum_{i=s_t}^{k-1}\alpha_i\varepsilon_i
    \right| \leq
    La_t+\alpha_{s_t}^{\gamma}.
\end{aligned}
\]
By Fact~\ref{fact:st} and \(\gamma>\zeta\),
\[
    \frac{\alpha_{s_t}^{\gamma}}{a_t}
    =
    O(\alpha_{s_t}^{\gamma-\zeta})
    \longrightarrow0.
\]
After increasing the path-dependent starting index, we therefore have
\(|x_k-x_{s_t}|\leq(L+1)a_t\) throughout every remaining window.  Set
\(c_b:=L+1\).  Then every \(v_k\) entering \(u_t\) belongs to
\(\bigcup_{z\in B(x_{s_t},c_ba_t)}\partial f(z)\).  Since the coefficients
\(\alpha_k/a_t\) are nonnegative and sum to one,
\[
    u_t
    =
    \sum_{k=s_t}^{s_{t+1}-1}\frac{\alpha_k}{a_t}v_k
    \in
    \co\bigcup_{z\in B(x_{s_t},c_ba_t)}\partial f(z)
    =
    \partial f(x_{s_t};c_ba_t).
\]
Summing \eqref{eq:inexact stoc} over the window and using
\eqref{eq:def_at} gives
\(x_{s_{t+1}}=x_{s_t}-a_tu_t+e_t\), as claimed.
\end{proof}

We now prove Theorem~\ref{theorem:1overk} by combining
Lemma~\ref{lemma:stoc_error} with Corollary~\ref{cor: converge}.
\begin{proof}[Proof of Theorem~\ref{theorem:1overk}]
Fix \(0<\zeta<\gamma<1/2\), and work on the probability-one event supplied
by Lemma~\ref{lemma:stoc_error}.  Let \((x_k)\) be a bounded realization in
this event.

\runinheading{Step~1: Convergence of the windowed sequence}
By Lemma~\ref{lemma:stoc_error}, after discarding finitely many windows,
\(x_{s_{t+1}}=x_{s_t}-a_tu_t+e_t\) and
\(u_t\in\partial f(x_{s_t};c_ba_t)\), where, by Fact~\ref{fact:st},
\(a_t=\Theta((t+1)^{-1})\).  Set \(\nu:=\gamma/\zeta>1\).  The same lemma
gives \(|e_t|=O(a_t^\nu)=O((t+1)^{-\nu})\).  Choose
\(\tau_e\in(1,\nu)\); since \(a_t\to0\), after discarding finitely many more
windows we have \(|e_t|\leq a_t^{\tau_e}\).  Moreover, for every
\(\theta\in(0,1)\), the integral estimate for a \(p\)-series gives
\[
    \sum_{j=t}^{\infty}|e_j|
    =
    O((t+1)^{1-\nu}),
    \qquad
    a_t\left(\sum_{j=t}^{\infty}|e_j|\right)^\theta
    =
    O((t+1)^{-1-\theta(\nu-1)}).
\]
The latter bound is summable in \(t\).  The windowed sequence is bounded as a
subsequence of \((x_k)\).  Hence Corollary~\ref{cor: converge} applies to its
tail, with step sizes \(a_t\), enlargement exponent
\(\xi=1\), and error exponent \(\tau_e\).  It follows that
\(x_{s_t}\to x^*\) for some critical point \(x^*\) of \(f\).

\runinheading{Step~2: Pass to the full sequence}
The closure of the bounded realization is compact.  By local boundedness of
the Clarke subdifferential, there exists \(L>0\) such that
\(|v_k|\leq L\) for every \(k\).  For
\(s_t\leq r<q\leq s_{t+1}\), we have
\(x_q-x_r=-\sum_{k=r}^{q-1}\alpha_kv_k
-\sum_{k=r}^{q-1}\alpha_k\varepsilon_k\).
For all sufficiently large \(t\), the maximal-noise estimate in
Lemma~\ref{lemma:stoc_error} implies
\[
    \left|\sum_{k=r}^{q-1}\alpha_k\varepsilon_k\right|
    \leq
    \left|\sum_{k=s_t}^{q-1}\alpha_k\varepsilon_k\right|
    +
    \left|\sum_{k=s_t}^{r-1}\alpha_k\varepsilon_k\right| \leq 2\alpha_{s_t}^{\gamma}.
\]
Consequently,
\[
    \operatorname{diam}
    (x_{\llbracket s_t,s_{t+1}\rrbracket})
    \leq La_t+2\alpha_{s_t}^{\gamma}
    \longrightarrow0.
\]
For \(s_t\leq k\leq s_{t+1}\), it follows that
\[
    |x_k-x^*|
    \leq
    \operatorname{diam}
    (x_{\llbracket s_t,s_{t+1}\rrbracket})
    +|x_{s_t}-x^*|
    \longrightarrow0.
\]
Thus the full sequence converges to \(x^*\).
\end{proof}

\section{Convergence of first-order methods under local active geometry}
\label{sec:active}

Section~\ref{sec:first-order} establishes convergence of first-order methods
for locally Lipschitz definable objectives with step sizes
\(\alpha_k=\Theta((k+1)^{-1})\).  This choice has three distinct properties:
\[
 \alpha_k\to0,\qquad
 \sum_{k=0}^\infty\alpha_k=\infty,\qquad
 \sum_{k=0}^\infty\alpha_k^2<\infty.
\]
The first says that the individual step sizes vanish, while the second says
that their cumulative sum is infinite.
The role of square summability is less transparent, although it appears as a
standard assumption in several convergence analyses of subgradient methods;
see, for example, \cite{alber1998projected,davis2020stochastic}.  This leads to
two questions: is square summability necessary for convergence, and under what
additional conditions is it sufficient?

We investigate these questions within the polynomial family
\(\alpha_k=\Theta((k+1)^{-a})\).  In the nontrivial nonsummable range
\(0<a\leq1\), the step sizes vanish for every \(a>0\), whereas
\(\sum_k\alpha_k^2<\infty\) holds exactly when \(a>1/2\).  At the boundary
\(a=1/2\), Section~\ref{sec:square_summability_obstruction} constructs a
semialgebraic objective admitting a bounded, nonconvergent sequence of the
subgradient method.  Remarkably, the accumulation points of this
counterexample satisfy the local active-manifold hypotheses introduced in
Section~\ref{sec:active_geometry}.  That section derives local estimates for
an inexact subgradient step: it controls the distance of the new iterate to
the active manifold and shows that the projected step behaves
like an inexact Riemannian gradient step.  Using these estimates, we widen the
admissible ranges to \(1/2<a\leq1\) for the momentum method and
\(2/3<a\leq1\) for the stochastic subgradient method in
Sections~\ref{sec:active_momentum} and~\ref{sec:active_stochastic},
respectively.

\subsection{The square-summability obstruction}
\label{sec:square_summability_obstruction}

For each \(K\geq1\), the step sizes \(\alpha_k=(k+K)^{-1/2}\) vanish and
are nonsummable, whereas their squares are not summable.  We next construct a locally Lipschitz semialgebraic objective for which the subgradient
method admits a bounded, nonconvergent sequence under every such step-size
sequence.

Write points of \(\mathbb R^6\) as
\((p,z)\in\mathbb R^2\times\mathbb R^4\), where \(p=(p_1,p_2)\) and
\(z=(z_1,\ldots,z_4)\).  Define \(q=(q_1,q_2)\colon\mathbb R^2\to\mathbb R^2\)
by
\begin{equation}
 q(p):=
 \begin{cases}
  2p, & |p|\leq1/2,\\
  p/|p|, & |p|>1/2,
 \end{cases}.
\label{eq:counterexample_q_definition}
\end{equation}
Set \(a_1=2+q_1\), \(a_2=2-q_1\), \(a_3=2+q_2\), and
\(a_4=2-q_2\), and define
\begin{equation}
 f(p,z):=\sum_{i=1}^4a_i(p)|z_i|.
\label{eq:counterexample_objective_definition}
\end{equation}

\begin{theorem}[Boundary counterexample]
\label{thm:counterexample-half}
    The function \(f\) in \eqref{eq:counterexample_objective_definition} is
    locally Lipschitz and semialgebraic.  For every real \(K\geq1\), there
    exist a bounded, nonconvergent realization \((x_k)\) of
    \(x_{k+1}\in x_k-\alpha_k\partial f(x_k)\), with
    \(\alpha_k=(k+K)^{-1/2}\), and a radius \(r_\infty\geq1\) such that
    \(f(x_k)\to0\) and the set of accumulation points is
    \(\{(p,0)\in\mathbb R^2\times\mathbb R^4:|p|=r_\infty\}\).
\end{theorem}

\begin{proof}
\runinheading{Step~1: Verify the regularity of the objective}
We first prove the asserted regularity and then derive the subdifferential
formula used below to choose the even- and odd-indexed subgradients.

\eqref{eq:counterexample_q_definition} is exactly
\(q(p)=P_{B(0,1)}(2p)\).  Since metric projection onto a closed convex set
is nonexpansive \cite[Proposition~1.1.9]{bertsekas2009convex}, \(q\) is
globally \(2\)-Lipschitz.  Its piecewise formula
also shows that it is semialgebraic.
Consequently, each $a_i$ is globally Lipschitz and semialgebraic, with
$1\leq a_i\leq3$.

To verify local Lipschitz continuity of \(f\), fix
\((\bar p,\bar z)\in\mathbb R^6\).  If \((p,z)\) and \((p',z')\) lie within
unit distance of \((\bar p,\bar z)\), then \(|z_i'|\leq|\bar z|+1\), and
\[
 \bigl|a_i(p)|z_i|-a_i(p')|z_i'|\bigr|
 \leq3|z_i-z_i'|
 +2(|\bar z|+1)|p-p'|.
\]
Indeed, add and subtract $a_i(p)|z_i'|$, use $|a_i(p)|\leq3$, and then
use the $2$-Lipschitz bound for $a_i$.  Summing over $i$ and using
\(\sum_{i=1}^4|z_i-z_i'|\leq2|z-z'|\) shows that \(f\) is Lipschitz on
this neighborhood.  Since \((\bar p,\bar z)\) was arbitrary, \(f\) is locally
Lipschitz.  Its semialgebraicity follows from the closure of semialgebraic
functions under absolute values, products, and finite sums
\cite[Section~2.2]{bochnak2013real}.

On the region \(|p|>1/2\), every $a_j$ is positive and $C^1$.  Set
$f_j(p,z):=a_j(p)|z_j|$.  If $z_j\neq0$, then $f_j$ is $C^1$.  If
$z_j=0$, its directional derivative in the direction $(u,w)$ is
$a_j(p)|w_j|$.  To compute the Clarke directional derivative, let
$(p',z')\to(p,z)$, $t\downarrow0$, and $(u',w')\to(u,w)$.  Decompose the
corresponding difference quotient as
\[
\begin{aligned}
 \frac{a_j(p'+tu')|z_j'+tw_j'|-a_j(p')|z_j'|}{t}
 &=\frac{a_j(p'+tu')-a_j(p')}{t}|z_j'+tw_j'|\\
 &\quad+a_j(p')\frac{|z_j'+tw_j'|-|z_j'|}{t}.
\end{aligned}
\]
The first term tends to zero: the first factor is bounded by the Lipschitz
constant of $a_j$ times $|u'|$, while $|z_j'+tw_j'|\to0$.  The limsup of
the second term is at most $a_j(p)|w_j|$ by the reverse triangle
inequality, and equality is obtained by taking $z'=0$.  Thus the Clarke
directional derivative is also $a_j(p)|w_j|$, so in both cases $f_j$ is
Clarke regular in the sense of
\cite[Definition~2.3.4, p.~39]{clarke1990}.  For
\(j\in\{1,2,3,4\}\), write \(\mathbf e_j\) for the
\(j\)th standard basis vector of \(\mathbb R^4\).  Then
\[
 \partial f_j(p,z)=
 \left\{
 \left(|z_j|\nabla a_j(p),a_j(p)\sigma_j\mathbf e_j\right):
 \sigma_j\in\partial|\cdot|(z_j)
 \right\}.
\]
The equality case of Clarke's finite-sum rule
\cite[Corollary~3 to Proposition~2.3.3]{clarke1990} therefore gives, on
\(|p|>1/2\),
\begin{equation}
\partial f(p,z)=
\left\{
\left(
 \sum_{j=1}^4|z_j|\nabla a_j(p),
 (a_j(p)\sigma_j)_{j=1}^4
\right):
\sigma_j\in\partial|\cdot|(z_j),\ 1\leq j\leq4
\right\}.
\label{eq:counterexample_subdifferential}
\end{equation}

\runinheading{Step~2: Construct the two-step cycles}
Fix an arbitrary real \(K\geq1\), and set
\(\alpha_k:=(k+K)^{-1/2}\).  On \(|p|>1/2\), each \(a_i(p)\) depends only
on the direction of \(p\).  Define
\(\bar a_i(\theta):=a_i(\cos\theta,\sin\theta)\), and let
\(\bar a_i'(\theta)\) denote its derivative.  The four profiles are
\[
\begin{array}{llll}
    \bar a_1(\theta)=2+\cos\theta, & \bar a_1'(\theta)=-\sin\theta,
    & \bar a_2(\theta)=2-\cos\theta, & \bar a_2'(\theta)=\sin\theta,\\[2pt]
    \bar a_3(\theta)=2+\sin\theta, & \bar a_3'(\theta)=\cos\theta,
    & \bar a_4(\theta)=2-\sin\theta, & \bar a_4'(\theta)=-\cos\theta .
\end{array}
\]
Set \(p_0=(1,0)\), \(r_0=1\), \(\theta_0=0\), and
\(x_0=(p_0,0)\).  Suppose inductively that
\(x_{2m}=(p_m,0)\), where
\(p_m=r_m(\cos\theta_m,\sin\theta_m)\) and \(r_m\geq1\).
Since $\max\{|\sin\theta|,|\cos\theta|\}\geq1/\sqrt2$, one of the four
displayed derivatives is at most \(-1/\sqrt2\).  Choose \(i_m\) so that
\(\bar a_{i_m}'(\theta_m)\leq-1/\sqrt2\).  Since
\(\alpha_{2m+1}\leq\alpha_{2m}\),
\[
    g_{2m}:=
    \left(
      0,\frac{\alpha_{2m+1}}{\alpha_{2m}}a_{i_m}(p_m)\mathbf e_{i_m}
    \right)
    \in\partial f(x_{2m}).
\]
Indeed, this follows from \eqref{eq:counterexample_subdifferential} by taking
\(\sigma_{i_m}=\alpha_{2m+1}/\alpha_{2m}\) and \(\sigma_j=0\) for
\(j\ne i_m\).  The choice is admissible because
\(\alpha_{2m+1}/\alpha_{2m}
=\sqrt{(2m+K)/(2m+1+K)}\in(0,1]
\subset[-1,1]=\partial|\cdot|(0)\).
Thus
\[
    x_{2m+1}
    =
    x_{2m}-\alpha_{2m}g_{2m}
    =
    (p_m,-\alpha_{2m+1}a_{i_m}(p_m)\mathbf e_{i_m}).
\]
At this odd iterate, take \(\sigma_{i_m}=-1\) and \(\sigma_j=0\) for every
\(j\ne i_m\) in \eqref{eq:counterexample_subdifferential}.  This gives
\[
    g_{2m+1}:=
    \left(
      \alpha_{2m+1}a_{i_m}(p_m)\nabla a_{i_m}(p_m),
      -a_{i_m}(p_m)\mathbf e_{i_m}
    \right)
    \in\partial f(x_{2m+1}).
\]
Indeed, \(z_{i_m}=-\alpha_{2m+1}a_{i_m}(p_m)<0\), so
\(\partial|\cdot|(z_{i_m})=\{-1\}\); for \(j\ne i_m\), one has \(z_j=0\)
and \(0\in\partial|\cdot|(0)=[-1,1]\).
Therefore
\(x_{2m+2}=x_{2m+1}-\alpha_{2m+1}g_{2m+1}=(p_{m+1},0)\), where
\(p_{m+1}=p_m-\alpha_{2m+1}^2a_{i_m}(p_m)
\nabla a_{i_m}(p_m)\).

Define
\begin{equation}
 \gamma_m:=
 \frac{\alpha_{2m+1}^2a_{i_m}(p_m)
 (-\bar a_{i_m}'(\theta_m))}{r_m^2}>0.
\label{eq:gamma_m_exp}
\end{equation}

Since \(a_i(r(\cos\theta,\sin\theta))=\bar a_i(\theta)\) for \(r>1/2\),
differentiating at \(p_m\) gives
\[
 \nabla a_{i_m}(p_m)
 =r_m^{-1}\bar a_{i_m}'(\theta_m)
   (-\sin\theta_m,\cos\theta_m).
\]
Hence
\begin{align*}
 p_{m+1}
 &=r_m\left(
      \cos\theta_m-\gamma_m\sin\theta_m,
      \sin\theta_m+\gamma_m\cos\theta_m
    \right)\\
 &=r_m\sqrt{1+\gamma_m^2}\left(
      \cos(\theta_m+\arctan\gamma_m),
      \sin(\theta_m+\arctan\gamma_m)
    \right).
\end{align*}
Consequently,
\begin{equation}
 r_{m+1}=r_m\sqrt{1+\gamma_m^2},
 \qquad
 \theta_{m+1}=\theta_m+\arctan\gamma_m.
\label{eq:counterexample_polar_recursion}
\end{equation}
In particular, \(r_{m+1}\geq r_m\geq1\), so the hypothesis \(r_m\geq1\)
propagates from \(m\) to \(m+1\).  Thus the preceding two updates define
\(x_{2m+1}\) and \(x_{2m+2}\) for every \(m\geq0\).

\runinheading{Step~3: Bound the radial motion}
Since \(a_{i_m}(p_m)\leq3\),
\(|\bar a_{i_m}'(\theta_m)|\leq1\), and \(r_m\geq1\),
equation~\eqref{eq:gamma_m_exp} gives
\(0<\gamma_m\leq3\alpha_{2m+1}^2\).  Moreover,
\[
 \sum_{m=0}^{\infty}\alpha_{2m+1}^4
 =\sum_{m=0}^{\infty}(2m+1+K)^{-2}<\infty.
\]
Hence $\sum_{m=0}^{\infty}\gamma_m^2<\infty$.  Iterating the radial
identity in \eqref{eq:counterexample_polar_recursion} now gives
\[
\begin{aligned}
    r_m
    &=r_0\prod_{j=0}^{m-1}\sqrt{1+\gamma_j^2} \\
    &\leq r_0\exp\left(\frac12\sum_{j=0}^{m-1}\gamma_j^2\right)
     \leq r_0\exp\left(\frac92\sum_{j=0}^{\infty}
                              \alpha_{2j+1}^4\right)=:R.
\end{aligned}
\]
The sequence \((r_m)\) is increasing and bounded, so
\(r_m\to r_\infty\) for some finite \(r_\infty\geq1\).

\runinheading{Step~4: Identify the angular motion}
We next show that the angles wind around the circle indefinitely, with
vanishing increments.  Since \(a_{i_m}(p_m)\geq1\) and
\(-\bar a_{i_m}'(\theta_m)\geq1/\sqrt2\), equation
\eqref{eq:gamma_m_exp} gives
\(
    \gamma_m\geq\alpha_{2m+1}^2/(\sqrt2\,R^2).
\)
The squared step sizes at odd indices form a divergent shifted harmonic
series:
\[
 \sum_{m=0}^{\infty}\alpha_{2m+1}^2
 =\sum_{m=0}^{\infty}(2m+1+K)^{-1}=+\infty.
\]
It follows that $\sum_m\gamma_m=+\infty$.  Since $\gamma_m\to0$, eventually
$\arctan\gamma_m\geq\gamma_m/2$.  Hence
\[
 \sum_{m=0}^\infty\arctan\gamma_m=+\infty,
 \qquad
 \theta_{m+1}-\theta_m=\arctan\gamma_m\longrightarrow0.
\]
Thus \(\theta_m\) is unbounded and has vanishing increments.  For any
$\phi\in[0,2\pi)$ and each integer $\ell\geq1$, let $m_\ell$ be the first
index satisfying \(\theta_{m_\ell}\geq2\pi\ell+\phi\).
Then \(m_\ell\to\infty\), and the overshoot satisfies
\[
 0\leq\theta_{m_\ell}-(2\pi\ell+\phi)
 <\theta_{m_\ell}-\theta_{m_\ell-1}\longrightarrow0.
\]
Therefore
\begin{equation}
 p_{m_\ell}\longrightarrow
 r_\infty(\cos\phi,\sin\phi).
\label{eq:counterexample_circle_subsequence}
\end{equation}
Taking \(\phi=0\) and \(\phi=\pi\) gives two distinct
subsequential limits.  Hence \((p_m)\), and therefore \((x_{2m})\), is
not convergent.

\runinheading{Step~5: Pass from the even subsequence to the full sequence}
The odd and even iterates satisfy
\[
 x_{2m+1}
 =\bigl(p_m,-\alpha_{2m+1}a_{i_m}(p_m)\mathbf e_{i_m}\bigr),
 \qquad x_{2m}=(p_m,0).
\]
Since \(|p_m|\leq R\), \(a_i\leq3\), and \(\alpha_k\leq1\), the full
sequence is bounded and, since its even subsequence is nonconvergent, is
itself nonconvergent.  Moreover, \(f(x_{2m})=0\) and
\(f(x_{2m+1})=\alpha_{2m+1}a_{i_m}(p_m)^2\to0\).

Since $r_m\to r_\infty$, every accumulation point of the even subsequence
lies on the circle of radius $r_\infty$, and
\eqref{eq:counterexample_circle_subsequence} supplies a subsequence converging
to each point of that circle.  Also,
\[
 |x_{2m+1}-x_{2m}|
 =\alpha_{2m+1}a_{i_m}(p_m)\leq3\alpha_{2m+1}\to0,
\]
so the odd and even subsequences have exactly the same accumulation points.
Thus every accumulation point has the form
\(\bigl(r_\infty(\cos\phi,\sin\phi),0\bigr)\), where
\(\phi\in[0,2\pi)\).
\end{proof}

\subsection{Active-manifold geometry and inexact update estimates}
\label{sec:active_geometry}
The counterexample in Theorem~\ref{thm:counterexample-half} isolates a
geometric mechanism caused by the failure of square summability.  Step~4 of
its proof shows that the cumulative angular motion
of the constructed sequence along its limiting circle grows on the scale of
\(\sum_k\alpha_k^2\).  Because this series diverges, the sequence winds
indefinitely around the circle and does not converge.  This mechanism suggests
that square summability is the natural condition for eliminating the persistent
tangential drift.

Motivated by this observation, we impose the following local geometric
assumption at the critical point.  It separates attraction toward an active
manifold from motion tangent to it and imposes a Verdier estimate along that
manifold.  The assumption builds on the partial-smoothness
framework of Lewis \cite{lewis2002active} and is closely related to geometric
conditions appearing in recent analyses of saddle avoidance for stochastic
subgradient methods \cite{bianchi2023stochastic,davis2025active}.

To formulate the assumption, recall from
\cite[Definition~8.3(a), p.~301]{rockafellar1998variational} that the
Fr\'echet, or regular, subdifferential of a locally Lipschitz function
\(f\colon\mathbb R^n\to\mathbb R\) is the set-valued mapping
\(\widehat\partial f:\mathbb R^n\rightrightarrows\mathbb R^n\) characterized
as follows: a vector \(v\in\mathbb R^n\) belongs to
\(\widehat\partial f(x)\) precisely when
\[
 f(z)\geq f(x)+\langle v,z-x\rangle+o(|z-x|)
 \qquad(z\to x).
\]

\begin{assumption}[Active geometry]
\label{assumption:active_aiming_verdier}
Let \(f\colon\mathbb R^n\to\mathbb R\) be locally Lipschitz and let
\(x^*\in\mathbb R^n\).  Suppose that
\(0\in\widehat\partial f(x^*)\) and that there exist a definable \(C^3\) manifold \(A\subset\mathbb R^n\) containing \(x^*\), an open neighborhood
\(U_0\) of \(x^*\), and constants \(\sigma_{\mathrm A},C_{\mathrm V}>0\) such that:
\begin{enumerate}[label=(\roman*),leftmargin=*]
    \item \emph{Activity.}  The restriction \(f|_A\) is \(C^2\) on
    \(A\cap U_0\), and $|v|\geq\sigma_{\mathrm A}$ for all $z\in U_0\setminus A,\ v\in\partial f(z)$.
    \item \emph{Verdier compatibility.}  For every \(z\in U_0\),
    \(y\in A\cap U_0\), and \(v\in\partial f(z)\),
    \begin{equation}
       \left|P_{T_A(y)}v-\nabla_A f(y)\right|
       \leq C_{\mathrm V}|z-y|.
    \label{eq:active_verdier}
    \end{equation}
\end{enumerate}
Any such \(A\) is called an \emph{active manifold} at \(x^*\).
\end{assumption}

At suitable critical points, the conditions in
Assumption~\ref{assumption:active_aiming_verdier} hold in many
structured nonsmooth models.  For \(\ell_1\)- and
group-sparsity regularization
\cite{tibshirani1996regression,yuan2006model,meier2008group}, the active
manifold records the nonzero coordinates or groups; for nuclear-norm
regularization \cite{candes2010power,recht2011simpler}, it is the fixed-rank
manifold.  Detailed active-manifold descriptions for these and related
regularizers appear in \cite[Examples~10--15]{vaiter2017degrees}.  The same
geometry arises in exact \(\ell_1\)-penalty formulations of definable smooth
constrained problems \cite{han1979exact,burke1991exact} and in
maximum-eigenvalue optimization \cite[Example~3.6]{lewis2002active}.  General
criteria for verifying Assumption~\ref{assumption:active_aiming_verdier} in
these and related settings are given in
\cite[Section~3.5]{davis2025active}.

We next verify Assumption~\ref{assumption:active_aiming_verdier} for the
counterexample constructed in
Section~\ref{sec:square_summability_obstruction}.

\begin{proposition}[Counterexample geometry]
\label{proposition:counterexample_active_geometry}
For the objective \(f\) in
\eqref{eq:counterexample_objective_definition}, every point
\(x^*=(p^*,0)\) with \(|p^*|>1/2\) satisfies
Assumption~\ref{assumption:active_aiming_verdier} with active manifold
\(A=\mathbb R^2\times\{0\}\).
\end{proposition}

\begin{proof}
\runinheading{Step~1: Verify the manifold and activity}
Fix \(x^*=(p^*,0)\) with \(|p^*|>1/2\), choose
\(\delta\in(0,|p^*|-1/2)\), and set
\(U_0:=\mathring B(x^*,\delta)\).  Every
\(x=(p,w)\in U_0\) then satisfies \(|p|>1/2\).  The affine space
\(A=\mathbb R^2\times\{0\}\) is a definable \(C^\infty\) manifold and
\(f|_A\equiv0\).  Moreover, \(a_i\geq1\) for every \(i\), so
\(f\geq0=f(x^*)\) and hence \(0\in\widehat\partial f(x^*)\).

To verify activity, let \(x=(p,w)\in U_0\setminus A\), and choose
\(i\) such that \(w_i\neq0\).  By
\eqref{eq:counterexample_subdifferential}, the \(i\)th component of the
\(\mathbb R^4\)-block of every \(v\in\partial f(x)\) equals
\(a_i(p)\operatorname{sgn}(w_i)\).  Consequently,
\(|v|\geq a_i(p)\geq1\), so the activity condition holds with \(\sigma_{\mathrm A}=1\).

\runinheading{Step~2: Verify Verdier compatibility}
On \(|p|>1/2\), we have
\(q(p)=p/|p|\) and
\(Dq(p)=|p|^{-1}(I-pp^\top/|p|^2)\).  Therefore
\(|\nabla a_i(p)|\leq |p|^{-1}<2\) for every \(i\).  Let
\(x=(p,w)\in U_0\), \(y=(\bar p,0)\in A\cap U_0\), and
\(v\in\partial f(x)\).  Since
\(T_A(y)=\mathbb R^2\times\{0\}\), \(\nabla_A f(y)=0\), and
\eqref{eq:counterexample_subdifferential} holds, we obtain
\[
 \left|P_{T_A(y)}v-\nabla_A f(y)\right|
 =\left|\sum_{i=1}^4|w_i|\nabla a_i(p)\right|
 \leq2\|w\|_1\leq4|w|\leq4|x-y|.
\]
Thus \eqref{eq:active_verdier} holds with \(C_{\mathrm V}=4\).
\end{proof}

Together, Theorem~\ref{thm:counterexample-half} and
Proposition~\ref{proposition:counterexample_active_geometry} show that the
bounded-sequence convergence conclusion of the subgradient method cannot be
extended to \(a=1/2\), even when every accumulation point satisfies
Assumption~\ref{assumption:active_aiming_verdier}.  We next collect the local
consequences of Assumption~\ref{assumption:active_aiming_verdier} that will be
used in the later convergence analyses.

\begin{lemma}[Local active estimates]
\label{lemma:active_local_consequences}
Let \(f\colon\mathbb R^n\to\mathbb R\) be locally Lipschitz and definable,
and suppose that \(x^*\) satisfies
Assumption~\ref{assumption:active_aiming_verdier} with active manifold \(A\).
For every \(\underline\theta\in(0,1)\), there exist
\(\varrho>0\), \(\theta\in(\underline\theta,1)\), \(L\geq1\), and
\(\mu>0\) with the following properties.
\begin{enumerate}[label=(\roman*),leftmargin=*]
    \item \emph{Projection geometry.}  The projection \(P_A\) is
    single-valued and \(C^2\) on \(B(x^*,8\varrho)\), and
    \[
       P_A(B(x^*,4\varrho))
       \subset A\cap B(x^*,8\varrho).
    \]
    Its derivative is bounded and Lipschitz there, and
    \begin{equation}
       DP_A(y)=P_{T_A(y)},
       \qquad\forall y\in A\cap B(x^*,8\varrho).
    \label{eq:active_projection_derivative}
    \end{equation}

    \item \emph{Uniform local bounds.}  For the indicated points,
    \begin{equation}
    \begin{aligned}
       |f(z)-f(z')|&\leq L|z-z'|
       &&\forall z,z'\in B(x^*,8\varrho),\\
       |v|&\leq L
       &&\forall z\in B(x^*,8\varrho),\ v\in\partial f(z),\\
       |DP_A(z)|&\leq L,\qquad
       |DP_A(z)-DP_A(z')|\leq L|z-z'|
       &&\forall z,z'\in B(x^*,8\varrho),\\
       |\nabla_A f(y)|&\leq L,\qquad
       |\nabla_A f(y)-\nabla_A f(y')|\leq L|y-y'|
       &&\forall y,y'\in A\cap B(x^*,8\varrho).
    \end{aligned}
    \label{eq:active_local_uniform_constants}
    \end{equation}

    \item\label{item:active_goldstein_enlargement}
    \emph{Active geometry and its Goldstein enlargement.}  One has
    \(\nabla_A f(x^*)=0\).  Let
    \(x\in\mathring B(x^*,2\varrho)\) and set \(y=P_A(x)\).  Suppose that \(r\geq0\) and
    \(B(x,r)\subset\mathring B(x^*,4\varrho)\).  Then every
    \(u\in\partial f(x;r)\) satisfies \(|u|\leq L\) and
    \begin{equation}
    \begin{aligned}
       |P_{T_A(y)}u-\nabla_A f(y)|
       &\leq L\bigl(d(x,A)+r\bigr),\\
       \langle u,x-y\rangle&\geq\mu d(x,A)-Lr.
    \end{aligned}
    \label{eq:active_aiming}
    \end{equation}

    \item \emph{Riemannian KL inequality.}  For
    \(y\in A\cap B(x^*,8\varrho)\),
    \begin{equation}
       |f(y)-f(x^*)|^\theta
       \leq L|\nabla_A f(y)|.
    \label{eq:local_active_KL}
    \end{equation}
\end{enumerate}
\end{lemma}

\begin{proof}
\runinheading{Step~1: Projection geometry}
Choose bounded open neighborhoods \(W,V\) of \(x^*\) with
\(\overline W\subset V\) such that \(A\cap V\) is closed in \(V\).  The set \(A\setminus V\) is separated from
\(\overline W\).  Hence, after shrinking \(W\), the distance and nearest
points to \(A\) agree on \(W\) with those to \(A\cap V\).  The
local existence and regularity results for the orthogonal projection
\cite[(3.6) and Theorems~(3.8) and~(4.1)]{dudek1994nonlinear}, applied to
$A\cap V$, give a neighborhood of \(x^*\) on which \(P_A\) is
single-valued and \(C^2\).  Choose \(\varrho>0\) sufficiently small that
\(B(x^*,8\varrho)\) lies in this neighborhood.  Since \(P_A(z)\) is a
nearest point to \(z\) in \(A\) and \(x^*\in A\),
\begin{equation}
 |z-P_A(z)|=d(z,A)\leq|z-x^*|,
 \qquad |P_A(z)-x^*|\leq2|z-x^*|.
\label{eq:active_projection_distance_bound}
\end{equation}
The stated containment follows from the second estimate.  The identity
\eqref{eq:active_projection_derivative} follows from
\cite[Theorem~(4.1)(i)]{dudek1994nonlinear}.
The \(C^2\) regularity of \(P_A\) makes its derivative bounded and Lipschitz
on \(B(x^*,8\varrho)\).

\runinheading{Step~2: Uniform local bounds}
Shrink \(\varrho\), if necessary, so that the closed balls in the statement
lie both in the projection neighborhood obtained in Step~1 and in \(U_0\)
from Assumption~\ref{assumption:active_aiming_verdier}.  Local Lipschitz
continuity of \(f\) and local boundedness of its Clarke subdifferential
\cite[Propositions~2.1.2 and~2.1.5]{clarke1990} give the first two bounds in
\eqref{eq:active_local_uniform_constants}.  Step~1 and
Assumption~\ref{assumption:active_aiming_verdier}(i) give the projection and
Riemannian-gradient bounds, respectively.  By compactness, all these bounds
hold with a common constant \(L_0\geq1\).

\runinheading{Step~3: Active geometry and its Goldstein enlargement}
The defining inequality for \(0\in\widehat\partial f(x^*)\), restricted to
smooth curves in \(A\) through \(x^*\), shows that every directional
derivative of \(f|_A\) at \(x^*\) vanishes.  Hence \(\nabla_A f(x^*)=0\).
Increase \(L_0\), if necessary, so that it also dominates the Verdier
constant in Assumption~\ref{assumption:active_aiming_verdier}.
The Verdier estimate \eqref{eq:active_verdier} gives strong
\((a)\)-regularity along \(A\) by
\cite[Definitions~3.2--3.3 and Theorem~3.6]{davis2025active}.  Applied to
\(\operatorname{epi}f\) and \(\operatorname{gph}(f|_A)\),
\cite[Definitions~3.1 and~3.3 and Theorem~3.11]{davis2025active} then gives
\((b_=)\)-regularity, hence \((b_{\leq})\)-regularity; the required local
closedness and smoothness follow from the continuity of \(f\) and the
\(C^2\)-smoothness of \(f|_A\).  Activity,
\(0\in\widehat\partial f(x^*)\), and
\cite[Corollary~3.5]{davis2025active} therefore yield \(\mu>0\) such that,
after shrinking \(\varrho\),
\begin{equation}
 \langle v,z-P_A(z)\rangle\geq\mu d(z,A)
\label{eq:active_base_aiming}
\end{equation}
for \(z\in B(x^*,8\varrho)\) and \(v\in\partial f(z)\).  This is also the angle
condition of \cite[Definition~5(iv)]{bianchi2023stochastic}.

It remains to pass these estimates to the Goldstein enlargement.  By the
definition of \(\partial f(x;r)\) in Section~\ref{sec:inexact}, \(u\) is a
finite convex combination
\(u=\sum_{j=1}^N\omega_jv_j\), where \(N\geq1\),
\(z_j\in B(x,r)\), \(v_j\in\partial f(z_j)\), \(\omega_j\geq0\), and
\(\sum_{j=1}^N\omega_j=1\).  Equations
\eqref{eq:active_local_uniform_constants} and \eqref{eq:active_verdier} give
\[
 |v_j|\leq L_0,\qquad
 |P_{T_A(y)}v_j-\nabla_A f(y)|
 \leq L_0|z_j-y|\leq L_0(d(x,A)+r).
\]
Moreover, the triangle inequality gives
\(d(z_j,A)\geq d(x,A)-|z_j-x|\geq d(x,A)-r\), while
\eqref{eq:active_local_uniform_constants} makes \(P_A\) \(L_0\)-Lipschitz.
Consequently, \eqref{eq:active_base_aiming} gives
\[
\begin{aligned}
 \langle v_j,x-y\rangle
 &=\langle v_j,z_j-P_A(z_j)\rangle
   +\left\langle v_j,
       (x-z_j)+(P_A(z_j)-P_A(x))\right\rangle\\
 &\geq\mu d(x,A)-\bigl(\mu+L_0(1+L_0)\bigr)r.
\end{aligned}
\]
Taking the corresponding convex combinations and setting
\(L:=\max\{L_0,\mu+L_0(1+L_0)\}\) proves the bound \(|u|\leq L\) and the two
estimates in \eqref{eq:active_aiming}, while retaining the bounds established
in Steps~1--2.

\runinheading{Step~4: Riemannian KL inequality}
Set \(g:=f|_A-f(x^*)\) on a sufficiently small bounded definable open patch
of \(A\) around \(x^*\). The definable-manifold Kurdyka--\L{}ojasiewicz
inequality \cite[Corollary~12 and Remark~7]{bolte2007clarke}, together with
the growth dichotomy in the fixed polynomially bounded structure
\cite[Growth Dichotomy~4.12]{van1996geometric}, gives
\(\theta_0\in(0,1)\) and \(C_{\rm KL}>0\) such that
\(|g(y)|^{\theta_0}\leq C_{\rm KL}|\nabla_A f(y)|\) near \(x^*\).
Shrink the radius so that \(|g(y)|\leq1\), choose
\(\theta\in(\max\{\underline\theta,\theta_0\},1)\), and increase \(L\) so
that \(L\geq C_{\rm KL}\).  Then \(|g(y)|^\theta\leq|g(y)|^{\theta_0}\),
which proves \eqref{eq:local_active_KL}.
\end{proof}

Section~\ref{sec:first-order} placed the momentum and stochastic subgradient
methods in the inexact-subgradient framework through different reductions.
Lemma~\ref{lemma:momentum_shadow} gives the momentum shadow update
\eqref{eq:momentum_matched_shadow}, whereas Lemma~\ref{lemma:stoc_error}
passes to a window-endpoint subsequence of the stochastic iterates.  Here we
reuse \eqref{eq:momentum_matched_shadow} for momentum but apply the
active-geometry estimates directly to the original stochastic update
\eqref{eq:inexact stoc}.  Both updates have the local form
\begin{equation}
 x_{k+1}=x_k-\alpha_k(u_k+\varepsilon_k),
 \qquad u_k\in\partial f(x_k;r_k),
 \qquad r_k\geq0.
\label{eq:active_inexact_meta_update}
\end{equation}
For the momentum shadow sequence, \eqref{eq:momentum_matched_shadow} gives,
after relabeling, \(r_k=c_b\alpha_k\) for some \(c_b>0\) and
\(\varepsilon_k=0\); for the stochastic subgradient sequence
\eqref{eq:inexact stoc}, \(u_k=v_k\), \(r_k=0\),
and Assumption~\ref{assumption:stochastic_noise} supplies the
martingale-difference and conditional second-moment properties of
\(\varepsilon_k\).  For the inexact update
\eqref{eq:active_inexact_meta_update}, the next proposition
isolates the distance-to-manifold and projected estimates supplied by the active geometry;
the later method-specific analyses control the enlargement and noise terms.
Related distance-to-manifold and projected recursions for exact stochastic subgradient steps
appear in \cite[Propositions~5.1--5.2]{davis2025active} and
\cite[Lemma~3 and Proposition~4]{bianchi2023stochastic}.  The proposition
below retains the Goldstein-enlargement and additive-error terms and adds the
cubic distance-to-manifold and projected-descent estimates needed here.

\begin{proposition}[Active-step estimates]
\label{proposition:active_inexact_local_estimates}
Let \(f\colon\mathbb R^n\to\mathbb R\) be locally Lipschitz and definable,
and suppose that \(x^*\) satisfies
Assumption~\ref{assumption:active_aiming_verdier} with active manifold \(A\).
Fix constants \(\varrho,\theta,L,\mu\) and the projection \(P_A\) supplied
by Lemma~\ref{lemma:active_local_consequences}.  There exist
\(\bar\alpha\in(0,1]\), \(c_{\rm N}>0\), and \(C\geq1\), depending only
on these fixed local data, with the following property.

Let \(p<q\), and suppose that \eqref{eq:active_inexact_meta_update} holds for
\(i\in\llbracket p,q-1\rrbracket\), where
\(0<\alpha_i\leq\bar\alpha\),
\(x_{\llbracket p,q\rrbracket}\subset\mathring B(x^*,\varrho)\), and
\(B(x_i,r_i)\subset\mathring B(x^*,4\varrho)\).
For \(i\in\llbracket p,q\rrbracket\), set
\(y_i:=P_A(x_i)\), \(n_i:=x_i-y_i\), and \(d_i:=|n_i|\).
Then, for every \(i\in\llbracket p,q-1\rrbracket\), the following estimates
hold.
\begin{enumerate}[label=\emph{(\roman*)},leftmargin=*]
\item \emph{Distance to the active manifold.}  The distance
satisfies
\begin{equation}
 d_{i+1}^2+\mu\alpha_i d_i
 \leq d_i^2-2\alpha_i\langle\varepsilon_i,n_i\rangle
 +C\bigl(\alpha_i r_i
       +\alpha_i^2(1+|\varepsilon_i|^2)\bigr).
 \label{eq:active_meta_normal}
\end{equation}
If \(\varepsilon_i=0\), then the sharper cubic estimate
\begin{equation}
 d_{i+1}^3+c_{\rm N}\alpha_i d_i^2
 \leq d_i^3+C(r_i+\alpha_i)^3
\label{eq:active_meta_cubic_normal}
\end{equation}
also holds.

\item \emph{Projected motion.}  The projected iterates satisfy
\begin{align}
 |y_{i+1}-y_i+\alpha_i\nabla_A f(y_i)
       +\alpha_iDP_A(x_i)\varepsilon_i|
 &\leq C\bigl(\alpha_i(d_i+r_i)
       +\alpha_i^2(1+|\varepsilon_i|^2)\bigr),
 \label{eq:active_meta_projected_recursion}\\
 |y_{i+1}-y_i|
 &\leq C\alpha_i
       (|\nabla_A f(y_i)|+d_i+r_i+\alpha_i+|\varepsilon_i|).
 \label{eq:active_meta_projected_move}
\end{align}

\needspace{13\baselineskip}
\item \emph{Projected descent and gradient variation.}  One has
\begin{align}
 f(y_{i+1})-f(y_i)
 &\leq-\frac12\alpha_i|\nabla_A f(y_i)|^2 -\alpha_i\left\langle\varepsilon_i,
       DP_A(x_i)^*\nabla_A f(y_i)\right\rangle \notag\\
 &\quad+C\left(
       \alpha_i(d_i+r_i)^2+\alpha_i^3
       +\alpha_i^2(|\varepsilon_i|+|\varepsilon_i|^2)
       \right),
 \label{eq:active_meta_projected_descent}\\
 |\nabla_A f(y_{i+1})|
 &\leq |\nabla_A f(y_i)|+C\alpha_i
       (|\nabla_A f(y_i)|+d_i+r_i+\alpha_i+|\varepsilon_i|) \notag\\
 &\leq |\nabla_A f(y_i)|+C\alpha_i(1+|\varepsilon_i|).
 \label{eq:active_meta_gradient_variation}
\end{align}
\end{enumerate}
In particular, for every \(c_r\geq0\), there is a constant \(C'\geq1\),
depending only on \(c_r\) and the fixed local data, such that whenever
\(r_i\leq c_r\alpha_i\) and \(\varepsilon_i=0\), the remainders in
\eqref{eq:active_meta_projected_descent} and
\eqref{eq:active_meta_cubic_normal} are bounded by
\(C'(\alpha_i d_i^2+\alpha_i^3)\) and \(C'\alpha_i^3\), respectively.
\end{proposition}

\begin{proof}
\runinheading{Step~1: Local setup and geometric inputs}
Decrease \(\bar\alpha\), if necessary, so that
\(\bar\alpha L\leq\varrho/2\).  The nearest-point bounds
\eqref{eq:active_projection_distance_bound} give
\[
 |x-P_A(x)|=d(x,A)\leq|x-x^*|,
 \qquad |P_A(x)-x^*|\leq2|x-x^*|.
\]
Consequently, all projected points used below lie in
\(A\cap B(x^*,8\varrho)\).  Moreover,
\(h:=(f|_A)\circ P_A\) is \(C^2\) on a neighborhood of
\(B(x^*,4\varrho)\), with uniformly bounded first and second derivatives.
Indeed, Lemma~\ref{lemma:active_local_consequences}(i) makes \(P_A\) \(C^2\)
there and maps this ball into the \(C^2\) patch of \(f|_A\) supplied by
Assumption~\ref{assumption:active_aiming_verdier}(i).  Thus, after increasing
\(C\),
\begin{equation}
 \sup_{z\in B(x^*,4\varrho)}
 \bigl(|Dh(z)|+|D^2h(z)|\bigr)\leq C.
\label{eq:active_meta_h_derivative_bounds}
\end{equation}

For brevity, fix \(i\in\llbracket p,q-1\rrbracket\), suppress the index
\(i\) on \(x,y,n,d,u,r,\alpha,\varepsilon\), and set
\(g:=\nabla_A f(y)\) and \(s:=|g|\).
Item~\ref{item:active_goldstein_enlargement} of
Lemma~\ref{lemma:active_local_consequences} gives
\begin{equation}
 |u|\leq L,\qquad
 \langle u,n\rangle\geq\mu d-Lr,
 \qquad
 |P_{T_A(y)}u-g|\leq L(d+r).
\label{eq:active_meta_geometric_input}
\end{equation}
By \eqref{eq:active_projection_derivative},
\eqref{eq:active_local_uniform_constants}, and
\eqref{eq:active_meta_geometric_input}, and since \(|x-y|=d\), we have
\begin{equation}
 |DP_A(x)u-g|
 \leq |(DP_A(x)-DP_A(y))u|
      +|P_{T_A(y)}u-g|
 \leq C(d+r).
\label{eq:active_meta_tangent_input}
\end{equation}

\runinheading{Step~2: Distance estimates}
Since \(y\in A\), the update \eqref{eq:active_inexact_meta_update} and
\eqref{eq:active_meta_geometric_input} yield
\begin{align*}
 d_{i+1}^2
 &\leq|x_{i+1}-y|^2\\
 &=d^2-2\alpha\langle u,n\rangle
       -2\alpha\langle\varepsilon,n\rangle
       +\alpha^2|u+\varepsilon|^2\\
 &\leq d^2-2\mu\alpha d
       -2\alpha\langle\varepsilon,n\rangle
       +C\bigl(\alpha r+\alpha^2(1+|\varepsilon|^2)\bigr).
\end{align*}
Moving one copy of \(\mu\alpha d\) to the left and dropping the other
proves \eqref{eq:active_meta_normal}.

Suppose now that \(\varepsilon=0\), and set \(\eta:=r+\alpha\).  The preceding
estimate gives
\begin{equation}
 d_{i+1}^2\leq d^2-2\mu\alpha d+C\alpha\eta.
\label{eq:active_meta_cubic_start}
\end{equation}
Choose a fixed threshold \(\Lambda_{\rm N}\geq1\) sufficiently large.  If
\(d\geq \Lambda_{\rm N}\eta\), then
\eqref{eq:active_meta_cubic_start} gives
\(d_{i+1}^2\leq d^2-\mu\alpha d\leq d^2\).  Hence
\[
 d^3-d_{i+1}^3
 =(d^2-d_{i+1}^2)
   \frac{d^2+dd_{i+1}+d_{i+1}^2}{d+d_{i+1}}
 \geq\frac{\mu}{2}\alpha d^2.
\]
If \(d<\Lambda_{\rm N}\eta\), then
\eqref{eq:active_meta_cubic_start}, after dropping its negative term, gives
\[
 d_{i+1}^2\leq d^2+C\alpha\eta
 \leq(\Lambda_{\rm N}^2+C)\eta^2,
 \qquad
 \alpha d^2\leq\Lambda_{\rm N}^2\eta^3.
\]
Taking \(c_{\rm N}:=\mu/4\) and increasing \(C\) now proves
\eqref{eq:active_meta_cubic_normal} in both cases.

\runinheading{Step~3: Projected and tangential estimates}
Since \(x,x_{i+1}\in B(x^*,\varrho)\), the segment
\([x,x_{i+1}]\) lies in \(B(x^*,4\varrho)\).  Taylor expansion of \(P_A\)
along this segment, using the Lipschitz bound
for \(DP_A\) in \eqref{eq:active_local_uniform_constants} and then
\eqref{eq:active_meta_tangent_input}, gives
\begin{align*}
 y_{i+1}-y
 &=-\alpha DP_A(x)(u+\varepsilon)+R,\\
 |R|&\leq C\alpha^2|u+\varepsilon|^2
       \leq C\alpha^2(1+|\varepsilon|^2),
\end{align*}
and therefore
\[
 |y_{i+1}-y+\alpha g+\alpha DP_A(x)\varepsilon|
 \leq C\bigl(\alpha(d+r)+\alpha^2(1+|\varepsilon|^2)\bigr).
\]
This proves \eqref{eq:active_meta_projected_recursion}.

To retain a cubic deterministic remainder in the descent estimate, introduce
the noiseless intermediate point
\[
 \bar x_{i+1}:=x-\alpha u,
 \qquad \bar y_{i+1}:=P_A(\bar x_{i+1}).
\]
The choice of \(\bar\alpha\) ensures
\(\bar x_{i+1}\in B(x^*,3\varrho/2)\).  Hence
\([x,\bar x_{i+1}]\subset B(x^*,2\varrho)\).  Taylor expansion of \(P_A\)
along this segment, using
\eqref{eq:active_local_uniform_constants} and
\eqref{eq:active_meta_tangent_input}, gives
\begin{equation}
 \bar y_{i+1}-y=-\alpha g+r_y,
 \qquad |r_y|\leq C\alpha(d+r+\alpha).
\label{eq:active_meta_noiseless_projected}
\end{equation}
Since \(x_{i+1},\bar x_{i+1}\in B(x^*,2\varrho)\), the Lipschitz bound for
\(P_A\) in \eqref{eq:active_local_uniform_constants} gives
\[
 |y_{i+1}-\bar y_{i+1}|
 \leq C|x_{i+1}-\bar x_{i+1}|
 =C\alpha|\varepsilon|.
\]
Together with \eqref{eq:active_meta_noiseless_projected}, this proves
\eqref{eq:active_meta_projected_move}.

For every \(v\in\mathbb R^n\),
\eqref{eq:active_projection_derivative} and \(g\in T_A(y)\) give
\[
 Dh(y)[v]
 =\langle g,DP_A(y)v\rangle
 =\langle g,v\rangle.
\]
The projection bound \eqref{eq:active_projection_distance_bound} gives
\(y,\bar y_{i+1}\in B(x^*,3\varrho)\), so their joining segment lies in
\(B(x^*,4\varrho)\).  Taylor expansion of \(h\) along this segment, using
\eqref{eq:active_meta_h_derivative_bounds} and then
\eqref{eq:active_meta_noiseless_projected}, yields
\[
 f(\bar y_{i+1})-f(y)
 \leq-\alpha s^2+\langle g,r_y\rangle
       +C|\bar y_{i+1}-y|^2.
\]
The two remainder terms satisfy
\[
 |\langle g,r_y\rangle|
 \leq C\alpha s(d+r)+C\alpha^2s
 \leq\frac14\alpha s^2
      +C\bigl(\alpha(d+r)^2+\alpha^3\bigr),
\]
and
\[
 |\bar y_{i+1}-y|^2
 \leq C\alpha^2s^2
      +C\alpha^2(d+r)^2+C\alpha^4.
\]
After decreasing \(\bar\alpha\), the first term in the last display is
absorbed into another quarter of \(\alpha s^2\), while the remaining
terms are bounded by \(C(\alpha(d+r)^2+\alpha^3)\).
Therefore,
\[
 f(\bar y_{i+1})-f(y)
 \leq-\frac12\alpha s^2
      +C\bigl(\alpha(d+r)^2+\alpha^3\bigr).
\]

It remains to compare the true step with its noiseless intermediate point.
Since \(\bar x_{i+1},x_{i+1}\in B(x^*,2\varrho)\), their joining segment
lies in the same ball.  Taylor expansion of \(h\) along this segment, using
\eqref{eq:active_meta_h_derivative_bounds}, gives
\begin{align*}
 f(y_{i+1})-f(\bar y_{i+1})
 =h(x_{i+1})-h(\bar x_{i+1})\leq-\alpha\langle\nabla h(\bar x_{i+1}),\varepsilon\rangle
       +C\alpha^2|\varepsilon|^2.
\end{align*}
The derivative bound \eqref{eq:active_meta_h_derivative_bounds}, the identity
\(\nabla h(x)=DP_A(x)^*g\), and \(|u|\leq L\) imply
\[
 |\nabla h(\bar x_{i+1})-DP_A(x)^*g|
 \leq C|\bar x_{i+1}-x|\leq C\alpha.
\]
Consequently,
\[
 f(y_{i+1})-f(\bar y_{i+1})
 \leq-\alpha\langle\varepsilon,DP_A(x)^*g\rangle
       +C\alpha^2(|\varepsilon|+|\varepsilon|^2).
\]
Adding the last estimate to the noiseless descent estimate proves
\eqref{eq:active_meta_projected_descent}.  Finally, the Lipschitz bound for
\(\nabla_A f\) in \eqref{eq:active_local_uniform_constants}, together with
\eqref{eq:active_meta_projected_move}, gives the first inequality in
\eqref{eq:active_meta_gradient_variation}.  The second follows from the
uniform bounds in \eqref{eq:active_local_uniform_constants}, since
\(d_i<\varrho\), \(r_i<4\varrho\), and
\(\alpha_i\leq\bar\alpha\).
\end{proof}

We next record a function-value convergence result for the common update
\eqref{eq:active_inexact_meta_update}.

\begin{corollary}[Inexact value convergence]
\label{lemma:perturbed_subgradient_value_convergence}
Let \(f\colon\mathbb R^n\to\mathbb R\) be locally Lipschitz and definable,
and let \((x_k)\) be a bounded sequence satisfying
\eqref{eq:active_inexact_meta_update} for every \(k\in\mathbb N\), with
\(\alpha_k>0\).  Suppose that
\(\sum_k\alpha_k=\infty\), \(\sum_k\alpha_k^2<\infty\),
\(r_k\to0\), and the series \(\sum_k\alpha_k\varepsilon_k\) converges.
Then \((f(x_k))\) converges.  If \(x^*\) is an accumulation point of
\((x_k)\), then its limit is
\(f(x^*)\).
\end{corollary}

\begin{proof}
\runinheading{Step~1: Verify the boundedness assumptions}
We apply \cite[Corollary~5.9]{davis2020stochastic} with
\(X=\mathbb R^n\), \(G=-\partial f\), \(y_k=-u_k\), and
\(\zeta_k=-\varepsilon_k\), where
\(u_k\in\partial f(x_k;r_k)\).  Items~1, 3, and 4 of
\cite[Assumption~A]{davis2020stochastic} follow immediately from the
hypotheses.  For Item~2, let
\(K:=\overline{\{x_k:k\in\mathbb N\}}\).  Since \(K\) is compact and
\(r_k\to0\), local boundedness of the Clarke subdifferential
\cite[Propositions~2.1.2 and~2.1.5]{clarke1990} uniformly bounds
\(\partial f(x_k;r_k)\), and hence \((y_k)\), on an eventual tail; the
remaining terms are finite in number.

\runinheading{Step~2: Verify graph consistency}
It remains to verify Item~5.  If \(x_{k_j}\to\bar x\), then
\(r_{k_j}\to0\), and upper semicontinuity of the Clarke subdifferential
\cite[Proposition~2.1.5]{clarke1990} gives
\begin{equation}
 d\bigl(u_{k_j},\partial f(\bar x)\bigr)\longrightarrow0.
\label{eq:active_graph_consistency}
\end{equation}
Indeed, every point in \(B(x_{k_j},r_{k_j})\) tends uniformly to \(\bar x\),
and \(\partial f(\bar x)+\epsilon B(0,1)\) is convex, so it contains the
convex hull defining \(\partial f(x_{k_j};r_{k_j})\) for all sufficiently
large \(j\).  Since \(\partial f(\bar x)\) is nonempty, compact, and convex
\cite[Proposition~2.1.2]{clarke1990}, convexity and Cesàro averaging yield
\[
 d\left(\frac1N\sum_{j=1}^Ny_{k_j},G(\bar x)\right)
 \leq\frac1N\sum_{j=1}^N
 d\bigl(u_{k_j},\partial f(\bar x)\bigr)\longrightarrow0,
\]
which is Item~5.

\runinheading{Step~3: Apply the value-convergence result}
Definability gives a finite definable \(C^p\) Whitney
stratification of \(\operatorname{gph}f\)
\cite[Lemma~8]{bolte2007clarke}; see also
\cite[Result~4.8(1)]{van1996geometric}.  Thus
\cite[Corollary~5.9]{davis2020stochastic} gives convergence of \(f(x_k)\).
Along any subsequence \(x_{k_j}\to x^*\), local Lipschitz continuity gives
\(f(x_{k_j})\to f(x^*)\), which identifies the limit.
\end{proof}

To convert the projected descent estimate
\eqref{eq:active_meta_projected_descent}, the Riemannian KL inequality
\eqref{eq:local_active_KL}, and the gradient-variation estimate
\eqref{eq:active_meta_gradient_variation} into a bound for the cumulative
projected motion in \eqref{eq:active_meta_projected_move}, we use the following
signed-power estimate with correction terms.

\begin{lemma}[Perturbed KL estimate]
\label{lemma:corrected_active_KL}
Let \(p<q\) be integers and let \(\theta\in(0,1)\).  Let
\(F_i\in\mathbb R\) and \(s_i,c_i\geq0\) for \(p\leq i\leq q\), with
\(c_q=0\), and let \(\alpha_i>0\) and \(\delta_i\geq0\) for
\(p\leq i<q\).  Suppose that, for some \(L>0\),
\(|F_i|^\theta\leq Ls_i\) for \(p\leq i\leq q\), and
\[
 s_{i+1}\leq s_i+\delta_i,
 \qquad
 (F_i+c_i)-(F_{i+1}+c_{i+1})
 \geq\frac12\alpha_i s_i^2
 \qquad\forall i\in\llbracket p,q-1\rrbracket.
\]
Then there are constants \(C_{\rm KL},C>0\), depending only on
\(L\) and \(\theta\), such that
\[
 \sum_{i=p}^{q-1}\alpha_i s_i\leq C_{\rm KL}\bigl(\varphi_\theta(F_p)-\varphi_\theta(F_q)\bigr)+Cc_p^{1-\theta}+C\sum_{i=p}^{q-1}\alpha_i(\delta_i+c_i^\theta+c_{i+1}^\theta).
\]
\end{lemma}

\begin{proof}
\runinheading{Step~1: Establish descent and secant bounds for the augmented values}
Set \(z_i:=F_i+c_i\) and
\(\rho_i:=\delta_i+c_i^\theta+c_{i+1}^\theta\) for \(p\leq i<q\).
Set \(M:=\max\{2,2L\}\). Since \(M\geq2\), the augmented descent
hypothesis gives the first inequality below.  Subadditivity of
\(t\mapsto t^\theta\), the KL bound, and
\(s_{i+1}\leq s_i+\delta_i\) give the second:
\[
\begin{aligned}
 z_i-z_{i+1}&\geq M^{-1}\alpha_i s_i^2,\\
 |z_i|^\theta+|z_{i+1}|^\theta
 &\leq2Ls_i+L\delta_i+c_i^\theta+c_{i+1}^\theta
 \leq M(s_i+\rho_i).
\end{aligned}
\]
In particular, \(z_i\geq z_{i+1}\).

Recall that
\(\varphi_\theta(t)=\operatorname{sgn}(t)|t|^{1-\theta}\).  We use the
following secant estimate:
\begin{equation}
 \varphi_\theta(x)-\varphi_\theta(y)
 \geq(1-\theta)\frac{x-y}{|x|^\theta+|y|^\theta}
 \qquad\forall x\geq y,\quad (x,y)\neq(0,0).
\label{eq:signed_power_secant}
\end{equation}
Indeed,
\(\varphi_\theta(x)-\varphi_\theta(y)
=(1-\theta)\int_y^x|t|^{-\theta}\,dt\), with the improper integral
interpreted across zero.  Since
\(|t|^\theta\leq |x|^\theta+|y|^\theta\) for \(t\in[y,x]\), this identity
implies \eqref{eq:signed_power_secant}.

\runinheading{Step~2: Derive the one-step length estimate}
Fix \(i\in\llbracket p,q-1\rrbracket\).  If \(z_i=z_{i+1}=0\), then
\(z_i-z_{i+1}\geq M^{-1}\alpha_i s_i^2\) and \(\alpha_i>0\) imply
\(s_i=0\).  Otherwise,
\(|z_i|^\theta+|z_{i+1}|^\theta\leq M(s_i+\rho_i)\) implies
\(s_i+\rho_i>0\), and
\eqref{eq:signed_power_secant} yields
\[
 \varphi_\theta(z_i)-\varphi_\theta(z_{i+1})
 \geq\frac{1-\theta}{M^2}
       \frac{\alpha_i s_i^2}{s_i+\rho_i}.
\]
For nonnegative \(u,v\) with \(u+v>0\), one has
\(u\leq u^2/(u+v)+v\), since the difference between the right- and
left-hand sides is \(v^2/(u+v)\).  Applying this inequality with
\(u=s_i\) and \(v=\rho_i\), with the zero case interpreted trivially, gives
\[
 \alpha_i s_i
 \leq\frac{M^2}{1-\theta}
       \bigl(\varphi_\theta(z_i)-\varphi_\theta(z_{i+1})\bigr)
       +\alpha_i\rho_i.
\]
\runinheading{Step~3: Telescope and remove the endpoint correction}
Summing and using \(z_p=F_p+c_p\) and \(z_q=F_q\) gives
\[
 \sum_{i=p}^{q-1}\alpha_i s_i
 \leq\frac{M^2}{1-\theta}
       \bigl(\varphi_\theta(F_p+c_p)-\varphi_\theta(F_q)\bigr)
       +\sum_{i=p}^{q-1}\alpha_i\rho_i.
\]
The signed-power H\"older estimate
\cite[Fact~4.12]{lai2025diameter} gives
\(\varphi_\theta(F_p+c_p)-\varphi_\theta(F_p)\leq2c_p^{1-\theta}\).
Taking \(C_{\rm KL}:=M^2/(1-\theta)\) and
\(C:=\max\{2C_{\rm KL},1\}\) proves the claim.
\end{proof}

\subsection{Momentum method}\label{sec:active_momentum}
Lemma~\ref{lemma:momentum_shadow} associates with the momentum iterates a
sequence \((\widehat x_k)\) satisfying \eqref{eq:momentum_matched_shadow} and
\(|\widehat x_k-x_k|\to0\).  We first prove convergence for bounded sequences
satisfying this enlarged-subgradient recurrence, using
Proposition~\ref{proposition:active_inexact_local_estimates} and explicit
bounds on the tails \(\sum_{i=k}^\infty\alpha_i^2\) and
\(\sum_{i=k}^\infty\alpha_i^3\).  The vanishing difference then transfers
convergence back to the original momentum iterates.

\begin{proposition}[Enlarged-subgradient convergence]
\label{proposition:active_goldstein}
Let \(f\colon\mathbb{R}^n\to\mathbb{R}\) be locally Lipschitz and definable,
and suppose that \(x^*\in\mathbb{R}^n\) satisfies
Assumption~\ref{assumption:active_aiming_verdier}.  Let \(c_b\geq0\),
and let \(\alpha_k=\Theta((k+1)^{-a})\), where
\(a\in(1/2,1]\).  Every bounded
sequence \((x_k)\) satisfying
\begin{equation}
 x_{k+1}=x_k-\alpha_ku_k,
 \qquad u_k\in\partial f(x_k;c_b\alpha_k),
\label{eq:active_goldstein_update}
\end{equation}
and having \(x^*\) as an accumulation point converges to \(x^*\).
\end{proposition}

\begin{proof}
\runinheading{Step~1: Function values, local setup, and polynomial tails}
Let \((x_k)\) satisfy the hypotheses of
Proposition~\ref{proposition:active_goldstein}.  Since \(1/2<a\leq1\), one
has \(\sum_k\alpha_k=\infty\), \(\sum_k\alpha_k^2<\infty\), and
\(c_b\alpha_k\to0\).  Moreover,
\(\sum_k\alpha_k\varepsilon_k=0\) when \(\varepsilon_k=0\).  Hence
Corollary~\ref{lemma:perturbed_subgradient_value_convergence}, applied with
\(r_k=c_b\alpha_k\) and \(\varepsilon_k=0\), gives
\begin{equation}
                         f(x_k)\to f(x^*).
\label{eq:active_goldstein_value_convergence}
\end{equation}

Let \(A\) be an active manifold at \(x^*\), and set
\(\underline\theta:=\max\{2/3,\allowbreak(1-a)/(3a-1)\}\).  Since
\(\underline\theta<1\), Lemma~\ref{lemma:active_local_consequences}, applied
with this value of \(\underline\theta\), fixes \(P_A,\mu,\varrho,L\), and an
exponent \(\theta\in(\underline\theta,1)\).
For \(k\in\mathbb N\), define the quadratic and cubic tails
\(E_k:=\sum_{i=k}^{\infty}\alpha_i^2\) and
\(H_k:=\sum_{i=k}^{\infty}\alpha_i^3\).
Since \(\alpha_i=\Theta((i+1)^{-a})\), comparison with the corresponding
\(p\)-series gives
\begin{equation}
\begin{aligned}
 &E_k=O((k+1)^{1-2a})\longrightarrow0,\\
 &H_k=O((k+1)^{1-3a})\longrightarrow0,\\
 &\sum_{i=k}^{\infty}\alpha_iH_i^\theta\longrightarrow0.
\end{aligned}
\label{eq:active_goldstein_tails}
\end{equation}
Indeed,
\(\alpha_iH_i^\theta=O((i+1)^{-a-\theta(3a-1)})\), and the choice of
\(\theta\) gives \(a+\theta(3a-1)>1\).

\runinheading{Step~2: Distance estimates}
Choose \(k_0\in\mathbb N\) so large that
\(B(B(x^*,\varrho),c_b\alpha_i)\subset
  \mathring B(x^*,4\varrho)\),
\(\alpha_i\leq\bar\alpha\), and \(\alpha_i\leq1\)
for all \(i\geq k_0\).  Fix \(q>p\geq k_0\) such that
\(x_{\llbracket p,q\rrbracket}\subset
  \mathring B(x^*,\varrho)\).  For \(p\leq i\leq q\), set
\(y_i:=P_A(x_i)\) and \(d_i:=|x_i-y_i|\).
Here and below, every unsubscripted \(C>0\) depends only on
\(a,c_b,\varrho,\theta,L,\mu,\bar\alpha,c_{\rm N}\) and the constants implicit
in \(\alpha_k=\Theta((k+1)^{-a})\), and is independent of \(i,p,q\); its value
may increase from line to line.
Apply Proposition~\ref{proposition:active_inexact_local_estimates} with
\(r_i=c_b\alpha_i\) and \(\varepsilon_i=0\).  Its distance estimate
\eqref{eq:active_meta_normal} gives, after increasing a fixed constant
\(C_{\rm N}>0\),
\begin{equation}
 d_{i+1}^2+\mu \alpha_id_i\leq d_i^2+C_{\rm N}\alpha_i^2.
\label{eq:active_goldstein_quadratic_normal}
\end{equation}
For any \(j\in\llbracket p,q\rrbracket\), summing
\eqref{eq:active_goldstein_quadratic_normal} from \(i=p\) to \(j-1\)
telescopes the squared-distance terms and gives
\[
 d_j^2+\mu\sum_{i=p}^{j-1}\alpha_id_i
 \leq d_p^2+C_{\rm N}\sum_{i=p}^{j-1}\alpha_i^2
 \leq d_p^2+C_{\rm N}E_p.
\]
Dropping the nonnegative sum, taking square roots, and using
\(\sqrt{a+b}\leq\sqrt a+\sqrt b\) gives the first estimate below.
Taking \(j=q\) and dropping \(d_q^2\) gives the second.  Thus,
\begin{equation}
\sup_{p\leq j\leq q}d_j
 \leq C(d_p+E_p^{1/2})\quad\text{and}\quad
 \sum_{i=p}^{q-1}\alpha_id_i
 \leq C(d_p^2+E_p).
 \label{eq:active_goldstein_quadratic_confinement}
\end{equation}
The cubic estimate
\eqref{eq:active_meta_cubic_normal} in
Proposition~\ref{proposition:active_inexact_local_estimates}, with the same
specialization, gives fixed constants \(c_{\rm N}>0\) and \(C>0\) such that
\begin{equation}
 d_{i+1}^3+c_{\rm N}\alpha_id_i^2\leq d_i^3+C\alpha_i^3.
\label{eq:active_goldstein_cubic_normal}
\end{equation}
Summing \eqref{eq:active_goldstein_cubic_normal} yields
\begin{equation}
 \sum_{i=p}^{q-1}\alpha_id_i^2
 \leq C(d_p^3+H_p).
\label{eq:active_goldstein_cubic_confinement}
\end{equation}

\runinheading{Step~3: Projected estimates}
The projected move estimate
\eqref{eq:active_meta_projected_move} in
Proposition~\ref{proposition:active_inexact_local_estimates}, again with
\(r_i=c_b\alpha_i\) and \(\varepsilon_i=0\), gives
\begin{equation}
 |y_{i+1}-y_i|
 \leq C\alpha_i(|\nabla_A f(y_i)|+d_i+\alpha_i).
\label{eq:active_goldstein_projected_move}
\end{equation}
The descent and gradient-variation estimates
\eqref{eq:active_meta_projected_descent} and
\eqref{eq:active_meta_gradient_variation} give, after enlarging a fixed
\(C_{\mathrm T}>0\),
\begin{align}
 f(y_{i+1})-f(y_i)
 &\leq-\frac12\alpha_i|\nabla_A f(y_i)|^2
       +C_{\mathrm T}(\alpha_id_i^2+\alpha_i^3).
\label{eq:active_goldstein_projected_descent}
\end{align}
\begin{equation}
 |\nabla_A f(y_{i+1})|\leq|\nabla_A f(y_i)|+C\alpha_i.
\label{eq:active_goldstein_gradient_variation}
\end{equation}

\runinheading{Step~4: Signed KL estimate and diameter bound}
Recall that
\(\varphi_\theta(t)=\operatorname{sgn}(t)|t|^{1-\theta}\).
For \(p\leq i\leq q\), set
\(G_i:=C_{\mathrm T}\sum_{\ell=i}^{q-1}
(\alpha_\ell d_\ell^2+\alpha_\ell^3)\).
Then \(G_q=0\), and
\eqref{eq:active_goldstein_projected_descent} and
\eqref{eq:active_goldstein_cubic_confinement} applied with left endpoint
\(i\) give
the following two estimates for every
\(i\in\llbracket p,q-1\rrbracket\):
\begin{align*}
 \bigl(f(y_i)-f(x^*)+G_i\bigr)
 -\bigl(f(y_{i+1})-f(x^*)+G_{i+1}\bigr)
 &\geq\frac12\alpha_i|\nabla_A f(y_i)|^2,\\
 0\leq G_{i+1}\leq G_i&\leq C(d_i^3+H_i).
\end{align*}
Since \(\theta>2/3\) and \(d_i<2\varrho\), it follows that
\[
 G_i^\theta+G_{i+1}^\theta
 \leq C(d_i^{3\theta}+H_i^\theta)
 \leq C(d_i^2+H_i^\theta).
\]
Apply Lemma~\ref{lemma:corrected_active_KL} with
\(F_i:=f(y_i)-f(x^*)\),
\(s_i:=|\nabla_A f(y_i)|\), \(c_i:=G_i\), and
\(\delta_i:=C\alpha_i\). The augmented descent inequality obtained from
\eqref{eq:active_goldstein_projected_descent}
verifies the descent hypothesis of Lemma~\ref{lemma:corrected_active_KL};
\eqref{eq:local_active_KL} verifies \(|F_i|^\theta\leq Ls_i\); and
\eqref{eq:active_goldstein_gradient_variation} verifies
\(s_{i+1}\leq s_i+\delta_i\).  The lemma gives
\[
\begin{aligned}
 \sum_{i=p}^{q-1}\alpha_i|\nabla_A f(y_i)|
 &\leq C_{\rm KL}
       \bigl(\varphi_\theta(f(y_p)-f(x^*))
             -\varphi_\theta(f(y_q)-f(x^*))\bigr)\\
 &\quad+CG_p^{1-\theta}
       +C\sum_{i=p}^{q-1}\alpha_i
          \bigl(\alpha_i+G_i^\theta+G_{i+1}^\theta\bigr).
\end{aligned}
\]
The inequalities \(G_p\leq C(d_p^3+H_p)\) and
\(G_i^\theta+G_{i+1}^\theta\leq C(d_i^2+H_i^\theta)\), together with
\eqref{eq:active_goldstein_cubic_confinement} and
\(\sum_{i=p}^{q-1}\alpha_i^2\leq E_p\), bound the last two terms on the
right-hand side by
\[
 C\left((d_p^3+H_p)^{1-\theta}
       +d_p^3+H_p+E_p
       +\sum_{i=p}^{\infty}\alpha_iH_i^\theta\right).
\]

Set
\[
 \mathfrak r_p^{\rm G}:=
 (d_p+E_p^{1/2})^{1-\theta}
 +\sum_{i=p}^{\infty}\alpha_iH_i^\theta.
\]
Since
\(H_p\leq(\sup_{i\geq p}\alpha_i)E_p\leq E_p^{3/2}\), one has
\(d_p^3+H_p\leq(d_p+E_p^{1/2})^3\).  Since
\(x_p\in\mathring B(x^*,\varrho)\) and \(x^*\in A\), one has
\(d_p\leq|x_p-x^*|<\varrho\); moreover, \(E_p\leq E_{k_0}\).  Thus
\(d_p+E_p^{1/2}\) is bounded independently of \(p\) and \(q\).  Because
\(0<1-\theta<1\),
\[
 (d_p^3+H_p)^{1-\theta}+d_p^3+H_p+d_p^2+E_p
 +d_p+E_p^{1/2}
 \leq C(d_p+E_p^{1/2})^{1-\theta}.
\]
Combining the estimate supplied by
Lemma~\ref{lemma:corrected_active_KL} with these bounds gives
\[
 \sum_{i=p}^{q-1}\alpha_i|\nabla_A f(y_i)|
 \leq C_{\rm KL}
      \bigl(\varphi_\theta(f(y_p)-f(x^*))
            -\varphi_\theta(f(y_q)-f(x^*))\bigr)
      +C\mathfrak r_p^{\rm G}.
\]

Choose a fixed \(C_{\rm P}\geq1\) for which
\eqref{eq:active_goldstein_projected_move} holds.  Summing that estimate and
then using \eqref{eq:active_goldstein_quadratic_confinement} and
\(\sum_{i=p}^{q-1}\alpha_i^2\leq E_p\) gives
\[
\begin{aligned}
 \sum_{i=p}^{q-1}|y_{i+1}-y_i|
 &\leq C_{\rm P}\sum_{i=p}^{q-1}
       \alpha_i|\nabla_A f(y_i)|
      +C_{\rm P}\sum_{i=p}^{q-1}\alpha_id_i+C_{\rm P}E_p\\
 &\leq C_{\rm P}C_{\rm KL}
       \bigl(\varphi_\theta(f(y_p)-f(x^*))
             -\varphi_\theta(f(y_q)-f(x^*))\bigr)
      +C\mathfrak r_p^{\rm G}.
\end{aligned}
\]
Set \(C_0:=C_{\rm P}C_{\rm KL}\).  Since
\[
 |x_m-x_\ell|
 \leq d_\ell+\sum_{i=\ell}^{m-1}|y_{i+1}-y_i|+d_m
 \qquad\forall p\leq\ell\leq m\leq q,
\]
the first bound in \eqref{eq:active_goldstein_quadratic_confinement} gives
\[
 d_\ell+d_m
 \leq2\sup_{p\leq j\leq q}d_j
 \leq C(d_p+E_p^{1/2})
 \leq C\mathfrak r_p^{\rm G}.
\]
Also,
\(\sum_{i=\ell}^{m-1}|y_{i+1}-y_i|
 \leq\sum_{i=p}^{q-1}|y_{i+1}-y_i|\).  Combining these estimates and
taking the supremum over \(p\leq\ell\leq m\leq q\) yields
\[
 \operatorname{diam}(x_{\llbracket p,q\rrbracket})
 \leq C_0\bigl(\varphi_\theta(f(y_p)-f(x^*))
                  -\varphi_\theta(f(y_q)-f(x^*))\bigr)
      +C\mathfrak r_p^{\rm G}.
\]
Finally, for \(i\in\{p,q\}\),
\eqref{eq:active_local_uniform_constants},
the signed-power H\"older estimate
\cite[Fact~4.12]{lai2025diameter}, and
\eqref{eq:active_goldstein_quadratic_confinement} give
\[
 |\varphi_\theta(f(x_i)-f(x^*))-\varphi_\theta(f(y_i)-f(x^*))|
 \leq Cd_i^{1-\theta}
 \leq C(d_p+E_p^{1/2})^{1-\theta}
 \leq C\mathfrak r_p^{\rm G}.
\]
Consequently,
\begin{equation}
 \operatorname{diam}(x_{\llbracket p,q\rrbracket})
 +C_0\bigl(\varphi_\theta(f(x_q)-f(x^*))
            -\varphi_\theta(f(x_p)-f(x^*))\bigr)
 \leq C\mathfrak r_p^{\rm G}.
\label{eq:active_goldstein_diameter}
\end{equation}

\runinheading{Step~5: Apply the diameter criterion}
Define the continuous potential
\(L_{\rm G}(x):=C_0\varphi_\theta(f(x)-f(x^*))\).
Equation~\eqref{eq:active_goldstein_value_convergence} gives
\(L_{\rm G}(x_k)\to0\).
Let \(K:=\overline{\{x_k:k\in\mathbb N\}}\), which is compact.  Since
\(c_b\alpha_k\to0\), one has
\(B(x_k,c_b\alpha_k)\subset B(K,1)\) for all sufficiently large \(k\).
Local boundedness of the Clarke subdifferential
\cite[Propositions~2.1.2 and~2.1.5]{clarke1990} gives a uniform bound on
\(\partial f(z)\) for \(z\in B(K,1)\).  The definition of
\(\partial f(x_k;c_b\alpha_k)\) therefore gives a uniform bound on \((u_k)\).
By \eqref{eq:active_goldstein_update},
\(|x_{k+1}-x_k|=\alpha_k|u_k|\to0\).

Take \(U:=\mathring B(x^*,\varrho)\).  If
\(x_{k_1}\in B(x^*,\delta)\) and
\(x_{\llbracket k_1,k_2\rrbracket}\subset U\), then
\(d_{k_1}\leq\delta\), and \eqref{eq:active_goldstein_diameter} gives
\[
 \operatorname{diam}(x_{\llbracket k_1,k_2\rrbracket})
 +L_{\rm G}(x_{k_2})-L_{\rm G}(x_{k_1})
 \leq C\left((\delta+E_{k_1}^{1/2})^{1-\theta}
       +\sum_{i=k_1}^{\infty}\alpha_iH_i^\theta\right).
\]
First let \(k_1,k_2\to\infty\), and then let \(\delta\downarrow0\).
The right-hand side tends to zero by \eqref{eq:active_goldstein_tails}, so
\eqref{eq:diam_cri} holds.  Theorem~\ref{thm:criterion} gives \(x_k\to x^*\).
\end{proof}

\begin{theorem}[Active momentum convergence]
\label{theorem:active_momentum}
Let \(f\colon\mathbb{R}^n\to\mathbb{R}\) be locally Lipschitz and definable,
and suppose that \(x^*\in\mathbb R^n\) satisfies
Assumption~\ref{assumption:active_aiming_verdier}.  Fix
\(\kappa\in[0,1)\) and \(\iota\in\mathbb R\), and consider
\eqref{eq:momentum-inexact} with \(x_{-1}=x_0\) and
\(\alpha_k=\Theta((k+1)^{-a})\), where \(a\in(1/2,1]\).
Every bounded sequence generated by \eqref{eq:momentum-inexact} and having
\(x^*\) as an accumulation point converges to \(x^*\).
\end{theorem}

\begin{proof}
Let \((x_k)\) satisfy the hypotheses.  Lemma~\ref{lemma:momentum_shadow}
gives a bounded nearby sequence \((\widehat x_k)\), with the same
accumulation points as \((x_k)\), such that
\[
 \widehat x_{k+1}=\widehat x_k-\widehat\alpha_kv_k,\qquad
 v_k\in\partial f(\widehat x_k;c_b\widehat\alpha_k),\qquad
 \widehat\alpha_k=\Theta((k+1)^{-a}).
\]
Proposition~\ref{proposition:active_goldstein} gives
\(\widehat x_k\to x^*\).  Since
\(|\widehat x_k-x_k|\to0\) by Lemma~\ref{lemma:momentum_shadow}, one also
has \(x_k\to x^*\).
\end{proof}

\subsection{Stochastic subgradient method}\label{sec:active_stochastic}

Theorem~\ref{theorem:1overk} shows the convergence of every bounded
realization of the stochastic subgradient method \eqref{eq:inexact stoc} for
step sizes of order \(1/k\), without prescribing an accumulation point.  When the
prescribed accumulation point satisfies
Assumption~\ref{assumption:active_aiming_verdier}, the next theorem covers
the polynomial step sizes of order $k^{-a}$ with \(2/3<a\leq1\).

\begin{theorem}[Active stochastic convergence]
\label{theorem:active_stochastic}
Let \(f\colon\mathbb{R}^n\to\mathbb{R}\) be locally Lipschitz and definable,
and suppose that \(x^*\in\mathbb{R}^n\) satisfies
Assumption~\ref{assumption:active_aiming_verdier}.  Consider the stochastic
subgradient method \eqref{eq:inexact stoc} under
Assumption~\ref{assumption:stochastic_noise}, with deterministic step sizes
\(\alpha_k=\Theta((k+1)^{-a})\), where \(a\in(2/3,1]\).  Then, almost surely,
every bounded realization \((x_k)\) having \(x^*\) as an accumulation point
converges to \(x^*\).
\end{theorem}

\begin{remark}[Smooth case]\label{remark:smooth_case}
If \(x^*\) is critical, \(f\) is definable and \(C^2\) near \(x^*\), then
\(x^*\) satisfies Assumption~\ref{assumption:active_aiming_verdier} with
\(A=\mathbb R^n\), and the update is standard stochastic gradient descent.
Indeed, for \(A=\mathbb R^n\), the restriction \(f|_A=f\) is \(C^2\) near
\(x^*\) and \(U_0\setminus A=\varnothing\), so
Assumption~\ref{assumption:active_aiming_verdier}(i) holds.  Moreover,
\(P_{T_A(y)}=I\), and \(P_{T_A(y)}\nabla f(z)-\nabla_A f(y)
=\nabla f(z)-\nabla f(y)\) satisfies
\eqref{eq:active_verdier} by the local Lipschitz continuity of \(\nabla f\).
Thus Theorem~\ref{theorem:active_stochastic} gives almost-sure convergence to
\(x^*\) on every bounded realization having \(x^*\) as an accumulation point when
\(2/3<a\leq1\).  Under their respective hypotheses, the same polynomial range
appears in \cite[Corollary~5.1]{qiu2024convergence},
\cite[Theorem~4.2 and Remark~4.3]{milzarek2023convergence}, and
\cite[Theorem~1.7]{dereich2024convergence}.
\end{remark}

\begin{remark}[Saddle nonaccumulation]\label{remark:saddle_nonaccumulation}
Recall that \(x^*\) is an \emph{active strict saddle} relative to an active
manifold \(A\) if \(0\in\partial f(x^*)\) and the Riemannian Hessian
\(\operatorname{Hess}(f|_A)(x^*)\) has a negative eigenvalue
\cite[Definition~2.7]{davis2019active}.  Let \(E^-\subset T_A(x^*)\) denote
the sum of its negative eigenspaces.

Recent works \cite{bianchi2023stochastic,davis2025active} show that the
stochastic subgradient iterates of \eqref{eq:inexact stoc} converge to an
active strict saddle with probability zero.  Their avoidance theorems impose additional smoothness
and regularity and require the noise to excite \(E^-\).  More precisely,
\cite[Assumption~3(iii)]{bianchi2023stochastic} requires
\[
 \liminf_{k\to\infty}
 \mathbb E[|P_{E^-}\varepsilon_k|\mid\mathcal F_k]>0
 \quad\text{a.s. on }\{x_k\to x^*\},
\]
together with the asymptotically bounded conditional fourth moment in
\cite[Assumption~3(ii)]{bianchi2023stochastic}.  Under their respective
regularity and noise hypotheses, both works prove
\(\mathbb P(x_k\to x^*)=0\); see
\cite[Definitions~4--6, Assumptions~1--3, and Theorem~3]{bianchi2023stochastic}
and
\cite[Definition~2.3, Assumptions~A, E, F, and Theorem~5.2]{davis2025active}.
These avoidance statements alone do not rule out \(x^*\) as an accumulation
point.

Suppose now that the fixed point \(x^*\) and the update satisfy the hypotheses
of Theorem~\ref{theorem:active_stochastic} and of either avoidance theorem.
Theorem~\ref{theorem:active_stochastic} shows that boundedness and accumulation
at \(x^*\) imply convergence to \(x^*\) almost surely.  Combining the two
results gives
\[
 \mathbb P\bigl((x_k)\text{ is bounded and }
 x^*\text{ is an accumulation point of }(x_k)\bigr)=0,
\]
for the common step-size regime \(\alpha_k=\Theta((k+1)^{-a})\) with \(2/3<a\leq1\); see \cite[Theorem~3]{bianchi2023stochastic} and \cite[Theorem~5.2]{davis2025active}.  The noise assumptions
are compatible: i.i.d. noise uniformly distributed on a centered Euclidean
ball, as in \cite[Assumption~F]{davis2025active}, satisfies both
Assumption~\ref{assumption:stochastic_noise} and the moment and excitation
conditions in \cite[Assumption~3]{bianchi2023stochastic}.
\end{remark}

We prove Theorem~\ref{theorem:active_stochastic} in the remainder of this
section.  The momentum analysis in Section~\ref{sec:active_momentum} reduces
the method to the enlarged-subgradient recurrence studied in
Proposition~\ref{proposition:active_goldstein}.  The stochastic proof uses the
same ingredients near the active manifold: the distance-to-manifold and projected estimates
from Proposition~\ref{proposition:active_inexact_local_estimates},
function-value convergence from
Corollary~\ref{lemma:perturbed_subgradient_value_convergence}, and the KL
estimate with correction terms from Lemma~\ref{lemma:corrected_active_KL}. It
remains to control the martingale and quadratic-noise errors on intervals
during which the iterates leave a neighborhood of \(x^*\).  The supporting
lemmas in Section~\ref{sec:active_stochastic_exit_estimates} provide this
control; the general martingale estimate used in their proofs is collected in
Appendix~\ref{sec:martingale_tail_estimates}.  We first record the geometric
and probabilistic setup, which remains fixed throughout the rest of this
section.

\runinheading{Fixed setup}
Since \(a>2/3\), one has
\(\max\{1/2,2(1-a)/a\}<1\).  Apply
Lemma~\ref{lemma:active_local_consequences} with
\(\underline\theta:=\max\{1/2,2(1-a)/a\}\), and fix the resulting active
manifold \(A\), metric projection \(P_A\), constants
\(\mu,\varrho,L>0\), and exponent
\(\theta\in(\underline\theta,1)\).  All conclusions of that lemma hold, and
\begin{equation}
 2\theta>1,\qquad
 a\left(1+\frac\theta2\right)>1,\qquad
 a+\theta(2a-1)>1.
\label{eq:stoch_theta_choice}
\end{equation}
Indeed, the first two inequalities follow from the choice of \(\theta\),
while the third follows from
\((1-a)/(2a-1)\leq2(1-a)/a<\theta\). The polynomial step-size bounds give
\begin{equation}
 \alpha_k\to0,\qquad
 \sum_{k=0}^{\infty}\alpha_k=\infty,\qquad
 \sum_{k=0}^{\infty}\alpha_k^2<\infty.
\label{eq:stoch_step_consequences}
\end{equation}
For each coordinate, the martingale increments
\(\alpha_k\varepsilon_k\) have summable conditional variances by
\eqref{eq:stoch_step_consequences} and
Assumption~\ref{assumption:stochastic_noise}.  Hence
\cite[Theorem~2.15]{hall1980martingale} gives, almost surely,
\begin{equation}
 \sum_{k=0}^{\infty}\alpha_k\varepsilon_k
 \quad\text{converges}
\label{eq:stoch_vector_martingale_convergence}
\end{equation}
coordinatewise.  Moreover, the tower property and Tonelli's theorem
\cite[Theorem~2.37]{folland1999real} give
\(\mathbb E[\sum_k\alpha_k^2(1+|\varepsilon_k|^2)]
\leq(1+\sigma^2)\sum_k\alpha_k^2<\infty\).  Consequently,
\begin{equation}
 \sum_{k=0}^{\infty}\alpha_k^2(1+|\varepsilon_k|^2)<\infty.
\label{eq:stoch_quadratic_noise_summability}
\end{equation}
Let \(\Omega_{\mathrm{act}}\) be a probability-one event on which
\eqref{eq:stoch_vector_martingale_convergence} and
\eqref{eq:stoch_quadratic_noise_summability} hold.  On this event,
convergence of the series in
\eqref{eq:stoch_vector_martingale_convergence} gives
\(\alpha_k\varepsilon_k\to0\).  Corollary~\ref{lemma:perturbed_subgradient_value_convergence},
applied with \(r_k=0\), further shows that every bounded realization having
\(x^*\) as an accumulation point satisfies
\begin{equation}
 f(x_k)\longrightarrow f(x^*).
\label{eq:stoch_value_convergence}
\end{equation}
Henceforth, these objects and the event \(\Omega_{\mathrm{act}}\) are fixed.

\subsubsection{Exit intervals and stochastic estimates}
\label{sec:active_stochastic_exit_estimates}

We follow the iterates from a close return to \(x^*\) until their next exit
from a fixed neighborhood.  On the resulting intervals, the proof controls
motion toward the active manifold, motion of the projected iterates along the
manifold, and the accumulated stochastic errors.  We first select the
intervals so that the relevant errors are summable, and then derive the
distance-to-manifold and projected estimates.

\runinheading{Entrance and exit indices}
Let \(\rho\in(0,\varrho/2]\) be arbitrary.  For the moment, let
\((K_j)_{j\geq1}\) be any deterministic strictly increasing sequence in
\(\mathbb N\), and set
\(q_0:=-1\).  The entrance index \(p_j\) is the first visit, after both
\(K_j\) and the preceding exit, to a ball shrinking geometrically toward
\(x^*\); the index \(q_j\) is the next exit from
\(\mathring B(x^*,\rho)\).  Thus, we define
\(p_j,q_j\colon\Omega\to\mathbb N\cup\{\infty\}\) by
\begin{equation}
 \begin{aligned}
 p_j&:=\inf\{k\geq K_j:k>q_{j-1},\
                    x_k\in B(x^*,2^{-j-1}\rho)\},\\
 q_j&:=\inf\{k>p_j:x_k\notin\mathring B(x^*,\rho)\},
 \end{aligned}
 \qquad\forall j\geq1.
 \label{eq:stoch_excursions}
\end{equation}

The sequence \((K_j)\) will be chosen deterministically so that the
stochastic and function-value tails associated with the resulting intervals
are summable.

Define the exit-interval indicator by
\begin{equation}
 \eta_k:=
 \mathbf1\{\exists j\geq1:\ p_j\leq k<q_j\},
 \qquad k\in\mathbb N.
\label{eq:stoch_excursion_indicator}
\end{equation}
If \(\eta_k=1\), then
\(x_k\in\mathring B(x^*,\rho)\), so \(P_A(x_k)\) is defined.  Set
\begin{equation}
 n_k:=\eta_k\bigl(x_k-P_A(x_k)\bigr).
\label{eq:stoch_normal_coefficient}
\end{equation}
Here and below, an indicator-weighted local expression is evaluated only on
\(\{\eta_k=1\}\) and is defined to be zero on \(\{\eta_k=0\}\).
In every sum \(\sum_{k=r}^{s-1}\) below, \(r\leq s\) are integer indices,
\(s\) is the exclusive upper endpoint, and the sum is zero when \(r=s\).
Since \((x_k)\) is adapted, a direct induction shows that \(p_j,q_j\) are
stopping times for every deterministic strictly increasing sequence
\((K_j)\).  Consequently, \(\eta_k\) and \(n_k\) are
\(\mathcal F_k\)-measurable, with \(|n_k|\leq\rho\).

The three tails in the next lemma serve two purposes.  Summing the distance
estimate \eqref{eq:active_meta_normal} produces the scalar martingale blocks
\(\sum_k\alpha_k\langle\varepsilon_k,n_k\rangle\) and the quadratic-noise
remainder \(\sum_k\alpha_k^2(1+|\varepsilon_k|^2)\).  Since that estimate
controls squared distance to \(A\), the square roots of these two errors must
be summable to control the distances to \(A\) entering the interval-diameter
estimate.
Summing the projected descent estimate
\eqref{eq:active_meta_projected_descent} produces endpoint function values,
which are controlled by the tail of \(|f(x_k)-f(x^*)|\), together with the
distances to \(A\).  The next lemma selects one sequence \((K_j)\) for which
all three contributions are summable.

\begin{lemma}[Summable exit intervals]
\label{lemma:active_stochastic_sparse}
There is a deterministic strictly increasing sequence \((K_j)\) such that
the following holds for every fixed \(\rho\in(0,\varrho/2]\): almost surely,
every bounded realization having \(x^*\) as an accumulation point satisfies
\begin{equation}
 \sum_{j=1}^{\infty}\Bigg[
 \left(\sup_{s\geq r\geq K_j}\left|
     \sum_{k=r}^{s-1}
       \alpha_k\langle\varepsilon_k,n_k\rangle
   \right|\right)^{1/2}+\left(
   \sum_{k=K_j}^{\infty}\alpha_k^2(1+|\varepsilon_k|^2)
   \right)^{1/2}
   +\sup_{k\geq K_j}|f(x_k)-f(x^*)|\Bigg]<\infty.
\label{eq:stoch_sparse_tails}
\end{equation}
\end{lemma}

The proof of Lemma~\ref{lemma:active_stochastic_sparse} is given in Appendix~\ref{sec:proof_active_stochastic_sparse}.
Fix a sequence \((K_j)\) supplied by
Lemma~\ref{lemma:active_stochastic_sparse}.  Henceforth,
\(p_j,q_j,\eta_k,n_k\) are formed using this sequence.

Whenever \(p_j,q_j<\infty\), their definition gives
\(x_{\llbracket p_j,q_j-1\rrbracket}\subset
\mathring B(x^*,\rho)\).  The exit iterate \(x_{q_j}\), however, may
overshoot the larger local neighborhood.  The following fact shows that this
can occur only finitely often on the realizations considered below.

\begin{fact}[Exit localization]
\label{fact:active_stochastic_exit_localization}
On \(\Omega_{\mathrm{act}}\), every bounded realization having \(x^*\) as an
accumulation point and leaving \(\mathring B(x^*,\rho)\) infinitely often has
all \(p_j,q_j\) finite and admits a finite \(j_0=j_0(\omega,\rho)\) such that
\begin{equation}
 x_{q_j}\in\mathring B(x^*,\varrho),
 \qquad\forall j\geq j_0.
 \label{eq:stoch_exit_endpoint_localization}
\end{equation}
For any such \(j_0\), we call
\(\llbracket p_j,q_j\rrbracket\), \(j\geq j_0\), the \emph{late exit
intervals}; their update indices are
\(\llbracket p_j,q_j-1\rrbracket\).
\end{fact}

Indeed, because \(x^*\) is an accumulation point and the realization leaves
\(\mathring B(x^*,\rho)\) infinitely often, induction shows that all
\(p_j,q_j\) are finite, while
\(x_{q_j-1}\in\mathring B(x^*,\rho)\),
\(q_j-1\geq p_j\geq K_j\to\infty\),
\eqref{eq:stoch_step_consequences},
\eqref{eq:stoch_vector_martingale_convergence},
\eqref{eq:active_local_uniform_constants}, and the update
\eqref{eq:inexact stoc} give
\( |x_{q_j}-x_{q_j-1}|
 \leq L\alpha_{q_j-1}+|\alpha_{q_j-1}\varepsilon_{q_j-1}|\to0\),
so eventually this increment is smaller than \(\varrho-\rho\), which yields
\eqref{eq:stoch_exit_endpoint_localization}.

On each late exit interval \(\llbracket p_j,q_j\rrbracket\), set
\(y_k:=P_A(x_k)\) and \(d_k:=|x_k-y_k|\), so \(y_k\in A\) and
\(d_k=d(x_k,A)\) for \(p_j\leq k\leq q_j\).  The next lemma controls these
distances and their weighted sum over the late exit intervals.

\begin{lemma}[Exit-distance bounds]
\label{lemma:active_stochastic_normal}
There is \(C>0\) such that the following holds almost surely.  Consider any
bounded realization having \(x^*\) as an accumulation point and leaving
\(\mathring B(x^*,\rho)\) infinitely often.  Let \(j_0\) and the late exit
intervals be as in
Fact~\ref{fact:active_stochastic_exit_localization}.  After increasing this
\(j_0\), if necessary, one has, for every \(j\geq j_0\),
\[
 \sup_{p_j\leq s\leq q_j}d_s^2
 +\mu\sum_{k=p_j}^{q_j-1}\alpha_kd_k
 \leq C\left(
 4^{-j-1}\rho^2
 +\sup_{s\geq r\geq K_j}\left|
  \sum_{k=r}^{s-1}
   \alpha_k\langle\varepsilon_k,n_k\rangle
 \right|
 +\sum_{k=K_j}^{\infty}\alpha_k^2(1+|\varepsilon_k|^2)
 \right),
\]
and consequently
\begin{equation}
 \sum_{j=j_0}^{\infty}\sup_{p_j\leq k\leq q_j}d_k<\infty,
 \qquad
 \sum_{j=j_0}^{\infty}\sum_{k=p_j}^{q_j-1}\alpha_kd_k<\infty.
\label{eq:stoch_normal_summability}
\end{equation}
\end{lemma}

The proof is given in Appendix~\ref{sec:proof_active_stochastic_normal}.
Related stopped-neighborhood distance estimates appear in
\cite[Proposition~5.1]{davis2025active} and
\cite[Lemmas~4--5 and Proposition~5]{bianchi2023stochastic};
Lemma~\ref{lemma:active_stochastic_normal} additionally gives pathwise
summability over the disjoint exit intervals.

To state the corresponding projected estimates, use the same exit-interval
indicator and set
\begin{equation}
 t_k:=\eta_kDP_A(x_k)^*\nabla_A f(P_A(x_k)).
\label{eq:stoch_tangential_coefficient}
\end{equation}
As with \(n_k\), this expression is evaluated on \(\{\eta_k=1\}\) and is
set to zero otherwise.

To measure the largest changes, after index \(i\), in the martingale partial
sums with coefficients \(n_k\) and \(t_k\), set
\begin{equation}
 \mathcal T_i:=\sup_{s\geq r\geq i}\left|
       \sum_{k=r}^{s-1}\alpha_k\langle\varepsilon_k,n_k\rangle
      \right|
 +\sup_{s\geq r\geq i}\left|
       \sum_{k=r}^{s-1}\alpha_k\langle\varepsilon_k,t_k\rangle
      \right|\in[0,\infty].
\label{eq:stoch_combined_martingale_tail}
\end{equation}
Thus \(\mathcal T_i\) adds the largest absolute block sum after \(i\) from
each scalar martingale, and
\(\mathcal T_{i+1}\leq\mathcal T_i\).  For an index \(i\) in a late exit
interval, combine \(\mathcal T_i\) with the other errors in the distance
recursion \eqref{eq:active_meta_normal} by setting
\begin{equation}
 Q_i:=d_i^2+
 \sum_{k=i}^{\infty}\alpha_k^2(1+|\varepsilon_k|^2)+\mathcal T_i
 \in[0,\infty].
\label{eq:stoch_Q_definition}
\end{equation}
The quantity \(Q_i\) will control the correction needed before applying
Lemma~\ref{lemma:corrected_active_KL}.  The next lemma records the corresponding
projected-gradient and stochastic-tail estimates.

\begin{lemma}[Projected exit estimates]
\label{lemma:active_stochastic_budget}
There is a fixed \(C_{\rm T}>0\) such that the following holds almost surely.
Consider any bounded realization having \(x^*\) as an accumulation point and
leaving \(\mathring B(x^*,\rho)\) infinitely often.  Let \(j_0\) and the late
exit intervals be as in
Fact~\ref{fact:active_stochastic_exit_localization}.  After increasing this
\(j_0\), if necessary, the estimate below holds for every \(j\geq j_0\) and
\(p_j\leq k<q_j\):
\begin{align}
 f(y_{k+1})-f(y_k)
 &\leq-\frac12\alpha_k|\nabla_A f(y_k)|^2
      -\alpha_k\langle\varepsilon_k,t_k\rangle
      +C_{\rm T}\alpha_kd_k
      +C_{\rm T}\alpha_k^2(1+|\varepsilon_k|^2).
 \label{eq:stoch_projected_descent}
\end{align}
Moreover, the realization has the following properties:
\begin{enumerate}[label=\emph{(\roman*)},leftmargin=*]
\item \emph{Summability of gradients and coefficients.}
One has
\[
 \sum_{j=j_0}^{\infty}\sum_{k=p_j}^{q_j-1}
       \alpha_k|\nabla_A f(y_k)|^2<\infty,
 \qquad
 \sum_k\alpha_k|n_k|^2<\infty,
 \qquad
 \sum_k\alpha_k|t_k|^2<\infty.
\]

\item \emph{Martingale oscillations.}
The quantity \(\mathcal T_i\) in
\eqref{eq:stoch_combined_martingale_tail} satisfies
\(\mathcal T_i\to0\) and
\(\sum_i\alpha_i\mathcal T_i^\theta<\infty\).

\item \emph{Combined stochastic errors.}
The quantity \(Q_i\) in \eqref{eq:stoch_Q_definition} satisfies
\begin{equation}
 Q_{p_j}\to0,
 \qquad
 \sum_{j=j_0}^{\infty}\sum_{i=p_j}^{q_j-1}
 \alpha_i(Q_i^\theta+Q_{i+1}^\theta)<\infty.
\label{eq:stoch_Q_summability}
\end{equation}
\end{enumerate}
\end{lemma}

The proof is given in Appendix~\ref{sec:proof_active_stochastic_budget}.

\subsubsection{Proof of Theorem~\ref{theorem:active_stochastic}}
\label{sec:proof_active_stochastic}

\begin{proof}
\runinheading{Step~1: Fix the probability-one event and exit intervals}
Choose deterministic radii \(\rho_m:=\varrho/(m+1)\), \(m\geq1\).  For each
fixed \(m\), use the deterministic sequence \((K_j)\) supplied by
Lemma~\ref{lemma:active_stochastic_sparse} and apply
\eqref{eq:stoch_excursions}, \eqref{eq:stoch_excursion_indicator},
\eqref{eq:stoch_normal_coefficient}, and
\eqref{eq:stoch_tangential_coefficient} with
\(\rho=\rho_m\), adding a superscript \((m)\) to the resulting quantities.
Set
\(\Pi_k^{(m)}:=\eta_k^{(m)}DP_A(x_k)\), evaluating the product on
\(\{\eta_k^{(m)}=1\}\) and setting it to zero otherwise.
The conclusions of
Lemmas~\ref{lemma:active_stochastic_sparse},
\ref{lemma:active_stochastic_normal}, and
\ref{lemma:active_stochastic_budget} hold almost surely.  Moreover,
since \(p_j^{(m)}\) and \(q_j^{(m)}\) are stopping times,
\eqref{eq:stoch_excursion_indicator} and
\eqref{eq:active_local_uniform_constants} make \(\Pi_k^{(m)}\) uniformly
bounded and \(\mathcal F_k\)-measurable.  For each coordinate, the
corresponding martingale increments have summable conditional variances by
Assumption~\ref{assumption:stochastic_noise} and
\eqref{eq:stoch_step_consequences}.  Hence
\cite[Theorem~2.15]{hall1980martingale} gives almost-sure convergence of
\(\sum_{i=0}^{\infty}\alpha_i\Pi_i^{(m)}\varepsilon_i\).
Intersecting these events over \(m\geq1\) with \(\Omega_{\mathrm{act}}\)
gives a single probability-one event on which all the conclusions hold for
every \(\rho_m\).

Fix a bounded realization in this event having \(x^*\) as an accumulation
point.  Fix \(m\geq1\), set \(\rho:=\rho_m\), and suppress the superscript
\((m)\) on the corresponding interval quantities.  Suppose that the
realization leaves \(\mathring B(x^*,\rho)\) infinitely often.
Let \(j_0\) and the late exit intervals be as in
Fact~\ref{fact:active_stochastic_exit_localization}.

\runinheading{Step~2: Descent after adding the correction}
Increase this \(j_0\), if necessary, so that \(Q_{p_j}\leq1\), as permitted
by \eqref{eq:stoch_Q_summability}, and \(\alpha_k\leq\bar\alpha\) on every
late exit interval for \(j\geq j_0\).  Recall that
\(y_i=P_A(x_i)\) and \(d_i=|x_i-y_i|\) there.
For \(j\geq j_0\), write \(p=p_j\) and \(q=q_j\).  For
\(p\leq i\leq q\), define the nonnegative scalar correction
\begin{equation}
 b_i^{p,q}:=
 C_{\rm T}\sum_{k=i}^{q-1}
       \bigl(\alpha_kd_k+\alpha_k^2(1+|\varepsilon_k|^2)\bigr)
 +\max_{i\leq s\leq q}
       \left\{-\sum_{k=i}^{s-1}
                  \alpha_k\langle\varepsilon_k,t_k\rangle\right\}
 \in\mathbb R_+.
 \label{eq:stoch_correction}
\end{equation}
The correction vanishes at \(i=q\).  For \(p\leq i<q\), compare the maximum
in \eqref{eq:stoch_correction} for the consecutive starting indices \(i\)
and \(i+1\):
\[
\max_{i\leq s\leq q}
       \left\{-\sum_{k=i}^{s-1}
          \alpha_k\langle\varepsilon_k,t_k\rangle\right\}
 -\max_{i+1\leq s\leq q}
       \left\{-\sum_{k=i+1}^{s-1}
          \alpha_k\langle\varepsilon_k,t_k\rangle\right\}\geq-\alpha_i\langle\varepsilon_i,t_i\rangle.
\]
Combining this inequality with \eqref{eq:stoch_projected_descent} gives
\begin{equation}
 \bigl(f(y_i)-f(x^*)+b_i^{p,q}\bigr)
 -\bigl(f(y_{i+1})-f(x^*)+b_{i+1}^{p,q}\bigr)
 \geq\frac12\alpha_i|\nabla_A f(y_i)|^2.
\label{eq:stoch_corrected_descent}
\end{equation}
The definition \eqref{eq:stoch_excursions} of \(q_j\) and
\eqref{eq:stoch_exit_endpoint_localization} verify the local-segment
hypotheses.  Applying the distance estimate \eqref{eq:active_meta_normal} in
Proposition~\ref{proposition:active_inexact_local_estimates} with
\(u_k=v_k\) and \(r_k=0\), summing from \(i\) to \(q-1\), and using the
first martingale oscillation in
\eqref{eq:stoch_combined_martingale_tail} gives
\[
 \mu\sum_{k=i}^{q-1}\alpha_kd_k
 \leq d_i^2+2\mathcal T_i
      +C\sum_{k=i}^{q-1}\alpha_k^2(1+|\varepsilon_k|^2)
 \leq CQ_i.
\]
The running maximum in \eqref{eq:stoch_correction} is at most
\(\mathcal T_i\), while
\(\sum_{k=i}^{q-1}\alpha_k^2(1+|\varepsilon_k|^2)\leq Q_i\) by
\eqref{eq:stoch_Q_definition}.
Therefore
\begin{equation}
 0\leq b_i^{p,q}\leq CQ_i,
 \qquad\forall i\in\llbracket p,q\rrbracket.
 \label{eq:stoch_correction_by_Q}
\end{equation}

\runinheading{Step~3: KL length estimate}
Recall that
\(\varphi_\theta(t)=\operatorname{sgn}(t)|t|^{1-\theta}\).
For \(p\leq i<q\), the gradient-variation estimate
\eqref{eq:active_meta_gradient_variation} in
Proposition~\ref{proposition:active_inexact_local_estimates}, applied with
\(u_i=v_i\) and \(r_i=0\), gives
\begin{equation}
 |\nabla_A f(y_{i+1})|
 \leq|\nabla_A f(y_i)|+C\alpha_i(1+|\varepsilon_i|).
 \label{eq:stoch_gradient_variation}
\end{equation}
Apply Lemma~\ref{lemma:corrected_active_KL} with
\[
 F_i:=f(y_i)-f(x^*),\qquad
 s_i:=|\nabla_A f(y_i)|,\qquad
 c_i:=b_i^{p,q},\qquad
 \delta_i:=C\alpha_i(1+|\varepsilon_i|).
\]
The augmented descent estimate \eqref{eq:stoch_corrected_descent}, the local
KL inequality
\eqref{eq:local_active_KL}, and
\eqref{eq:stoch_gradient_variation} verify its hypotheses.  By
\eqref{eq:stoch_correction_by_Q}, its correction terms satisfy
\[
 (b_p^{p,q})^{1-\theta}\leq CQ_p^{1-\theta},
 \qquad
 (b_i^{p,q})^\theta+(b_{i+1}^{p,q})^\theta
 \leq C(Q_i^\theta+Q_{i+1}^\theta).
\]
The sum underlying the
\(\delta_i=C\alpha_i(1+|\varepsilon_i|)\) contribution in
Lemma~\ref{lemma:corrected_active_KL} satisfies
\[
 \sum_{i=p}^{q-1}\alpha_i^2(1+|\varepsilon_i|)
 \leq2\sum_{i=p}^{\infty}\alpha_i^2(1+|\varepsilon_i|^2)
 \leq2Q_p.
\]
Since \(Q_p\leq1\), one has \(Q_p\leq Q_p^{1-\theta}\).  Thus there exist fixed constants
\(C_{\rm KL}^{\rm S}>0\) and \(C>0\), independent of \(j,p,q\), such that
\begin{equation}
 \sum_{i=p}^{q-1}\alpha_i|\nabla_A f(y_i)|
 \leq C_{\rm KL}^{\rm S}
       \bigl(\varphi_\theta(f(y_p)-f(x^*))
             -\varphi_\theta(f(y_q)-f(x^*))\bigr)
 +C\left(
  Q_p^{1-\theta}
  +\sum_{i=p}^{q-1}\alpha_i(Q_i^\theta+Q_{i+1}^\theta)
 \right).
 \label{eq:stoch_excursion_KL}
\end{equation}

\runinheading{Step~4: Projected diameter estimate}
By construction of the probability-one event in Step~1, the series
\(\sum_{i=0}^{\infty}\alpha_i\Pi_i\varepsilon_i\)
converges almost surely, and hence its tails converge to zero.  Applying the
projected recursion \eqref{eq:active_meta_projected_recursion} from
Proposition~\ref{proposition:active_inexact_local_estimates} with
\(u_k=v_k\) and \(r_k=0\) gives, for \(p_j\leq k<q_j\),
\begin{equation}
 |y_{k+1}-y_k+\alpha_k\nabla_A f(y_k)
       +\alpha_kDP_A(x_k)\varepsilon_k|
 \leq C\alpha_kd_k+C\alpha_k^2(1+|\varepsilon_k|^2).
 \label{eq:stoch_projected_recursion}
\end{equation}
For \(p_j\leq k<q_j\) on a late exit interval,
\(\Pi_k=DP_A(x_k)\).  Summing
\eqref{eq:stoch_projected_recursion} therefore gives, for any
\(p_j\leq u<v\leq q_j\),
\begin{align*}
 |x_v-x_u|
 &\leq d_u+d_v+|y_v-y_u|\\
 &\leq2\sup_{p_j\leq k\leq q_j}d_k
       +2\sup_{r\geq p_j}\left|
          \sum_{i=r}^{\infty}
          \alpha_i\Pi_i\varepsilon_i
        \right|
       +\sum_{k=p_j}^{q_j-1}\alpha_k|\nabla_A f(y_k)|+C\sum_{k=p_j}^{q_j-1}\alpha_kd_k+CQ_{p_j}.
\end{align*}
Taking the supremum over \(u,v\) gives
\begin{equation}
\begin{aligned}
 \operatorname{diam}(x_{\llbracket p_j,q_j\rrbracket})
 &\leq\sum_{i=p_j}^{q_j-1}\alpha_i|\nabla_A f(y_i)|
 +C\sup_{p_j\leq k\leq q_j}d_k
 +C\sup_{r\geq p_j}\left|
      \sum_{i=r}^{\infty}
      \alpha_i\Pi_i\varepsilon_i
    \right|+C\sum_{k=p_j}^{q_j-1}\alpha_kd_k+CQ_{p_j}.
\end{aligned}
\label{eq:stoch_excursion_diameter}
\end{equation}

\runinheading{Step~5: Vanishing remainder and local diameter estimate}
On these intervals, \(x_k\in\mathring B(x^*,\varrho)\subset
\mathring B(x^*,2\varrho)\), and the projection containment in
Lemma~\ref{lemma:active_local_consequences}(i) gives
\(y_k=P_A(x_k)\in A\cap B(x^*,8\varrho)\).  The Lipschitz bound for \(f\) in
\eqref{eq:active_local_uniform_constants} and
the signed-power H\"older estimate
\cite[Fact~4.12]{lai2025diameter} give
\begin{equation}
 |\varphi_\theta(f(x_k)-f(x^*))-\varphi_\theta(f(y_k)-f(x^*))|
 \leq Cd_k^{1-\theta}
 \qquad\forall k\in\llbracket p_j,q_j\rrbracket.
 \label{eq:stoch_endpoint_projection_error}
\end{equation}
After increasing \(j_0\) once more so that
\(\sup_{p_j\leq k\leq q_j}d_k\leq1\), define the nonnegative combined
exit-interval remainder \(\mathfrak r_j^{\rm S}\in\mathbb R_+\) by
\begin{align*}
 \mathfrak r_j^{\rm S}&:=Q_{p_j}^{1-\theta}
 +\sum_{i=p_j}^{q_j-1}
      \alpha_i(Q_i^\theta+Q_{i+1}^\theta)+\left(\sup_{p_j\leq k\leq q_j}d_k\right)^{1-\theta}
 +\sup_{r\geq p_j}\left|
      \sum_{i=r}^{\infty}
      \alpha_i\Pi_i\varepsilon_i
    \right|
 +\sum_{k=p_j}^{q_j-1}\alpha_kd_k.
\end{align*}
The two terms involving \(Q_i\) vanish by
\eqref{eq:stoch_Q_summability}, and the supremum-distance term vanishes by
\eqref{eq:stoch_normal_summability}.
Convergence of \(\sum_i\alpha_i\Pi_i\varepsilon_i\) gives
\[
 \sup_{r\geq p_j}\left|\sum_{i=r}^{\infty}
 \alpha_i\Pi_i\varepsilon_i\right|\longrightarrow0.
\]
Finally, \eqref{eq:stoch_normal_summability} gives
\(\sum_{k=p_j}^{q_j-1}\alpha_kd_k\to0\).  Hence
\(\mathfrak r_j^{\rm S}\to0\).  Since
\(Q_{p_j}\leq Q_{p_j}^{1-\theta}\), equations
\eqref{eq:stoch_excursion_KL},
\eqref{eq:stoch_excursion_diameter}, and
\eqref{eq:stoch_endpoint_projection_error} give
\begin{equation}
 \operatorname{diam}(x_{\llbracket p_j,q_j\rrbracket})
 +C_{\rm KL}^{\rm S}
   \bigl(\varphi_\theta(f(x_{q_j})-f(x^*))
         -\varphi_\theta(f(x_{p_j})-f(x^*))\bigr)
 \leq C\mathfrak r_j^{\rm S}.
 \label{eq:stoch_excursion_final}
\end{equation}

Since \(p_j\geq K_j\to\infty\),
\eqref{eq:stoch_value_convergence},
\eqref{eq:stoch_excursion_final}, and
\(\mathfrak r_j^{\rm S}\to0\) give
\(\operatorname{diam}(x_{\llbracket p_j,q_j\rrbracket})\to0\).  This is
impossible because
\[
 |x_{q_j}-x_{p_j}|
 \geq \rho-2^{-j-1}\rho.
\]
This contradiction shows that the realization eventually remains in
\(\mathring B(x^*,\rho_m)\).  Since this holds for every \(m\geq1\) and
\(\rho_m\downarrow0\), we conclude that \(x_k\to x^*\).
\end{proof}

\tocgroup{Appendices}

\appendix

\section[Technical overlaps with and differences from our previous work]{Technical overlaps with and differences from \cite{lai2025diameter}}
\label{sec:technical_comparison}
The next two appendices separate the geometric diameter argument in
Appendix~\ref{sec:proof_diameter} from its verification for the inexact
subgradient method in Appendix~\ref{sec:proof_inexact}. This section compares
their proof mechanisms with the corresponding results in our previous work
\cite{lai2025diameter}.

\subsection{The abstract diameter mechanism}
Appendix~\ref{sec:proof_diameter} proves the abstract diameter theorem,
Theorem~\ref{thm:diameter}. Fact~\ref{fact:y_k^i} supplies the local
projected-length input corresponding to
\cite[Corollary~4.2(3)]{lai2025diameter}; the remaining results correspond to
the statements listed below from
\cite[Sections~4.2--4.4]{lai2025diameter}.
\begin{center}
\small
\begin{tabular}{@{}p{0.66\textwidth}l@{}}
\toprule
Result in Appendix~\ref{sec:proof_diameter}
& Corresponding result in \cite{lai2025diameter}\\
\midrule
Fact~\ref{fact:y_k^i} (shifted projection estimate)
& Corollary~4.2(3)\\
Fact~\ref{fact:IC_close} (crossing detection)
& Fact~4.3\\
Fact~\ref{fact:IC_away} (open-stratum persistence)
& Fact~4.4\\
Fact~\ref{fact:no_repeat} (no stratum return)
& Fact~4.6\\
Definition~\ref{def:RL} (relative length) & Definition~4.7\\
Identity~\eqref{eq:RL_additivity} (exact additivity and interval monotonicity)
& Remark~4.8\\
Lemma~\ref{lemma:RM_within_one} (one-block displacement)
& Lemma~4.9\\
Lemma~\ref{lemma:RM_multi} (multiblock displacement)
& Lemma~4.10\\
Lemma~\ref{lemma:diameter_by_RM} (relative-length diameter bound)
& Lemma~4.11\\
Lemma~\ref{lemma:endpoint_conversion} (endpoint conversion) & Lemma~4.13\\
Lemma~\ref{lemma:core_bridge} (one-block length bound)
& Lemma~4.14\\
Corollary~\ref{cor:lm->lm+1} (consecutive crossings)
& Corollary~4.15\\
Lemma~\ref{lemma:RM_or_half} (length-scale alternative)
& Lemma~4.16\\
Lemma~\ref{lemma:sum_alpha_p} (crossing-error bound) & Lemma~4.17\\
Lemma~\ref{lemma:RM_arbitrary_interval} (crossing length)
& Lemma~4.18\\
Lemma~\ref{lemma:RM_terminal_piece} (terminal length bound)
& Lemma~4.19\\
\bottomrule
\end{tabular}
\end{center}
The open-stratum bridge used in Lemma~\ref{lemma:core_bridge} follows here
from Fact~\ref{fact:IC_away} and corresponds to
\cite[Fact~4.5]{lai2025diameter}. Lemma~\ref{lemma:endpoint_conversion} uses
the signed-power estimate in \cite[Fact~4.12]{lai2025diameter}.

Compared with the subgradient diameter estimate in
\cite[Theorem~1.1]{lai2025diameter}, Theorem~\ref{thm:diameter} takes as an
input the blockwise projected-length inequality \eqref{eq:y_k^i} in
item~\ref{item:stratified_descent} of
Definition~\ref{def:stratified_descent}; Appendix~\ref{sec:proof_diameter}
uses its shifted form in Fact~\ref{fact:y_k^i}. Item~\ref{item:g} of the same
definition retains the error functions \(g_1,g_2\). Theorem~\ref{thm:diameter} also replaces
monotonicity of the step sizes by a uniform ratio bound. Accordingly, the
crossing radii in \cite[(4.26)]{lai2025diameter} become the enlarged radii
\(\widehat c_i\) in \eqref{eq:hat_c} and \eqref{eq:def_IC}. The same ratio
bound changes the alternative in \cite[Lemma~4.16]{lai2025diameter} to that in
Lemma~\ref{lemma:RM_or_half}: on a completed crossing block, either
\(\operatorname{RL}(x_{\llbracket l,s(l)\rrbracket})
\geq Q\alpha_l^{p_{G(l)}}\) or
\(\alpha_{s(l)+1}\leq\alpha_l/(2c_d)\).

Every fact, lemma, and corollary established below is followed by a proof or a proof sketch.
A proof sketch records only the hypotheses, notation, and constants needed for
the corresponding argument in \cite{lai2025diameter}; a full proof is given
when the error terms or the step-size ratio bound change the argument.
Elementary consequences of definitions are stated directly.

\subsection{Verification for the inexact subgradient method}
Appendix~\ref{sec:proof_inexact} verifies
items~\ref{item:small_step} and~\ref{item:stratified_descent} of
Definition~\ref{def:stratified_descent} for the inexact subgradient method and
then applies Theorem~\ref{thm:diameter} to prove
Theorem~\ref{thm:inexact_diameter}. Its principal components have the
following counterparts or inputs in \cite{lai2025diameter}.
\begin{center}
\small
\begin{tabular}{@{}p{0.58\textwidth}p{0.34\textwidth}@{}}
\toprule
Component of Appendix~\ref{sec:proof_inexact}
& Counterpart or input in \cite{lai2025diameter}\\
\midrule
Proof of Theorem~\ref{thm:inexact_diameter} (inexact diameter bound)
& Proof of Corollary~4.2\\
Lemma~\ref{lemma:inexact_projected_length} (inexact projected length)
& Lemma~4.1\\
Proposition~\ref{proposition:neighborhoods of cells}
(stratified neighborhoods)
& Proposition~3.5(i), (ii), and (iv)\\
The \(C^3\) stratification used in the proof
& Proposition~3.2\\
\bottomrule
\end{tabular}
\end{center}
Proposition~\ref{proposition:neighborhoods of cells} records the neighborhood
and quantitative estimates from
\cite[Proposition~3.5(i), (ii), and~(iv)]{lai2025diameter}, based on the
stratification supplied by \cite[Proposition~3.2]{lai2025diameter}. The
sets \(\mathcal N(i,\alpha)\) constructed in Step~3 of the proof of
Theorem~\ref{thm:inexact_diameter} correspond to
\cite[Corollary~4.2(2)]{lai2025diameter} and Steps~2--4 of its proof. The
uniform signed desingularizing power chosen in Step~1 of the proof of
Theorem~\ref{thm:inexact_diameter} corresponds to Step~1 in the proof of
\cite[Corollary~4.2]{lai2025diameter}. The exponent choice in Step~2 realizes
item~\ref{item:strata} of Definition~\ref{def:stratified_descent}; its
counterpart is \cite[Corollary~4.2(1)]{lai2025diameter} and Step~2 of that
proof, with the additional constraints involving \(\xi\) and \(\tau\).

Lemma~\ref{lemma:inexact_projected_length} contains the projected Taylor
calculation corresponding to \cite[Lemma~4.1, equations~(4.2)--(4.3)]{lai2025diameter}
and the signed-power telescoping calculation corresponding to
\cite[(4.5)--(4.7)]{lai2025diameter}. The additional ingredients are
\eqref{eq:separate_error}, which isolates the additive perturbation, and
\eqref{verdier control}, which applies
\cite[Proposition~3.5(iv)]{lai2025diameter} to the Goldstein-type enlargement
before the argument corresponding to
\cite[Lemma~4.1, equation~(4.4)]{lai2025diameter}. Finally, the tail functions
\eqref{eq:g1}--\eqref{eq:g2} extend the step-size tails in
\cite[(4.25)]{lai2025diameter} by the additive-error tails and supply the errors in
items~\ref{item:small_step} and~\ref{item:stratified_descent} of
Definition~\ref{def:stratified_descent}; Theorem~\ref{thm:diameter} then yields
Theorem~\ref{thm:inexact_diameter}.

\section{Proof of Theorem~\ref{thm:diameter}}\label{sec:proof_diameter}

\subsection{Setup and notation}
\label{sec:diameter_setup}

Fix a difference inclusion \eqref{eq:DI_with_noise} satisfying
Definition~\ref{def:stratified_descent}, together with all the data appearing
in that definition.  Let \(L_f>0\) be a Lipschitz constant for \(f\) on
\(\widetilde X\).  After relabeling, assume that
$M_1,\dots,M_T$ are non-open and
$M_{T+1},\dots,M_{\overline T}$ are open.

By Definition~\ref{def:stratified_descent}, the constants satisfy
\[
0<\beta_i<(1-\theta)\gamma_i<\tau(1-\theta)
\quad\text{and}\quad
M_j\subset\partial M_i \ \Longrightarrow\ \beta_i>\gamma_j .
\]
For \(i\in\llbracket1,T\rrbracket\), define
\(p_i:=\gamma_i(1-\theta)\) and
\(\underline\beta_i:=\min\bigl(\{1\}\cup
\{\beta_j:M_i\subset\partial M_j\}\bigr)\).
    We also set
    \[
    \underline\gamma:=\min_{i\in \llbracket 1,\overline{T}\rrbracket}\gamma_i,
    \qquad
    \underline\beta:=\min_{i\in \llbracket 1,\overline{T}\rrbracket}\beta_i,
    \qquad
    \underline p:=\min\bigl(\{1\}\cup
    \{p_i:i\in\llbracket1,T\rrbracket\}\bigr),
    \qquad
    \bar c:=\max_{i\in \llbracket 1,\overline{T}\rrbracket}c_i.
    \]
    If \(M_i\subset\partial M_j\), then \(\beta_j>\gamma_i>p_i\); if no
    such \(j\) exists, then \(\underline\beta_i=1>p_i\). Hence
    \begin{equation}\label{eq:relation_beta_p}
    \beta_i<p_i<\gamma_i < \tau \leq 1,
    \qquad
    M_j\subset \partial M_i \Longrightarrow \gamma_j<p_i,
    \qquad
    \underline\beta_i>p_i.      
    \end{equation}
For
$i\in \llbracket 1,\overline{T}\rrbracket$, set
\begin{equation}\label{eq:hat_c}
    \widehat c_i:=c_d^{\gamma_i}c_i.
\end{equation}

The conclusion of Theorem~\ref{thm:diameter} is immediate when \(K=0\).
Throughout this section, we therefore fix an integer \(K\geq1\) and a
realization $(x_k)$ of \eqref{eq:DI_with_noise}
such that
\begin{equation}\label{eq:running_x}
    x_{\llbracket 0,K\rrbracket}\subset X,\qquad \sup_{k\in \mathbb{N}}\alpha_k \leq \bar{\alpha},\qquad
\sup_{k\geq i}\frac{\alpha_k}{\alpha_i}\leq c_d
\quad\text{for every }i\in\mathbb N.
\end{equation}
Choose
$\bar{\alpha}\in (0,\min\{1,\hat{\alpha}\}]$ small enough that, for every
$\alpha\in(0,\bar{\alpha}]$ and $i\in \llbracket 1, \overline{T}\rrbracket$,
\begin{equation}\label{eq:standing_alpha_small_hat}
\iota \alpha^\tau \leq \widehat c_i\alpha^{\gamma_i},
\qquad
2\widehat c_i\alpha^{\gamma_i}<c_i\alpha^{\beta_i}.
\end{equation}
Any further smallness requirements on $\bar{\alpha}$ will be stated when
used.
For $i\in \llbracket 1,\overline{T}\rrbracket$ and $k\in \llbracket 0,K\rrbracket$, whenever $x_k\in \mathcal{N}(i,\alpha_k)$ we write \(y_k^i:=P_{M_i}(x_k)\) and \(z_k^i:=f(y_k^i)+g_{1,k}\),
where \(g_{\ell,k}:=g_\ell((\alpha_m)_{m\geq k},(e_m)_{m\geq k})\)
for \(\ell\in\{1,2\}\).
We also write
\begin{equation}\label{eq:z_k_def}
    z_k:=f(x_k)+g_{1,k},\qquad\forall k\in\llbracket0,K\rrbracket.
\end{equation}
Every later length bound uses Definition~\ref{def:stratified_descent} only
through the following consequence for shifted segments.

\begin{fact}[Shifted projection estimate]
    \label{fact:y_k^i}
Let $i\in \llbracket 1,\overline{T}\rrbracket$ and let $0\leq a<b\leq K$. If
\(x_k\in\mathcal{N}(i,\alpha_k)\) for all
\(k\in\llbracket a,b\rrbracket\),
then
\[
\sum_{k=a}^{b-1}|y_{k+1}^i-y_k^i|
\leq
\psi(z_a^i)-\psi(z_b^i)
+
\sum_{k=a}^{b-1} g_{2,k}
+
\iota c_d^{\tau} \alpha_a^\tau.
\]
\end{fact}

\begin{proof}
This is precisely \eqref{eq:y_k^i} of Definition~\ref{def:stratified_descent}, applied to the shifted segment
$(x_{a+k})_{k\in \llbracket 0,b-a\rrbracket}$, with $\max_{\llbracket a,b-1\rrbracket}\alpha_k\leq c_d \alpha_a$.
\end{proof}

\subsection{Relative length of realizations}
\label{sec:diameter_crossings}
\label{sec:diameter_relative_length}

To identify where the projected motion may change strata, mark every step
whose joining segment meets an enlarged neighborhood of a non-open stratum.
Thus, denoting the segment by $[x_k,x_{k+1}]$, define
\begin{equation}\label{eq:def_IC}
I_C:=
\left\{
k\in \llbracket 0,K-1\rrbracket:
[x_k,x_{k+1}]
\cap
\bigcup_{i=1}^{T}
B(M_i,\widehat c_i\alpha_k^{\gamma_i})
\neq \emptyset
\right\}.
\end{equation}

\begin{fact}[Crossing detection]\label{fact:IC_close}
For every \(k\in I_C\), the set
\[
\mathcal{A}_k
:=
\left\{
i\in \llbracket 1,T\rrbracket:
d(x_k,M_i)\leq 2 \widehat c_i\alpha_k^{\gamma_i}
\right\}
\]
is nonempty. Moreover, among the strata indexed by \(\mathcal{A}_k\), there is a unique one
of minimal dimension. Denoting its index by \(i_0\), we have $x_k\in \mathcal{N}(i_0,\alpha_k)$.
\end{fact}

\begin{proof}
The proof of \cite[Fact~4.3]{lai2025diameter} applies with its step bound
replaced by \(\iota\alpha_k^\tau\) and its crossing radius replaced by
\(\widehat c_i\alpha_k^{\gamma_i}\). Indeed, \eqref{eq:def_IC} gives a
non-open stratum \(M_i\) and \(w\in[x_k,x_{k+1}]\) such that
\[
d(x_k,M_i)\leq |x_k-w|+d(w,M_i)
\leq\iota\alpha_k^\tau+\widehat c_i\alpha_k^{\gamma_i}
\leq2\widehat c_i\alpha_k^{\gamma_i}.
\]
Thus \(\mathcal A_k\neq\emptyset\). If \(i_0\in\mathcal A_k\) has minimal
dimension and \(M_h\subset\partial M_{i_0}\), strict dimension drop gives
\(h\notin\mathcal A_k\), and hence
\(d(x_k,M_h)>2\widehat c_h\alpha_k^{\gamma_h}
\geq c_h\alpha_k^{\gamma_h}\). Together with
\eqref{eq:standing_alpha_small_hat}, this proves
\(x_k\in\mathcal N(i_0,\alpha_k)\). Applying the same argument to two
minimal-dimensional candidates would place \(x_k\) in both neighborhoods.
Their equal dimensions and strict dimension drop make the strata unrelated,
contradicting the separation condition. Thus \(i_0\) is unique.
\end{proof}

\begin{fact}[Open-stratum persistence]\label{fact:IC_away}
Let $k_1,k_2\in \mathbb{N}$ satisfy $0\leq k_1<k_2\leq K$ and $\llbracket k_1,k_2-1\rrbracket\cap I_C=\emptyset$. Then there exists $i\in \llbracket T+1,\overline{T}\rrbracket$ such that
\(x_k\in M_i\cap\mathcal{N}(i,\alpha_k)\) for all
\(k\in\llbracket k_1,k_2\rrbracket\).
\end{fact}

\begin{proof}
Fix \(k\in\llbracket k_1,k_2-1\rrbracket\), and let \(M_i\) and \(M_j\)
be the strata containing \(x_k\) and \(x_{k+1}\), respectively. Since
\(k\notin I_C\), the segment \([x_k,x_{k+1}]\) is disjoint from
\(B(M_h,\widehat c_h\alpha_k^{\gamma_h})\) for every
\(h\in\llbracket1,T\rrbracket\). In particular, \(M_i\) and \(M_j\) are
open strata. Every stratum in the frontier of an open stratum is non-open,
so the same disjointness and \(\widehat c_h\geq c_h\) give
\(x_k\in\mathcal N(i,\alpha_k)\). Since
\(\alpha_{k+1}\leq c_d\alpha_k\) and
\(c_h\alpha_{k+1}^{\gamma_h}\leq
 c_h c_d^{\gamma_h}\alpha_k^{\gamma_h}
=\widehat c_h\alpha_k^{\gamma_h}\), they also give
\(x_{k+1}\in\mathcal N(j,\alpha_{k+1})\).

It remains to show that \(i=j\). Since \(x_{k+1}\in M_j\), item
\ref{item:small_step} of Definition~\ref{def:stratified_descent} and
\eqref{eq:standing_alpha_small_hat} yield
\[
d(x_k,M_j)\leq|x_{k+1}-x_k|
\leq\iota\alpha_k^\tau
\leq\widehat c_j\alpha_k^{\gamma_j}
<c_j\alpha_k^{\beta_j}.
\]
The segment also avoids
\(B(M_h,c_h\alpha_k^{\gamma_h})\) for every
\(M_h\subset\partial M_j\). Hence \(x_k\in\mathcal N(j,\alpha_k)\). If
\(i\neq j\), the two open strata are unrelated by the frontier condition,
and the separation condition in Definition~\ref{def:stratified_descent}
forbids \(x_k\in\mathcal N(i,\alpha_k)\cap\mathcal N(j,\alpha_k)\).
Therefore \(i=j\). Iterating this one-step conclusion from \(k_1\) to
\(k_2-1\) proves the claim.
\end{proof}

Fact~\ref{fact:IC_close} allows us to associate to each index in $I_C$ a unique nearby non-open
stratum of minimal dimension. We therefore define the assignment $G\colon I_C\to \llbracket 1, T\rrbracket$
by
\begin{equation}\label{eq:def_G}
G(k):= \operatorname*{argmin}\Big\{\dim(M_i): d(x_k,M_i)\leq 2\widehat c_i\alpha_k^{\gamma_i},\
i\in \llbracket 1, T\rrbracket\Big\}.
\end{equation}
By Fact \ref{fact:IC_close}, $x_k\in \mathcal{N}(G(k),\alpha_k)$ for all $k\in I_C$.

Assume for the moment that $I_C\neq \emptyset$. We recursively define
\begin{subequations}\label{eq:def_lsq}
\begin{align}
l_0&:=\min I_C, \label{eq:def_l0}\\
s(l_m)&:=\max\Big\{k\in \llbracket l_m,K\rrbracket:
x_j\in \mathcal{N}(G(l_m),\alpha_j), \forall j\in \llbracket l_m,k\rrbracket\Big\}, \label{eq:def_s}\\
q(l_m)&:=\max\Big\{k\in \llbracket l_m,s(l_m)\rrbracket:
d\!\left(x_k,M_{G(l_m)}\right)\leq 2\widehat c_{G(l_m)}\alpha_k^{\gamma_{G(l_m)}}\Big\}, \label{eq:def_q}\\
l_{m+1}&:=\inf\Big\{k\in (q(l_m),\infty): k\in I_C\Big\},
\label{eq:def_lnext}
\end{align}
\end{subequations}
Starting from \(l_0\), for each constructed \(l_m\), first define \(s(l_m)\)
and \(q(l_m)\). If the set in \eqref{eq:def_lnext} is nonempty, define
\(l_{m+1}\) and continue; otherwise stop. Let \(\overline m\) be the index of
the last constructed \(l_m\), set
\(\mathcal{L}:=\{l_0,\dots,l_{\overline m}\}\), and, only as a terminal
endpoint, set \(l_{\overline m+1}:=K\). This terminal endpoint does not belong
to \(\mathcal L\). Thus \(s\) and \(q\) are defined on all of \(\mathcal L\).

The block beginning at $l_m$ remains in
$\mathcal N(G(l_m),\alpha_k)$ through $s(l_m)$; within that block, $q(l_m)$
is the last index satisfying the stronger distance bound in
\eqref{eq:def_q}.  The next block begins at the first marked index after
$q(l_m)$.  The following fact rules out selecting the same non-open stratum
again before the current block ends.

\begin{fact}[No stratum return]\label{fact:no_repeat}
Let $l_m\in \mathcal{L}$. Then \(G(l_m)\neq G(k)\) for all
\(k\in(q(l_m),s(l_m)]\cap\mathcal{L}\).
Consequently, if $l,\tilde l\in \mathcal{L}$ satisfy $\tilde l>l$ and $G(\tilde l)=G(l)$, then $\tilde l>s(l)$.
\end{fact}

\begin{proof}
Write \(l:=l_m\).
If \(k\in(q(l),s(l)]\cap\mathcal L\) and \(G(k)=G(l)\), then \(k\) is
admissible in the maximum defining \(q(l)\), a contradiction. Combining this
with the fact that every later element of \(\mathcal L\) exceeds \(q(l)\)
proves the second claim. This is the argument of
\cite[Fact~4.6]{lai2025diameter}.
\end{proof}
The crossing decomposition determines which increments are counted.  Before
$q(l_m)$ we use increments projected onto $M_{G(l_m)}$; after $q(l_m)$ we use
the original increments.  The following definition concatenates these
contributions over arbitrary intervals.
\begin{definition}[Relative length]\label{def:RL}
Let \(0\leq k_1<k_2\leq K\). We define the relative length of
\(x_{\llbracket k_1,k_2\rrbracket}\), denoted by
\(\operatorname{RL}(x_{\llbracket k_1,k_2\rrbracket})\), in the following cases.
We also set
\(\operatorname{RL}(x_{\llbracket k,k\rrbracket})=0\) for every
\(k\in\llbracket0,K\rrbracket\). If \(I_C=\emptyset\), set
\[
\operatorname{RL}(x_{\llbracket k_1,k_2\rrbracket})
:=\sum_{k=k_1}^{k_2-1}|x_{k+1}-x_k|.
\]
In the remainder of the definition, assume that \(I_C\neq\emptyset\).

\runinheading{Case~1}
If $k_1<k_2 \leq l_0$, then
\[\operatorname{RL}(x_{\llbracket k_1,k_2\rrbracket})
:=
\sum_{k=k_1}^{k_2-1}|x_{k+1}-x_k|.\]

\runinheading{Case~2}
Assume that there exists \(m\in \llbracket 0,\overline m\rrbracket\) such that
\(l_m\leq k_1<k_2\leq l_{m+1}\),
and write \(i:=G(l_m)\) and \(q:=q(l_m)\).
\begin{enumerate}
    \item If \(q\leq k_1<k_2\), set
\[
\operatorname{RL}(x_{\llbracket k_1,k_2\rrbracket})
:=
\sum_{k=k_1}^{k_2-1}|x_{k+1}-x_k|.
\]
\item If \(l_m<k_1<q\), set
\[
\operatorname{RL}(x_{\llbracket k_1,k_2\rrbracket})
:=
\begin{cases}
\displaystyle
\sum_{k=k_1}^{k_2-1}|y_{k+1}^{i}-y_k^{i}|,
& \text{if } k_2<q,\\[1.2em]
\displaystyle
\sum_{k=k_1}^{q-1}|y_{k+1}^{i}-y_k^{i}|
+|x_q-y_q^{i}|
+\sum_{k=q}^{k_2-1}|x_{k+1}-x_k|,
& \text{if } k_1<q\leq k_2.
\end{cases}
\]
\item If \(k_1=l_m<q\), set
\[
\operatorname{RL}(x_{\llbracket l_m,k_2\rrbracket})
:=
\begin{cases}
\displaystyle
|x_{l_m}-y_{l_m}^{i}|
+\sum_{k=l_m}^{k_2-1}|y_{k+1}^{i}-y_k^{i}|,
&\text{if }k_2<q,\\[1.2em]
\displaystyle
|x_{l_m}-y_{l_m}^{i}|
+\sum_{k=l_m}^{q-1}|y_{k+1}^{i}-y_k^{i}|
+|x_q-y_q^i|
+\sum_{k=q}^{k_2-1}|x_{k+1}-x_k|,
&\text{if }q\leq k_2.
\end{cases}
\]
\end{enumerate}

\runinheading{Case~3}
Assume that
\(
(k_1,k_2)\cap \mathcal{L}\neq \emptyset.
\)
Let \(m_1:=\min\{m\in\llbracket0,\overline m\rrbracket:l_m>k_1\}\)
and \(m_2:=\max\{m\in\llbracket0,\overline m\rrbracket:l_m<k_2\}\).
Then
\begin{equation*}
\operatorname{RL}(x_{\llbracket k_1,k_2\rrbracket})
:=
\operatorname{RL}(x_{\llbracket k_1,l_{m_1}\rrbracket})
+\sum_{m=m_1}^{m_2-1}\operatorname{RL}(x_{\llbracket l_m,l_{m+1}\rrbracket})
+\operatorname{RL}(x_{\llbracket l_{m_2},k_2\rrbracket}).
\end{equation*}
\end{definition}

By inspection of the three cases in Definition~\ref{def:RL}, the terminal
projection bridge at \(q(l_m)\) belongs to the interval ending there, while
the initial projection bridge at \(l_m\) belongs to the interval beginning
there. Hence, for \(0\leq a\leq b\leq c\leq K\),
\begin{equation}\label{eq:RL_additivity}
\operatorname{RL}(x_{\llbracket a,c\rrbracket})
=\operatorname{RL}(x_{\llbracket a,b\rrbracket})
+\operatorname{RL}(x_{\llbracket b,c\rrbracket}).
\end{equation}
All remaining terms split as ordinary sums. Since all summands are
nonnegative, relative length is also nondecreasing under interval inclusion.
This is the observation in
\cite[Remark~4.8]{lai2025diameter}.

The next three lemmas show that relative length controls displacement
within one block, across several blocks, and hence on an arbitrary interval.
\begin{lemma}[One-block displacement]\label{lemma:RM_within_one}
For any \(m\in\llbracket0,\overline m\rrbracket\), it holds that
\begin{enumerate}
\item If \(l_m\leq k_2<q(l_m)\), then $|x_{l_m} - x_{k_2}| \leq \operatorname{RL}(x_{\llbracket l_m,k_2\rrbracket})+c_{G(l_m)}\alpha_{k_2}^{\beta_{G(l_m)}}$;
    \item If \(q(l_m)\leq k_2\leq l_{m+1}\), then $|x_{l_m} - x_{k_2}| \leq \operatorname{RL}(x_{\llbracket l_m,k_2\rrbracket})$.
\end{enumerate} 
\end{lemma}

\begin{proof}
Expand the broken path in Definition~\ref{def:RL} and apply the triangle
inequality, as in \cite[Lemma~4.9]{lai2025diameter}. Before \(q(l_m)\), only
the terminal projection error remains, and it is at most
\(c_{G(l_m)}\alpha_{k_2}^{\beta_{G(l_m)}}\) because \(k_2\leq s(l_m)\). At and after \(q(l_m)\), relative
length contains the full path, including the terminal projection bridge when
\(q(l_m)>l_m\); the singleton case is immediate.
\end{proof}

\begin{lemma}[Multiblock displacement]\label{lemma:RM_multi}
Let \(l\in\mathcal L\) and
\(k_2\in\llbracket q(l),s(l)\rrbracket\) satisfy \(k_2>l\), and set
\(m_2:=\max\{m\in\llbracket0,\overline m\rrbracket:l_m<k_2\}\).
Then either $M_{G(l_{m_2})}\subset \partial M_{G(l)}$ or
\[
|x_l-x_{k_2}|
\leq
\operatorname{RL}(x_{\llbracket l,k_2\rrbracket})
+\bar c c_d\alpha_l^{\underline\beta_{G(l)}}.
\]
\end{lemma}

\begin{proof}[Proof sketch]
Write \(l=l_{m_1}\). Since \(k_2\geq q(l)\),
Lemma~\ref{lemma:RM_within_one}(ii) proves the claim without a remainder when
\(k_2\leq l_{m_1+1}\). Suppose that \(k_2>l_{m_1+1}\), and set
\(h:=G(l_{m_2})\). For every \(m_1\leq m<m_2\), one has
\(q(l_m)<l_{m+1}\), so Lemma~\ref{lemma:RM_within_one}(ii) applies to the
complete block. Applying item~(i) or item~(ii) of the same lemma to the last
block, according as \(k_2<q(l_{m_2})\) or \(k_2\geq q(l_{m_2})\), and using
\eqref{eq:RL_additivity} gives
\[
\begin{aligned}
|x_l-x_{k_2}|
&\leq
\sum_{m=m_1}^{m_2-1}|x_{l_m}-x_{l_{m+1}}|
   +|x_{l_{m_2}}-x_{k_2}|\leq
\operatorname{RL}(x_{\llbracket l,k_2\rrbracket})
   +c_h\alpha_{k_2}^{\beta_h}.
\end{aligned}
\]
Here \(l_{m_2}\in(q(l),s(l)]\cap\mathcal L\). Hence
Fact~\ref{fact:no_repeat} gives \(h\neq G(l)\), while the definition of
\(s(l)\) and Fact~\ref{fact:IC_close} place \(x_{l_{m_2}}\) in both
\(\mathcal N(G(l),\alpha_{l_{m_2}})\) and
\(\mathcal N(h,\alpha_{l_{m_2}})\). The separation and frontier conditions
therefore imply that either \(M_h\subset\partial M_{G(l)}\), which is the
first alternative in the statement, or \(M_{G(l)}\subset\partial M_h\). In
the latter case, \(\beta_h\geq\underline\beta_{G(l)}\), and
\[
c_h\alpha_{k_2}^{\beta_h}
\leq\bar c(c_d\alpha_l)^{\beta_h}
\leq\bar c c_d\alpha_l^{\underline\beta_{G(l)}},
\]
where we used \eqref{eq:running_x}, \(\alpha_l\leq1\), \(\beta_h<1\), and
\(c_d\geq1\). This proves the second alternative.
\end{proof}

\begin{lemma}[Relative-length diameter bound]\label{lemma:diameter_by_RM}
For any $0\leq k_1<k_2\leq K$, one has
\begin{equation}\label{eq:endpoint_by_RM}
    |x_{k_1}-x_{k_2}|
    \leq
    \operatorname{RL}(x_{\llbracket k_1,k_2\rrbracket})
    +4\bar cc_d\alpha_{k_1}^{\underline{\beta}}.
\end{equation}
Consequently,
\begin{equation}\label{eq:diameter_by_RM}
    \operatorname{diam}(x_{\llbracket k_1,k_2\rrbracket})
    \leq
    \operatorname{RL}(x_{\llbracket k_1,k_2\rrbracket})
    +4\bar cc_d \alpha_{k_1}^{\underline{\beta}}.
\end{equation}
\end{lemma}

\begin{proof}
If \(I_C=\emptyset\), relative length is ordinary length, and both conclusions
follow from the triangle inequality. Assume that \(I_C\neq\emptyset\). First
consider \(a<b\) contained in one block
\(\llbracket l_m,l_{m+1}\rrbracket\), and write
\(i:=G(l_m)\) and \(q:=q(l_m)\). Expanding the broken path in
Definition~\ref{def:RL} gives
\[
|x_a-x_b|
\leq\operatorname{RL}(x_{\llbracket a,b\rrbracket})
+\mathbf1_{\{l_m<a<q\}}c_i\alpha_a^{\beta_i}
+\mathbf1_{\{b<q\}}c_i\alpha_b^{\beta_i}.
\]
Indeed, a projection error at the left endpoint occurs only when the interval
starts strictly inside the projected part, while the error at the right
endpoint occurs only when the interval ends before \(q\). The initial and
terminal projection bridges in Definition~\ref{def:RL} cover the other two
endpoints. An interval contained in \(\llbracket0,l_0\rrbracket\) is ordinary
and has no projection error.

Now split \(\llbracket k_1,k_2\rrbracket\) at every point of
\(\mathcal L\cap(k_1,k_2)\). The complete intervening blocks have no endpoint
error, and each of the initial and terminal partial blocks has at most two.
The triangle inequality and \eqref{eq:RL_additivity} therefore leave at most
four errors of the form \(c_i\alpha_r^{\beta_i}\), where \(r\geq k_1\). By
\eqref{eq:running_x}, \(\alpha_r\leq c_d\alpha_{k_1}\); hence
\[
c_i\alpha_r^{\beta_i}
\leq\bar c(c_d\alpha_{k_1})^{\beta_i}
\leq\bar c c_d\alpha_{k_1}^{\underline\beta},
\]
because \(\alpha_{k_1}\leq1\),
\(\underline\beta\leq\beta_i<1\), and \(c_d\geq1\). This proves
\eqref{eq:endpoint_by_RM}.

Finally, fix \(k_1\leq p<q\leq k_2\) and repeat the same decomposition for
\(\llbracket p,q\rrbracket\). Every projection error still occurs at an index
\(r\geq k_1\), so the preceding bound applies directly, while interval
monotonicity gives
\(\operatorname{RL}(x_{\llbracket p,q\rrbracket})
\leq\operatorname{RL}(x_{\llbracket k_1,k_2\rrbracket})\).
Taking the supremum over \(p,q\) proves \eqref{eq:diameter_by_RM}.
\end{proof}

\subsection{Bounding the relative length}
\label{sec:diameter_local_bounds}
\label{sec:diameter_halving}
\label{sec:diameter_global_bounds}

The remaining task is to bound the projected increments in
$\operatorname{RL}$.  Fact~\ref{fact:y_k^i} does so on one stratum using the
decrease of $\psi(z_k^i)$ and the accumulated $g_{2,k}$.  At a block junction,
however, the adjacent pieces may use $z_k^i$ and $z_k$, so we first compare
these endpoint quantities.
\begin{lemma}[Endpoint conversion]\label{lemma:endpoint_conversion}
There exists \(C_{\mathrm{end}}>0\) such that the following holds.

If \(x_k\in \mathcal{N}(i,\alpha_k)\) for some
\(i\in \llbracket 1,\overline{T}\rrbracket\), then
\(|\psi(z_k^i)-\psi(z_k)|\leq C_{\mathrm{end}}\alpha_k^{\beta_i(1-\theta)}\).
If, in addition, \(d(x_k,M_i)\leq 2\widehat c_i\alpha_k^{\gamma_i}\), then
\(|\psi(z_k^i)-\psi(z_k)|\leq2\psi(|z_k^i-z_k|)
\leq C_{\mathrm{end}}\alpha_k^{p_i}\).
\end{lemma}

\begin{proof}
This is the calculation of \cite[Lemma~4.13]{lai2025diameter}. The correction
\(g_{1,k}\) cancels, so \(|z_k^i-z_k|\leq L_f d(x_k,M_i)\). Multiply
\cite[Fact~4.12]{lai2025diameter} by \(c\), then use respectively the two
distance bounds in the statement. Since \(p_i=\gamma_i(1-\theta)\) and the
stratification is finite, one constant \(C_{\mathrm{end}}\) suffices.
\end{proof}

Fact~\ref{fact:y_k^i} controls the portion through $q(l)$; if the interval
continues, Facts~\ref{fact:IC_away} and~\ref{fact:y_k^i} control the subsequent
open-stratum portion.  Lemma~\ref{lemma:endpoint_conversion} compares the
scalar values at the junction.  Combining these inputs gives the following
one-block estimate.
\begin{lemma}[One-block length bound]\label{lemma:core_bridge}
There exists \(\hat C>0\) such that the following holds.

Let \(l\in \mathcal{L}\) and \(k_2\in \llbracket l,K\rrbracket\) satisfy either
\(k_2\in \llbracket l,q(l)\rrbracket\) or no integer \(k\) with
\(q(l)<k<k_2\) belongs to \(I_C\).
Define
\[
\widetilde z_{k_2}^{\,l}:=
\begin{cases}
z_{k_2}^{G(l)}, & k_2\leq q(l),\\[0.3em]
z_{k_2}, & k_2>q(l).
\end{cases}
\]
Then
\begin{equation}\label{eq:core_bridge}
\psi(z_l^{G(l)})-\psi(\widetilde z_{k_2}^{\,l})
\geq
\operatorname{RL}(x_{\llbracket l,k_2\rrbracket})
-\sum_{k=l}^{k_2-1}g_{2,k}
-\hat C\,\alpha_l^{p_{G(l)}}.
\end{equation}
\end{lemma}

\begin{proof}
Let \(i=G(l)\) and \(q=q(l)\).

\runinheading{Case~1: \(k_2\leq q\)}
If \(k_2=l\), the left-hand side and relative length are zero, so
\eqref{eq:core_bridge} is immediate. Assume \(l<k_2\leq q\). Then
\(x_k\in\mathcal N(i,\alpha_k)\) for all
\(k\in\llbracket l,k_2\rrbracket\). Applying Fact~\ref{fact:y_k^i} gives
\[
\psi(z_l^i)-\psi(z_{k_2}^i)
\geq
\sum_{k=l}^{k_2-1}|y_{k+1}^i-y_k^i|
-\sum_{k=l}^{k_2-1}g_{2,k}
-\iota c_d^\tau\alpha_l^\tau.
\]
By the definition of \(\operatorname{RL}(x_{\llbracket l,k_2\rrbracket})\),
\[
\operatorname{RL}(x_{\llbracket l,k_2\rrbracket})
=
|x_l-y_l^i|+\sum_{k=l}^{k_2-1}|y_{k+1}^i-y_k^i|
+\mathbf 1_{\{k_2=q\}}|x_q-y_q^i|.
\]
Hence
\[
\psi(z_l^i)-\psi(z_{k_2}^i)
\geq \operatorname{RL}(x_{\llbracket l,k_2\rrbracket})
-\sum_{k=l}^{k_2-1}g_{2,k}
-\iota c_d^\tau\alpha_l^\tau-
|x_l-y_l^i|
-\mathbf 1_{\{k_2=q\}}|x_q-y_q^i|.
\]
Since \(l\in\mathcal L\subset I_C\) and \(i=G(l)\),
\[
|x_l-y_l^i|\leq2\widehat c_i\alpha_l^{\gamma_i},
\qquad
|x_q-y_q^i|\leq2\widehat c_i\alpha_q^{\gamma_i}
\leq2\widehat c_i c_d^{\gamma_i}\alpha_l^{\gamma_i}.
\]
By \eqref{eq:relation_beta_p}, \(p_i<\min\{\gamma_i,\tau\}\). Since
\(\alpha_l\leq1\), both bridge terms and the step-size remainder are bounded
by a fixed multiple of \(\alpha_l^{p_i}\). This proves
\eqref{eq:core_bridge} when \(k_2\leq q\).

\runinheading{Case~2: \(k_2>q\)}
If \(q>l\), applying Fact~\ref{fact:y_k^i} on
\(\llbracket l,q\rrbracket\) along \(M_i\) gives
\[
\psi(z_l^i)-\psi(z_q^i)
\geq
\sum_{k=l}^{q-1}|y_{k+1}^i-y_k^i|
-\sum_{k=l}^{q-1}g_{2,k}
-\iota c_d^\tau\alpha_l^\tau,
\]
while this inequality is trivial when \(q=l\), because its left-hand side and
both sums vanish.

If \(k_2\geq q+2\), then no integer \(k\) with \(q<k<k_2\) belongs to
\(I_C\). Fact~\ref{fact:IC_away} yields an open stratum \(M_{i_0}\) such that
\(x_k\in\mathcal N(i_0,\alpha_k)\) for
\(k\in\llbracket q+1,k_2\rrbracket\). Applying
Fact~\ref{fact:y_k^i} on this interval and using
\(y_k^{i_0}=x_k\) and \(z_k^{i_0}=z_k\) gives
\[
\psi(z_{q+1})-\psi(z_{k_2})
\geq
\sum_{k=q+1}^{k_2-1}|x_{k+1}-x_k|
-\sum_{k=q+1}^{k_2-1}g_{2,k}
-\iota c_d^\tau\alpha_{q+1}^\tau.
\]
When \(k_2=q+1\), the same displayed inequality is the trivial estimate
\(0\geq-\iota c_d^\tau\alpha_{q+1}^\tau\).
Since \(\alpha_{q+1}\leq c_d\alpha_l\), summing the two inequalities and
inserting the bridge terms at \(q\) gives the estimate below. We also insert
the omitted term \(-g_{2,q}\); this only weakens the lower bound because
\(g_{2,q}\geq0\).
\begin{align*}
\psi(z_l^i)-\psi(z_{k_2})
&\geq\sum_{k=l}^{q-1}|y_{k+1}^i-y_k^i|
+\sum_{k=q+1}^{k_2-1}|x_{k+1}-x_k|
-\sum_{k=l}^{k_2-1}g_{2,k}
-\iota(c_d^\tau+c_d^{2\tau})\alpha_l^\tau\\
&\quad+\bigl(\psi(z_q^i)-\psi(z_q)\bigr)
+\bigl(\psi(z_q)-\psi(z_{q+1})\bigr)\\
&\geq\sum_{k=l}^{q-1}|y_{k+1}^i-y_k^i|
+\sum_{k=q+1}^{k_2-1}|x_{k+1}-x_k|
-\sum_{k=l}^{k_2-1}g_{2,k}
-\iota(c_d^\tau+c_d^{2\tau})\alpha_l^\tau\\
&\quad-2\psi(|z_q^i-z_q|)
-2\psi(|z_q-z_{q+1}|),
\end{align*}
where we used the signed-power H\"older estimate
\cite[Fact~4.12]{lai2025diameter}.
Now
\[
\operatorname{RL}(x_{\llbracket l,k_2\rrbracket})
\leq
|x_l-y_l^i|
+\sum_{k=l}^{q-1}|y_{k+1}^i-y_k^i|
+|x_q-y_q^i|
+\sum_{k=q}^{k_2-1}|x_{k+1}-x_k|.
\]
Substituting this inequality yields
\begin{align}
\psi(z_l^i)-\psi(z_{k_2})
&\geq\operatorname{RL}(x_{\llbracket l,k_2\rrbracket})
-\sum_{k=l}^{k_2-1}g_{2,k} \notag\\
&\quad-\Bigl(|x_l-y_l^i|+|x_q-y_q^i|+|x_q-x_{q+1}|\Bigr)
-\iota(c_d^\tau+c_d^{2\tau})\alpha_l^\tau \label{eq:error_1}\\
&\quad-2\psi(|z_q^i-z_q|)-2\psi(|z_q-z_{q+1}|).\label{eq:error_2}
\end{align}
It remains in Case~2 to absorb the junction errors.  
By the definitions of $l$ and $q$ in \eqref{eq:def_lsq},
\[
d(x_l,M_i)\leq 2\widehat c_i\alpha_l^{\gamma_i},\qquad d(x_q,M_i)\leq 2\widehat c_i\alpha_q^{\gamma_i}\leq 2\widehat c_ic_d^{\gamma_i}\alpha_l^{\gamma_i}.
\]
Therefore,
\[
|x_l-y_l^i|\leq 2\widehat c_i\alpha_l^{\gamma_i},
\qquad
|x_q-y_q^i|\leq 2\widehat c_ic_d^{\gamma_i}\alpha_l^{\gamma_i}.
\]
By \eqref{eq:relation_beta_p}, we have \(p_i:= \gamma_i(1-\theta)<\gamma_i\) and \(p_i<\tau\). As \(\alpha_l\leq 1\) and $|x_q-x_{q+1}|\leq \iota \alpha_q^\tau$, it follows that
\[
\Bigl(|x_l-y_l^i|+|x_q-y_q^i|+|x_q-x_{q+1}|\Bigr)+\iota(c_d^\tau+c_d^{2\tau}) \alpha_l^\tau
\leq \Bigl(2\widehat c_i\alpha_l^{\gamma_i}+2\widehat c_ic_d^{\gamma_i}\alpha_l^{\gamma_i}+\iota c_d^\tau\alpha_l^\tau\Bigr)+\iota(c_d^\tau+c_d^{2\tau})\alpha_l^\tau \leq C_1\alpha_l^{p_i}
\]
for some \(C_1>0\).

To bound \eqref{eq:error_2}, Lemma~\ref{lemma:endpoint_conversion}, the
definition of \(q\) in \eqref{eq:def_q}, and
\(\alpha_q\leq c_d\alpha_l\) give
\[
\psi(|z_q^i-z_q|)
\leq \frac{C_{\mathrm{end}}}{2}\alpha_q^{p_i}
\leq \frac{C_{\mathrm{end}}c_d}{2}\alpha_l^{p_i}.
\]

Finally, by the definition of $z_k$ in \eqref{eq:z_k_def},
\[
|z_q-z_{q+1}|
\leq |f(x_q)-f(x_{q+1})|+|g_{1,q}-g_{1,q+1}|
\leq L_f\iota \alpha_q^\tau+\iota\alpha_q^{\tau},
\]
where we used item \ref{item:small_step} in Definition \ref{def:stratified_descent} and Lipschitz continuity of $f$. Hence,
\[
\psi(|z_q-z_{q+1}|) \leq c(L_f\iota +\iota)^{1-\theta}\alpha_q^{\tau (1-\theta)}
\leq c(L_f\iota +\iota)^{1-\theta} c_d^{\tau (1-\theta)}\alpha_l^{p_i},
\]
as $p_i =  \gamma_i (1-\theta)<\tau(1-\theta)$ from \eqref{eq:relation_beta_p}.

The preceding bounds show that every term in
\eqref{eq:error_1}--\eqref{eq:error_2} is at most a fixed multiple of
$\alpha_l^{p_i}$.  Hence there exists $\hat C>0$ such that
\[
\psi(z_l^i)-\psi(z_{k_2})
\geq
\operatorname{RL}(x_{\llbracket l,k_2\rrbracket})
-\sum_{k=l}^{k_2-1}g_{2,k}
-\hat C\,\alpha_l^{p_i},
\]
which proves \eqref{eq:core_bridge} when \(k_2>q\).  Enlarging $\hat C$ if
necessary also covers the preceding case $k_2\leq q$.
\end{proof}

We shall sum one estimate for each adjacent pair $l_m,l_{m+1}$. The required
two-endpoint form is the following corollary.
\begin{corollary}[Consecutive crossings]\label{cor:lm->lm+1}
There exists \(\overline C>0\) such that for any \(m=0,\ldots,\overline m-1\),
\[
\psi\left(z^{G(l_m)}_{l_m}\right)-\psi\left(z^{G(l_{m+1})}_{l_{m+1}}\right)
\geq \operatorname{RL}(x_{\llbracket l_m,l_{m+1}\rrbracket})
-\sum_{k = l_m}^{l_{m+1}-1}g_{2,k}
-\overline{C} \left(\alpha_{l_m}^{p_{G(l_m)}} +\alpha_{l_{m+1}}^{p_{G(l_{m+1})}}\right).
\]
\end{corollary}

\begin{proof}
By construction, \(l_{m+1}>q(l_m)\) and
\((q(l_m),l_{m+1})\cap I_C=\emptyset\). Apply
Lemma~\ref{lemma:core_bridge} with \(k_2=l_{m+1}\), and then use
Fact~\ref{fact:IC_close} and Lemma~\ref{lemma:endpoint_conversion} at
\(l_{m+1}\). This gives the stated estimate with
\(\overline C:=\max\{\hat C,C_{\mathrm{end}}\}\), as in
\cite[Corollary~4.15]{lai2025diameter}.
\end{proof}

The endpoint powers in Corollary~\ref{cor:lm->lm+1} must not accumulate freely
under repeated visits to a stratum.  On the completed block
$\llbracket l,s(l)\rrbracket$, the ratio bound gives the decisive alternative:
either relative length dominates $\alpha_l^{p_{G(l)}}$, or
$\alpha_{s(l)+1}\leq\alpha_l/(2c_d)$. The factor \(1/(2c_d)\) ensures that the
ratio bound yields a geometric decrease between successive visits.

\begin{lemma}[Length-scale alternative]\label{lemma:RM_or_half}
For any $Q>0$, there exists $\bar{\alpha}_Q\in (0,\hat{\alpha}]$ such that, for every $\alpha\in (0,\bar{\alpha}_Q]$,
\begin{equation}\label{eq:alphaQ_case1_new}
c_i(2c_d)^{-\beta_i}\alpha^{\beta_i-p_i}-2\widehat c_i\alpha^{\gamma_i-p_i}-\iota c_d^\tau\alpha^{\tau-p_i}\geq Q+\bar c c_d,
\qquad \forall i\in \llbracket 1,T\rrbracket,
\end{equation}
and
\begin{equation}\label{eq:alphaQ_case2_new}
\widehat c_j\alpha^{\gamma_j-p_i}-\iota c_d^\tau\alpha^{\tau-p_i}\geq Q+\bar c c_d,
\qquad \forall i,j\in \llbracket 1,T\rrbracket \text{ with } M_j\subset \partial M_i.
\end{equation}
If $\bar{\alpha}\leq \bar{\alpha}_Q$, then for any $l\in \mathcal{L}$ such that $s(l)<K$, it holds that either
\[
\operatorname{RL}(x_{\llbracket l,s(l)\rrbracket})\geq Q\alpha_l^{p_{G(l)}}
\qquad\text{or}\qquad
\alpha_{s(l)+1}\leq \alpha_l/(2c_d).
\]
\end{lemma}

\begin{proof}
The existence of \(\bar{\alpha}_Q\) satisfying
\eqref{eq:alphaQ_case1_new}--\eqref{eq:alphaQ_case2_new}
follows from \eqref{eq:relation_beta_p}. Fix \(l\in\mathcal{L}\) with
\(s(l)<K\), and write \(i:=G(l)\), \(p:=p_i\), and \(s:=s(l)\).

\runinheading{Step~1: Handle a block with no interior step}
Assume \(s=l\). It suffices to prove
\(\alpha_{l+1}\leq \alpha_l/(2c_d)\). Suppose otherwise that
\(\alpha_{l+1}>\alpha_l/(2c_d)\). Since
\(x_l\in\mathcal{N}(i,\alpha_l)\) and
\(x_{l+1}\notin\mathcal{N}(i,\alpha_{l+1})\), either \(d(x_{l+1},M_i)>c_i\alpha_{l+1}^{\beta_i}\),
or there exists \(j\) with \(M_j\subset\partial M_i\) and
\(d(x_{l+1},M_j)\leq c_j\alpha_{l+1}^{\gamma_j}\).

In the first case, the definition of \(G\) gives
\(d(x_l,M_i)\leq2\widehat c_i\alpha_l^{\gamma_i}\). Thus
\[
\iota\alpha_l^\tau
\geq |x_{l+1}-x_l|
> c_i\alpha_{l+1}^{\beta_i}-2\widehat c_i\alpha_l^{\gamma_i}
> c_i(2c_d)^{-\beta_i}\alpha_l^{\beta_i}
  -2\widehat c_i\alpha_l^{\gamma_i},
\]
contradicting \eqref{eq:alphaQ_case1_new}. In the second case, the definition
of \(G\) yields \(d(x_l,M_j)>2\widehat c_j\alpha_l^{\gamma_j}\). Hence, using
\(\alpha_{l+1}\leq c_d\alpha_l\) and
\(\widehat c_j=c_d^{\gamma_j}c_j\),
\[
\iota\alpha_l^\tau
\geq |x_{l+1}-x_l|
>2\widehat c_j\alpha_l^{\gamma_j}
  -c_j\alpha_{l+1}^{\gamma_j}
\geq \widehat c_j\alpha_l^{\gamma_j},
\]
contradicting \eqref{eq:alphaQ_case2_new}. Therefore
\(\alpha_{l+1}\leq \alpha_l/(2c_d)\).

\runinheading{Step~2: Reduce the remaining case to endpoint separation}
We now assume \(s\geq l+1\) and \(\alpha_{s+1}>\alpha_l/(2c_d)\). We show that it
is enough to prove
\begin{equation}\label{eq:displacement_bound}
    |x_l-x_s|\geq (Q+\bar c c_d)\alpha_l^p .
\end{equation}
Let \(m_2:=\max\{m\in\llbracket0,\overline m\rrbracket:l_m<s\}\)
and \(i_2:=G(l_{m_2})\).
By Lemma~\ref{lemma:RM_multi}, either \(M_{i_2}\subset\partial M_i\), or
\[
|x_l-x_s|
\leq
\operatorname{RL}(x_{\llbracket l,s\rrbracket})
+\bar c c_d\,\alpha_l^{\underline\beta_i}
\leq
\operatorname{RL}(x_{\llbracket l,s\rrbracket})
+\bar c c_d\,\alpha_l^p,
\]
where we used \(\underline\beta_i>p\) and \(\alpha_l\leq1\). Thus, in this latter
case, \eqref{eq:displacement_bound} implies
\(\operatorname{RL}(x_{\llbracket l,s\rrbracket})\geq Q\alpha_l^p\).

It remains to handle the case \(M_{i_2}\subset\partial M_i\). Write
\(l=l_{m_1}\). By Definition~\ref{def:RL} and Lemma~\ref{lemma:RM_within_one},
\begin{equation}\label{eq:RL_split_last_block}
\operatorname{RL}(x_{\llbracket l,s\rrbracket})
\geq
|x_l-x_{l_{m_2}}|
+
\operatorname{RL}(x_{\llbracket l_{m_2},s\rrbracket}).
\end{equation}
If \(l_{m_2}=q(l_{m_2})\), then Lemma~\ref{lemma:RM_within_one} gives
\[
\operatorname{RL}(x_{\llbracket l,s\rrbracket})
\geq |x_l-x_{l_{m_2}}|+|x_{l_{m_2}}-x_s|
\geq |x_l-x_s|,
\]
so \eqref{eq:displacement_bound} again gives the desired bound.

Suppose instead that \(l_{m_2}<q(l_{m_2})\). Then \(l_{m_2}+1\leq s\), and
the explicit formula in Definition~\ref{def:RL} gives
\[
\operatorname{RL}(x_{\llbracket l_{m_2},s\rrbracket})
\geq
|x_{l_{m_2}}-y^{i_2}_{l_{m_2}}|
+
|y^{i_2}_{l_{m_2}+1}-y^{i_2}_{l_{m_2}}|
\geq
|x_{l_{m_2}}-y^{i_2}_{l_{m_2}+1}|.
\]
Combining this with \eqref{eq:RL_split_last_block}, we obtain
\[
\operatorname{RL}(x_{\llbracket l,s\rrbracket})
\geq |x_l-y^{i_2}_{l_{m_2}+1}|
\geq d(x_l,M_{i_2}).
\]
Since \(M_{i_2}\subset\partial M_i\), the definition of \(i=G(l)\) gives
\(d(x_l,M_{i_2})>2\widehat c_{i_2}\alpha_l^{\gamma_{i_2}}\). Hence
\[
\operatorname{RL}(x_{\llbracket l,s\rrbracket})
>2\widehat c_{i_2}\alpha_l^{\gamma_{i_2}}
\geq Q\alpha_l^p
\]
by \eqref{eq:alphaQ_case2_new}. Thus, except for the already settled cases, it
remains only to prove \eqref{eq:displacement_bound}.

\runinheading{Step~3: Prove separation at the exit}
By the definition of \(s=s(l)\), we have
\(x_k\in\mathcal{N}(i,\alpha_k)\) for \(k\in\llbracket l,s\rrbracket\), while
\(x_{s+1}\notin\mathcal{N}(i,\alpha_{s+1})\). Hence either
\(d(x_{s+1},M_i)>c_i\alpha_{s+1}^{\beta_i}\), or there exists
\(j\) with \(M_j\subset\partial M_i\) and
\(d(x_{s+1},M_j)\leq c_j\alpha_{s+1}^{\gamma_j}\).

In the first case, using \(d(x_l,M_i)\leq2\widehat c_i\alpha_l^{\gamma_i}\),
\(\alpha_{s+1}>\alpha_l/(2c_d)\), and \(\alpha_s\leq c_d\alpha_l\), we get
\[
\begin{aligned}
|x_s-x_l|
&\geq d(x_{s+1},M_i)-d(x_l,M_i)-\iota\alpha_s^\tau  \\
&>
c_i(2c_d)^{-\beta_i}\alpha_l^{\beta_i}
-2\widehat c_i\alpha_l^{\gamma_i}
-\iota c_d^\tau\alpha_l^\tau
\geq (Q+\bar c c_d)\alpha_l^p,
\end{aligned}
\]
where the last inequality follows from \eqref{eq:alphaQ_case1_new}.

In the second case, the definition of \(G\) gives
\(d(x_l,M_j)>2\widehat c_j\alpha_l^{\gamma_j}\). Since
\(\alpha_{s+1}\leq c_d\alpha_l\), \(\alpha_s\leq c_d\alpha_l\), and
\(\widehat c_j=c_d^{\gamma_j}c_j\), we obtain
\[
\begin{aligned}
|x_s-x_l|
&\geq d(x_l,M_j)-d(x_{s+1},M_j)-\iota\alpha_s^\tau\\
&>
2\widehat c_j\alpha_l^{\gamma_j}
-c_j\alpha_{s+1}^{\gamma_j}
-\iota\alpha_s^\tau\\
&\geq
\widehat c_j\alpha_l^{\gamma_j}
-\iota c_d^\tau\alpha_l^\tau
\geq (Q+\bar c c_d)\alpha_l^p,
\end{aligned}
\]
by \eqref{eq:alphaQ_case2_new}.  Thus \eqref{eq:displacement_bound} holds.
Lemma~\ref{lemma:RM_multi} gives the claimed relative-length bound when
$M_{i_2}\not\subset\partial M_i$,
Lemma~\ref{lemma:RM_within_one} gives it when
$M_{i_2}\subset\partial M_i$ and $l_{m_2}=q(l_{m_2})$, and the remaining
frontier case already satisfies the bound by \eqref{eq:alphaQ_case2_new}.
The proof is complete.
\end{proof}

We next group the endpoint powers in Corollary~\ref{cor:lm->lm+1} by the
selected stratum.  For a completed block,
Lemma~\ref{lemma:RM_or_half} either bounds
$\alpha_\ell^{p_{G(\ell)}}$ by a disjoint portion of relative length or,
together with \eqref{eq:running_x}, makes the step size at
the next occurrence at most half as large.  Fact~\ref{fact:no_repeat} allows
at most one unfinished block per stratum.  These observations give a
geometric-series bound.  For a finite set $A$, write
$\operatorname{card}(A)$ for its cardinality.
\begin{lemma}[Crossing-error bound]\label{lemma:sum_alpha_p}
Given \(0\leq k_1<k_2\leq K\), let
\(
 u:=\operatorname{card}(G(\llbracket k_1,k_2\rrbracket\cap I_C))
\).
For any \(Q>0\), let \(\bar{\alpha}_Q\) be given by Lemma~\ref{lemma:RM_or_half}. If
\(\bar{\alpha}\leq \bar{\alpha}_Q\), then
\begin{equation}\label{eq:sum_alpha_p_final}
\sum_{\ell\in \mathcal{L}\cap \llbracket k_1,k_2\rrbracket}\alpha_{\ell}^{p_{G(\ell)}}
\leq
\frac{u}{Q}\,\operatorname{RL}(x_{\llbracket k_1,k_2\rrbracket})
+u C_{\mathrm{geo}}\alpha_{k_1}^{\underline p},
\end{equation}
where
\(C_{\mathrm{geo}}:=c_d\left(1+\frac{1}{1-2^{-\underline p}}\right)>0\).
\end{lemma}
\begin{proof}
\runinheading{Step~1: Split completed and unfinished blocks}
Separate the crossing indices according to whether their blocks end before
$k_2$:
\begin{equation}\label{eq:alpha_split}
 \sum_{\ell\in \mathcal{L}\cap \llbracket k_1,k_2\rrbracket}\alpha_{\ell}^{p_{G(\ell)}}
=
\sum_{\substack{\ell\in \mathcal{L}\\ k_1 \leq \ell \leq k_2,\ s(\ell)<k_2}}\alpha_{\ell}^{p_{G(\ell)}}
+
\sum_{\substack{\ell\in \mathcal{L}\\ k_1 \leq \ell \leq k_2\leq s(\ell)}}\alpha_{\ell}^{p_{G(\ell)}}.   
\end{equation}

\runinheading{Step~2: Estimate the completed blocks}
For each \(i\in \llbracket 1,T\rrbracket\), define
\[
\mathcal{L}_i:= \{\ell\in \mathcal{L} \cap \llbracket k_1,k_2\rrbracket: \ s(\ell)<k_2,\ G(\ell) = i\}.
\]
Then \(\{\mathcal{L}_i\}_{i\in \llbracket 1,T\rrbracket}\) forms a partition of
\(\{\ell\in \mathcal{L}\cap \llbracket k_1,k_2\rrbracket: s(\ell)<k_2\}\). Moreover, if \(\mathcal{L}_i\neq \emptyset\), then
\(
G(\mathcal{L}_i)=\{i\}\subset G(\llbracket k_1,k_2\rrbracket\cap I_C),
\)
so this can happen for at most \(u\) different values of \(i\).

Fix \(i\in \llbracket 1,T\rrbracket\) such that \(\mathcal{L}_i\neq \emptyset\), and write \(\mathcal{L}_i=\{\ell_1<\ell_2<\cdots<\ell_r\}\).
By Fact~\ref{fact:no_repeat},
\begin{equation}\label{eq:ell_sep_revised}
\ell_{t+1}>s(\ell_t)\qquad \forall t\in \llbracket 1,r-1\rrbracket.
\end{equation}
Since \(s(\ell_t)<k_2\leq K\), Lemma~\ref{lemma:RM_or_half} applies to each \(\ell_t\), and yields
\[
\operatorname{RL}(x_{\llbracket \ell_t,s(\ell_t)\rrbracket})
\geq Q\bigl(1-I_t\bigr)\alpha_{\ell_t}^{p_i},
\]
where
\[
I_t:= \left\{\begin{array}{ll}
1 &\text{if }\alpha_{s(\ell_t)+1}\leq \alpha_{\ell_t}/(2c_d),\\
0 &\text{if }\alpha_{s(\ell_t)+1}> \alpha_{\ell_t}/(2c_d).
\end{array}\right.
\]
Hence
\[
\alpha_{\ell_t}^{p_i}
\leq \frac1Q\,\operatorname{RL}(x_{\llbracket \ell_t,s(\ell_t)\rrbracket})
+I_t\alpha_{\ell_t}^{p_i}.
\]
Summing over \(t=1,\dots,r\), we obtain
\begin{equation}\label{eq:sum_fixed_i_revised}
\sum_{\ell\in \mathcal{L}_i}\alpha_\ell^{p_i}
\leq \frac1Q\sum_{t=1}^r\operatorname{RL}(x_{\llbracket \ell_t,s(\ell_t)\rrbracket})
+\sum_{t=1}^r I_t\alpha_{\ell_t}^{p_i}.
\end{equation}

By \eqref{eq:ell_sep_revised}, the intervals \(\llbracket \ell_t,s(\ell_t)\rrbracket\), \(t=1,\dots,r\),
are pairwise disjoint and all lie in \(\llbracket k_1,k_2\rrbracket\). By Definition~\ref{def:RL}, each
\(\operatorname{RL}(x_{\llbracket \ell_t,s(\ell_t)\rrbracket})\) is a sum of some of the nonnegative terms appearing in
\(\operatorname{RL}(x_{\llbracket k_1,k_2\rrbracket})\). Since the corresponding intervals are disjoint, these groups of terms are disjoint as well. Therefore,
\[
\sum_{t=1}^r\operatorname{RL}(x_{\llbracket \ell_t,s(\ell_t)\rrbracket})
\leq \operatorname{RL}(x_{\llbracket k_1,k_2\rrbracket}).
\]

It remains to estimate \(\sum_{t=1}^r I_t\alpha_{\ell_t}^{p_i}\). Let \(J:=\{t\in \llbracket 1,r\rrbracket:I_t=1\}=\{t_1<\cdots<t_B\}\).
If \(J=\emptyset\), then this sum is zero. Otherwise, since \(I_{t_b}=1\), we have
\(\alpha_{s(\ell_{t_b})+1}\leq \alpha_{\ell_{t_b}}/(2c_d)\).
If \(b<B\), then by \eqref{eq:ell_sep_revised}, we have
\[
\alpha_{\ell_{t_{b+1}}}\leq c_d\alpha_{s(\ell_{t_b})+1}\leq \alpha_{\ell_{t_b}}/2.
\]
It follows inductively that
\[
\alpha_{\ell_{t_b}}
\leq 2^{-(b-1)}\alpha_{\ell_{t_1}}
\leq 2^{-(b-1)}c_d\alpha_{k_1},
\qquad \forall b\in\llbracket1,B\rrbracket.
\]
Consequently,
\[
\sum_{t=1}^r I_t\alpha_{\ell_t}^{p_i}
=\sum_{b=1}^{B}\alpha_{\ell_{t_b}}^{p_i}
\leq c_d^{p_i}\alpha_{k_1}^{p_i}\sum_{b=0}^{B-1}2^{-bp_i}
\leq \frac{c_d^{p_i}}{1-2^{-p_i}}\alpha_{k_1}^{p_i}.
\]
Since \(p_i\geq \underline p\), \(p_i<1\), \(c_d\geq 1\), and \(\alpha_{k_1}\leq 1\), we have
\[
\frac{c_d^{p_i}}{1-2^{-p_i}}\alpha_{k_1}^{p_i}
\leq
\frac{c_d}{1-2^{-\underline p}}\alpha_{k_1}^{\underline p}.
\]
Substituting this into \eqref{eq:sum_fixed_i_revised} gives
\[
\sum_{\ell\in \mathcal{L}_i}\alpha_\ell^{p_i}
\leq \frac1Q\,\operatorname{RL}(x_{\llbracket k_1,k_2\rrbracket})
+\frac{c_d}{1-2^{-\underline p}}\alpha_{k_1}^{\underline p}.
\]
Summing this inequality over all \(i\in \llbracket 1,T\rrbracket\) such that \(\mathcal{L}_i\neq \emptyset\), and using that there are at most \(u\) such indices, yields
\begin{equation}\label{eq:first_sum_bound_final}
\sum_{\substack{\ell\in \mathcal{L}\\ k_1 \leq \ell \leq k_2,\ s(\ell)<k_2}}\alpha_{\ell}^{p_{G(\ell)}}
\leq \frac{u}{Q}\,\operatorname{RL}(x_{\llbracket k_1,k_2\rrbracket})
+\frac{u c_d}{1-2^{-\underline p}}\alpha_{k_1}^{\underline p}.
\end{equation}

\runinheading{Step~3: Estimate the unfinished blocks}
For each \(i\in \llbracket 1,T\rrbracket\), define
\[
\mathcal{J}_i:=\{\ell\in \mathcal{L}:k_1\leq \ell\leq k_2\leq s(\ell),\ G(\ell)=i\}.
\]
Again, \(\{\mathcal{J}_i\}_{i\in \llbracket 1,T\rrbracket}\) forms a partition of
\(\{\ell\in \mathcal{L}:k_1\leq \ell\leq k_2\leq s(\ell)\}\), and \(\mathcal{J}_i\neq \emptyset\) for at most \(u\) different values of \(i\). Also, \(\operatorname{card}(\mathcal{J}_i)\leq 1\) for every \(i\). Indeed, if \(\bar\ell<\tilde\ell\) both belong to \(\mathcal{J}_i\), then Fact~\ref{fact:no_repeat} gives
\(\tilde\ell>s(\bar\ell)\geq k_2\),
contradicting \(\tilde\ell\leq k_2\). Hence
\begin{equation}\label{eq:second_sum_bound_final}
\sum_{\substack{\ell\in \mathcal{L}\\ k_1 \leq \ell \leq k_2\leq s(\ell)}}\alpha_{\ell}^{p_{G(\ell)}}
= \sum_{i=1}^T\sum_{\ell\in \mathcal{J}_i}\alpha_\ell^{p_i} 
\leq  \sum_{i\in \llbracket 1,T\rrbracket: \mathcal{J}_i \neq \emptyset}c_d^{p_i}\alpha_{k_1}^{p_i}
\leq u c_d\alpha_{k_1}^{\underline p}.
\end{equation}

Combining \eqref{eq:first_sum_bound_final} and \eqref{eq:second_sum_bound_final} gives
\[
\sum_{\ell\in \mathcal{L}\cap \llbracket k_1,k_2\rrbracket}\alpha_{\ell}^{p_{G(\ell)}}
\leq
\frac{u}{Q}\,\operatorname{RL}(x_{\llbracket k_1,k_2\rrbracket})
+u c_d\left(1+\frac{1}{1-2^{-\underline p}}\right)\alpha_{k_1}^{\underline p},
\]
which is \eqref{eq:sum_alpha_p_final}.
\end{proof}

Summing Corollary~\ref{cor:lm->lm+1} between two crossing indices counts each
endpoint power at most twice.  Lemma~\ref{lemma:sum_alpha_p} bounds their total
by $(u/Q)\operatorname{RL}$ and a geometric remainder.  The choice
$Q>2\overline C u$ absorbs the relative-length term.

\begin{lemma}[Crossing length]\label{lemma:RM_arbitrary_interval}
Let \(k_1,k_2\in \mathcal{L}\) satisfy \(k_1<k_2\), and let \(u:=\operatorname{card}\bigl(G(\llbracket k_1,k_2\rrbracket\cap I_C)\bigr)\).
Let \(\overline C>0\) be given by Corollary~\ref{cor:lm->lm+1}. Fix
\(Q>2\overline C u\), and let \(\bar{\alpha}_Q,C_{\mathrm{geo}}>0\) be given
by Lemma~\ref{lemma:sum_alpha_p}. If \(\bar{\alpha}\leq \bar{\alpha}_Q\), then
\[\operatorname{RL}(x_{\llbracket k_1,k_2\rrbracket})
\leq
\frac{Q}{Q-2\overline C u}
\Bigg(
\psi\bigl(z_{k_1}^{G(k_1)}\bigr)
-\psi\bigl(z_{k_2}^{G(k_2)}\bigr)
+\sum_{k=k_1}^{k_2-1}g_{2,k}+2\overline C\,u C_{\mathrm{geo}}\alpha_{k_1}^{\underline p}
\Bigg).\]
\end{lemma}

\begin{proof}
Write \(k_1=l_a\) and \(k_2=l_b\), where \(0\leq a<b\leq \overline m\).
Summing Corollary~\ref{cor:lm->lm+1} for \(m=a,\ldots,b-1\), using
\eqref{eq:RL_additivity}, and counting each crossing-index power at
most twice give
\[
\psi\bigl(z_{k_1}^{G(k_1)}\bigr)-\psi\bigl(z_{k_2}^{G(k_2)}\bigr)
\geq\operatorname{RL}(x_{\llbracket k_1,k_2\rrbracket})
-\sum_{k=k_1}^{k_2-1}g_{2,k}-2\overline C
\sum_{\ell\in\mathcal{L}\cap\llbracket k_1,k_2\rrbracket}
\alpha_\ell^{p_{G(\ell)}}.
\]
Lemma~\ref{lemma:sum_alpha_p} bounds the last sum by
\((u/Q)\operatorname{RL}(x_{\llbracket k_1,k_2\rrbracket})
+uC_{\mathrm{geo}}\alpha_{k_1}^{\underline p}\). Hence the coefficient of
relative length is \(1-2\overline C u/Q>0\), and rearranging gives the
stated estimate.
\end{proof}
Lemma~\ref{lemma:RM_arbitrary_interval} requires both endpoints to lie in
$\mathcal L$.  For the segment from the last crossing index to $K$, apply the
one-block estimate directly and use Lemma~\ref{lemma:endpoint_conversion} only
if the terminal value is still projected.

\begin{lemma}[Terminal length bound]\label{lemma:RM_terminal_piece}
There exists \(C_{\mathrm{tail}}>0\) such that the following holds.

Assume that \(\mathcal{L}\neq\emptyset\) and let \(l:=\max_{\ell\in\mathcal{L}}\ell\). Then
\begin{equation}\label{eq:terminal_piece_RM}
\operatorname{RL}(x_{\llbracket l,K\rrbracket})
\leq
\psi(z_l^{G(l)})-\psi(z_K)
+\sum_{k=l}^{K-1}g_{2,k}
+C_{\mathrm{tail}}\alpha_l^{\underline\beta(1-\theta)}.
\end{equation}
\end{lemma}

\begin{proof}
Write \(i:=G(l)\), \(q:=q(l)\), and
\(r:=\underline\beta(1-\theta)\). If \(K=l\), then relative length is zero
and Lemma~\ref{lemma:endpoint_conversion} gives
\(\psi(z_l^i)-\psi(z_l)\geq-C_{\mathrm{end}}\alpha_l^r\), which proves the
claim in this case. Hence assume that \(K>l\).

Suppose that \(K>q\). Since \(l\) is the last element of \(\mathcal L\),
\((q,K)\cap I_C=\emptyset\). Lemma~\ref{lemma:core_bridge}, applied with
\(k_2=K\), therefore gives
\[
\operatorname{RL}(x_{\llbracket l,K\rrbracket})
\leq\psi(z_l^i)-\psi(z_K)
+\sum_{k=l}^{K-1}g_{2,k}+\hat C\alpha_l^{p_i}.
\]
Since \(p_i>r\) and \(\alpha_l\leq1\), this has the required form.

It remains to consider \(l<K\leq q\). Lemma~\ref{lemma:core_bridge} gives
the preceding display with \(z_K^i\) in place of \(z_K\). Since
\(K\leq q\leq s(l)\), Lemma~\ref{lemma:endpoint_conversion} and the ratio
bound yield
\[
|\psi(z_K^i)-\psi(z_K)|
\leq C_{\mathrm{end}}\alpha_K^{\beta_i(1-\theta)}
\leq C_{\mathrm{end}}c_d^{\beta_i(1-\theta)}
\alpha_l^{\beta_i(1-\theta)}
\leq C\alpha_l^{\underline\beta(1-\theta)},
\]
where the last inequality uses \(\beta_i(1-\theta)\geq r\) and
\(\alpha_l\leq1\). Also \(\alpha_l^{p_i}\leq\alpha_l^r\). Choosing
\(C_{\mathrm{tail}}\) larger than the finitely many coefficients in these
three cases proves the result.
\end{proof}

\subsection{Final diameter estimate}\label{sec:final_proof}
Split the realization at its first and last crossing indices.
Fact~\ref{fact:y_k^i} controls the initial open-stratum piece,
Lemma~\ref{lemma:RM_arbitrary_interval} controls the crossings between them,
and Lemma~\ref{lemma:RM_terminal_piece} controls the final block.  Adding these
estimates telescopes all shifted values except one projection mismatch at the
first crossing.

\begin{proof}
Let \(\overline C>0\) be the constant that appears in
Corollary~\ref{cor:lm->lm+1} and Lemma~\ref{lemma:RM_arbitrary_interval}, and
set \(Q:=4\overline C\max\{1,T\}\).
Let \(\bar{\alpha}_Q,C_{\mathrm{geo}}>0\) be given by
Lemma~\ref{lemma:sum_alpha_p}. Fix a realization
\((x_k)\) of \eqref{eq:DI_with_noise} satisfying
\eqref{eq:running_x} with \(\bar{\alpha}\leq \min\{1,\bar{\alpha}_Q,\hat{\alpha}\}\),
in addition to the smallness requirements in \eqref{eq:standing_alpha_small_hat}.
For any such sequence, we may define its relative length following
Definition~\ref{def:RL}, and the arguments in
Sections~\ref{sec:diameter_setup}--\ref{sec:diameter_global_bounds} apply.

We first record the estimate used for every open-stratum interval. If
\(0\leq a<b\leq K\) and
\(\llbracket a,b-1\rrbracket\cap I_C=\emptyset\), then
Fact~\ref{fact:IC_away} places the iterates in one open-stratum neighborhood.
Applying Fact~\ref{fact:y_k^i} there gives
\begin{equation}\label{eq:open_interval_length}
\sum_{k=a}^{b-1}|x_{k+1}-x_k|
\leq\psi(z_a)-\psi(z_b)
+\sum_{k=a}^{b-1}g_{2,k}+\iota c_d^\tau\alpha_a^\tau.
\end{equation}

\runinheading{Case~1: No non-open crossing, \(I_C=\emptyset\)}
If \(I_C=\emptyset\), then \eqref{eq:open_interval_length} and the triangle
inequality give
\[
\operatorname{diam}(x_{\llbracket0,K\rrbracket})
\leq
\sum_{k=0}^{K-1}|x_{k+1}-x_k|
\leq
\psi(z_0)-\psi(z_K)
+\sum_{k=0}^{K-1}g_{2,k}
+\iota c_d^\tau\alpha_0^\tau .
\]
Since \(z_k=f(x_k)+g_{1,k}\), the signed-power H\"older estimate
\cite[Fact~4.12]{lai2025diameter} gives
\[
\psi(z_0)-\psi(z_K)
\leq
\psi(f(x_0))-\psi(f(x_K))
+2\psi(g_{1,0})+2\psi(g_{1,K}).
\]
Substituting this bound, doubling the resulting nonnegative right-hand side,
and using
\(\underline\beta(1-\theta)<\tau\) and \(\alpha_0\leq1\) give
\eqref{eq:diam_formula} after enlarging its constant.

\runinheading{Case~2: At least one non-open crossing, \(I_C\neq\emptyset\)}
Set \(l^-:=l_0=\min I_C\) and
\(l^+:=l_{\overline m}=\max \mathcal{L}\).
We treat this case in five steps.

\runinheading{Step~1: Control the portion before the first crossing}
We claim that
\begin{equation} \label{eq:RL_initial}
 \operatorname{RL}(x_{\llbracket 0,l^-\rrbracket})
 =\sum_{k=0}^{l^--1}|x_{k+1}-x_k|\leq \psi(z_0)-\psi(z_{l^-})
 +\sum_{k=0}^{l^--1}g_{2,k}+\iota c_d^\tau\alpha_0^\tau.
\end{equation}
If \(l^-=0\), the claim is immediate because its left-hand side is zero.
If \(l^->0\), the ordinary-length branch of Definition~\ref{def:RL} and
\eqref{eq:open_interval_length}, applied with \(a=0\) and \(b=l^-\), prove
\eqref{eq:RL_initial}.

\runinheading{Step~2: Absorb the errors between crossings}
If \(l^-=l^+\), then
\(
\operatorname{RL}(x_{\llbracket l^-,l^+\rrbracket})=0.
\)
Assume now that \(l^-<l^+\). Since
\(u:=\operatorname{card}(G(\llbracket l^-,l^+\rrbracket\cap I_C))
\leq\operatorname{card}(\llbracket1,T\rrbracket)=T\),
this case has \(T\geq1\). Hence \(Q=4\overline C T\), and
Lemma~\ref{lemma:RM_arbitrary_interval} gives
\begin{equation}\label{eq:RL_middle}
\operatorname{RL}(x_{\llbracket l^-,l^+\rrbracket})
\leq
2\Biggl(
\psi\bigl(z_{l^-}^{G(l^-)}\bigr)-\psi\bigl(z_{l^+}^{G(l^+)}\bigr)
+\sum_{k=l^-}^{l^+-1}g_{2,k}
+2\overline C T C_{\mathrm{geo}}\,\alpha_{l^-}^{\underline p}
\Biggr).
\end{equation}
The same inequality is trivial when $l^-=l^+$, because its left-hand side is
zero.

\runinheading{Step~3: Control the portion after the last crossing}
Lemma~\ref{lemma:RM_terminal_piece} yields
\begin{equation}\label{eq:RL_terminal}
\operatorname{RL}(x_{\llbracket l^+,K\rrbracket})
\leq \psi\bigl(z_{l^+}^{G(l^+)}\bigr)-\psi(z_K)
+\sum_{k=l^+}^{K-1}g_{2,k}+C_{\mathrm{tail}}\alpha_{l^+}^{\underline\beta(1-\theta)}
,
\end{equation}
for some $C_{\mathrm{tail}}>0$; Lemma~\ref{lemma:RM_terminal_piece} also covers
$l^+=K$.

\runinheading{Step~4: Telescope the shifted values}
Equation~\eqref{eq:RL_additivity} gives
\[
\begin{aligned}
 \operatorname{RL}(x_{\llbracket 0,K\rrbracket})
 &=\operatorname{RL}(x_{\llbracket 0,l^-\rrbracket})
 +\operatorname{RL}(x_{\llbracket l^-,l^+\rrbracket})
 +\operatorname{RL}(x_{\llbracket l^+,K\rrbracket}).
\end{aligned}
\]
Enlarging the bounds in \eqref{eq:RL_initial} and \eqref{eq:RL_terminal} by a
factor of \(2\), and then summing them with \eqref{eq:RL_middle}, we obtain
\begin{align*}
\operatorname{RL}(x_{\llbracket 0,K\rrbracket})
&\leq2\Bigl(
\psi(z_0)-\psi(z_{l^-})
+\psi\bigl(z_{l^-}^{G(l^-)}\bigr)-\psi\bigl(z_{l^+}^{G(l^+)}\bigr)
+\psi\bigl(z_{l^+}^{G(l^+)}\bigr)-\psi(z_K)
\Bigr)\\
&+2\sum_{k=0}^{K-1}g_{2,k}
+2\Bigl(
\iota c_d^\tau\alpha_0^\tau
+C_{\mathrm{tail}}\alpha_{l^+}^{\underline\beta(1-\theta)}
+2\overline{C}T C_{\mathrm{geo}}\alpha_{l^-}^{\underline p}
\Bigr).
\end{align*}
The $l^+$-terms cancel.  At $l^-$, the projected and unprojected values leave
the single mismatch
\[
\begin{aligned}
 &\psi(z_0)-\psi(z_{l^-})
 +\psi\bigl(z_{l^-}^{G(l^-)}\bigr)-\psi\bigl(z_{l^+}^{G(l^+)}\bigr)+\psi\bigl(z_{l^+}^{G(l^+)}\bigr)-\psi(z_K)\\
 =&\psi(z_0)-\psi(z_K)
 +\psi\bigl(z_{l^-}^{G(l^-)}\bigr)-\psi(z_{l^-}).
\end{aligned}
\]
Since \(l^-\in I_C\), combining Fact \ref{fact:IC_close} and Lemma~\ref{lemma:endpoint_conversion} gives
\[
\bigl|\psi\bigl(z_{l^-}^{G(l^-)}\bigr)-\psi(z_{l^-})\bigr|
\leq
C_{\mathrm{end}}\alpha_{l^-}^{p_{G(l^-)}} \leq C_{\mathrm{end}}\alpha_{l^-}^{\underline p}.
\]
Therefore,
\begin{align}
\operatorname{RL}(x_{\llbracket 0,K\rrbracket})
&\leq2\bigl(\psi(z_0)-\psi(z_K)\bigr)
+2\sum_{k=0}^{K-1}g_{2,k} \notag\\
&+2\Bigl(
\iota c_d^\tau\alpha_0^\tau
+C_{\mathrm{tail}}\alpha_{l^+}^{\underline\beta(1-\theta)}
+2\overline{C}T C_{\mathrm{geo}}\alpha_{l^-}^{\underline p}+C_{\mathrm{end}}\alpha_{l^-}^{\underline{p}}
\Bigr).\label{eq:RL_total_pre}
\end{align}

\runinheading{Step~5: Recover the diameter and original endpoint values}
By \eqref{eq:running_x},
$\alpha_{l^-},\alpha_{l^+}\leq c_d\alpha_0$; moreover,
$\alpha_0^{\underline\beta}\leq
\alpha_0^{\underline\beta(1-\theta)}$.  Hence
Lemma~\ref{lemma:diameter_by_RM} and \eqref{eq:RL_total_pre}, after enlarging
the fixed coefficients, give
\begin{align*}
\operatorname{diam}(x_{\llbracket 0,K\rrbracket})
&\leq2\bigl(\psi(z_0)-\psi(z_K)\bigr)
+2\sum_{k=0}^{K-1}g_{2,k} \notag\\
&+2\Bigl(
\iota c_d^\tau\alpha_0^\tau
+C_{\mathrm{tail}}c_d\alpha_{0}^{\underline\beta(1-\theta)}
+2c_d\overline{C}TC_{\mathrm{geo}}\alpha_{0}^{\underline p}
+C_{\mathrm{end}}c_d\alpha_{0}^{\underline{p}}
\Bigr).
\end{align*}
Because $z_k=f(x_k)+g_{1,k}$,
the signed-power H\"older estimate
\cite[Fact~4.12]{lai2025diameter} separates the additive corrections at
the two endpoints:
\[
 \psi(z_0)-\psi(z_K)
 \leq \psi(f(x_0))-\psi(f(x_K))+2\psi(g_{1,0})+2\psi(g_{1,K}).
\]
Finally, $\underline\beta(1-\theta)<\min\{\underline p,\tau\}$ by
\eqref{eq:relation_beta_p}, and $\alpha_0\leq1$.  Thus all remaining
step-size powers are bounded by a fixed multiple of
$\alpha_0^{\underline\beta(1-\theta)}$.  Therefore, for some $C>0$,
\[\operatorname{diam}(x_{\llbracket 0,K\rrbracket}) \leq 2(\psi(f(x_0))-\psi(f(x_K))) + 4(\psi(g_{1,0})+ \psi(g_{1,K})) + 2\sum_{k=0}^{K-1}g_{2,k} + C\alpha_0^{\underline\beta(1-\theta)}.\]
\end{proof}

\section{Proof of Theorem \ref{thm:inexact_diameter}}\label{sec:proof_inexact}
The proof of Theorem~\ref{thm:inexact_diameter} has two parts. We first
establish a local projected-length estimate for the inexact subgradient
method. We then construct scale-dependent stratified neighborhoods and verify
the stratified descent condition in Definition~\ref{def:stratified_descent},
and apply Theorem~\ref{thm:diameter}.

\subsection{The local projected-length estimate}

Let \(M\subset\mathbb{R}^n\) be a \(C^2\) manifold and let
\(g\colon M\to\mathbb{R}\) be \(C^1\). A power function
\(\psi(t):=c\operatorname{sgn}(t)|t|^{1-\theta}\), with \(c>0\) and
\(\theta\in(0,1)\), is a \emph{desingularizing function} of \(g\) on
\([|g|\leq\varepsilon]\) if
\(1/\psi'(g(x))\leq|\nabla_M g(x)|\) whenever
\(0<|g(x)|\leq\varepsilon\), with \(1/\psi'(0)=0\). If \(M\) is bounded
and definable and \(g\) is definable, applying the signed \L{}ojasiewicz
gradient inequality to \(g\) and \(-g\), and taking common constants, gives
such a function, with
\(1/\psi'(g(x))=|g(x)|^\theta/(c(1-\theta))\)
\cite{lojasiewicz1958,kurdyka1994wf}; see also
\cite[Corollary~12 and Remark~7]{bolte2007clarke}. The following lemma gives the local projected-length estimate needed below.
\begin{lemma}[Inexact projected length]\label{lemma:inexact_projected_length}
Let \(f\colon\mathbb{R}^n\to\mathbb{R}\) be locally Lipschitz, let \(M\subset \mathbb{R}^n\) be a \(C^3\) manifold, and assume that \(f|_M\) is \(C^2\). Fix \(c_b,\xi>0\). Assume that there exist \(L>1\), a set \(U\subset\mathbb{R}^n\), and \(\varepsilon>0\) such that
\begin{itemize}
    \item \(P_M\) is single-valued and \(L\)-Lipschitz continuous on \(U\), and is \(C^2\) on an open neighborhood of \(U\),
    \item \(f\) is \(L\)-Lipschitz continuous on \(U\), and \(f|_M\) is \(L\)-Lipschitz continuous on \(P_M(U)\),
    \item \(\psi\colon\mathbb{R}\to\mathbb{R}\) is a desingularizing function of \(f|_M\) on \([|f|\leq\varepsilon]\).
\end{itemize}

Let \(K\geq1\), and let \(x_{\llbracket 0,K\rrbracket}\subset U\), \(\alpha_{\llbracket 0,K-1\rrbracket}\subset(0,1]\), and \(e_{\llbracket 0,K-1\rrbracket}\subset\mathbb{R}^n\) satisfy
\[
x_{k+1}=x_k-\alpha_ku_k+e_k,
\qquad
u_k\in\partial f(x_k;c_b\alpha_k^\xi),
\qquad |u_k|\leq L,
\qquad
\forall k\in \llbracket 0,K-1\rrbracket.
\]
Set \(y_k:=P_M(x_k)\) and \(d_k:=d(x_k,M)\) for
\(k\in\llbracket0,K\rrbracket\). Assume that
\(y_k\in[|f|\leq\varepsilon]\) for all
\(k\in\llbracket0,K\rrbracket\) and, for every
\(k\in\llbracket0,K-1\rrbracket\),
\(B(x_k,\alpha_kL)\cup B(x_k,c_b\alpha_k^\xi)\subset U\).
Assume moreover that there exist constants \(L_{f,k},L_{V,k},L_{P,k}\geq L\) for
\(k\in \llbracket0,K-1\rrbracket\) such that, for each such \(k\),
\begin{enumerate}
    \item \( |\nabla_M f(x)-\nabla_M f(x')|
    \leq L_{f,k}|x-x'|\) for all
    \(x,x'\in B(y_k,\alpha_kL^2+L|e_k|)\cap M\),
    \item \( |P_{T_M(y_k)}v-\nabla_M f(y_k)|
    \leq L_{V,k}|x-y_k|\) for all \(x\in B(x_k,c_b\alpha_k^\xi)\) and \(v\in\partial f(x)\),
    \item \( |DP_M(x)-DP_M(x')|
    \leq L_{P,k}|x-x'|\) for all \(x,x'\in B(x_k,\alpha_kL)\cup\{y_k\}\).
\end{enumerate}

Let \(g_0,\ldots,g_K\) be any nonnegative scalars such that
\[
g_k-g_{k+1}
\geq
\alpha_kL_{V,k}^2(d_k^2+c_b^2\alpha_k^{2\xi})
+L^2\alpha_kL_{P,k}d_k
+L^2|e_k|
+\frac{L^4\alpha_k^2}{2}(L_{f,k}+L_{P,k})
\]
for every \(k\in\llbracket0,K-1\rrbracket\). Define \(z_k:=f(y_k)+g_k\) for \(k\in\llbracket0,K\rrbracket\). Then $z_0\geq z_1\geq\cdots\geq z_K$ and
\begin{align}
\sum_{k=0}^{K-1}|y_{k+1}-y_k|
&\leq2L\left(\psi(z_0)-\psi(z_K)\right) \notag\\
&+\sum_{k=0}^{K-1}
\left(
L^3L_{f,k}\alpha_k^2
+L^2L_{f,k}\alpha_k|e_k|
+L\alpha_kL_{V,k}(d_k+c_b\alpha_k^\xi)
+\frac{L\alpha_k}{\psi'(g_k)}
+L|e_k|
\right)
\notag\\
&+L^2\max_{k\in\llbracket0,K-1\rrbracket}\alpha_k .
\label{eq:y_k-error}
\end{align}
\end{lemma}

\begin{proof}[Proof of Lemma~\ref{lemma:inexact_projected_length}]
Fix \(k\in\llbracket0,K-1\rrbracket\), and write
\(r_k:=c_b\alpha_k^\xi\), \(h:=f\circ P_M\), and
\(x_k^+:=x_k-\alpha_ku_k\).
The assumed bound \(|u_k|\leq L\) gives
\(x_k^+\in B(x_k,\alpha_kL)\subset U\).  The chain-rule calculation in
\cite[(4.2)]{lai2025diameter} applies to \(h\) and shows that its gradient is
\(L^2(L_{f,k}+L_{P,k})\)-Lipschitz on \(B(x_k,\alpha_kL)\).

\runinheading{Step~1: Separate the additive error}
The map \(h\) is \(L^2\)-Lipschitz on \(U\).  Since
\(x_{k+1}=x_k^++e_k\) and \(x_{k+1},x_k^+\in U\),
\begin{equation}
 f(y_{k+1})-f(y_k)
 \leq L^2|e_k|+h(x_k^+)-h(x_k).
\label{eq:separate_error}
\end{equation}
Applying the Taylor estimate \cite[(4.3)]{lai2025diameter} to the exact
displacement \(x_k^+-x_k=-\alpha_ku_k\), and using \(|u_k|\leq L\), gives
\[
 f(y_{k+1})-f(y_k)
 \leq L^2|e_k|
 -\alpha_k\langle u_k,\nabla_M f(y_k)\rangle
 +\alpha_kL^2L_{P,k}d_k
 +\frac{L^4(L_{f,k}+L_{P,k})}{2}\alpha_k^2.
\]

\runinheading{Step~2: Control the Goldstein--Verdier gap}
Write \(u_k=\sum_\ell\lambda_\ell v_\ell\), where
\(w_\ell\in B(x_k,r_k)\),
\(v_\ell\in\partial f(w_\ell)\), \(\lambda_\ell\geq0\), and
\(\sum_\ell\lambda_\ell=1\).  Item~2 in the statement of
Lemma~\ref{lemma:inexact_projected_length} gives
\[
 |P_{T_M(y_k)}v_\ell-\nabla_M f(y_k)|
 \leq L_{V,k}|w_\ell-y_k|
 \leq L_{V,k}(d_k+r_k).
\]
Taking the convex combination and using the linearity of
\(P_{T_M(y_k)}\) yields
\begin{equation}\label{verdier control}
\bigl|P_{T_M(y_k)}u_k-\nabla_M f(y_k)\bigr|
\leq L_{V,k}(d_k+r_k).
\end{equation}
Since \(\nabla_M f(y_k)\in T_M(y_k)\), polarization and
\eqref{verdier control} give
\[
\begin{aligned}
-2\langle u_k,\nabla_M f(y_k)\rangle
&=|P_{T_M(y_k)}u_k-\nabla_M f(y_k)|^2
  -|P_{T_M(y_k)}u_k|^2-|\nabla_M f(y_k)|^2\\
&\leq2L_{V,k}^2(d_k^2+r_k^2)-|\nabla_M f(y_k)|^2,
\end{aligned}
\]
where \((d_k+r_k)^2\leq2(d_k^2+r_k^2)\). Multiplying the Taylor bound by
two and substituting this estimate yields
\[
\begin{aligned}
2\bigl(f(y_{k+1})-f(y_k)\bigr)
&\leq-\alpha_k|\nabla_M f(y_k)|^2
 +2\alpha_kL_{V,k}^2(d_k^2+r_k^2)\\
&\quad+2L^2\alpha_kL_{P,k}d_k+2L^2|e_k|
 +L^4\alpha_k^2(L_{f,k}+L_{P,k})\\
&\leq2(g_k-g_{k+1})-\alpha_k|\nabla_M f(y_k)|^2.
\end{aligned}
\]
The last inequality uses the assumed lower bound on \(g_k-g_{k+1}\) and
\(r_k^2=c_b^2\alpha_k^{2\xi}\). Since \(z_k=f(y_k)+g_k\), rearranging gives
\[
 \frac12\alpha_k|\nabla_M f(y_k)|^2\leq z_k-z_{k+1},
\]
so \(z_0\geq z_1\geq\cdots\geq z_K\).
Whenever \(z_k=z_{k+1}=0\), this inequality gives
\(\nabla_Mf(y_k)=0\); such updates contribute nothing to the telescoping
argument below.

\runinheading{Step~3: Complete the projected-length estimate}
The additive perturbation also changes one input to the signed-power
telescoping.  Indeed, the Lipschitz continuity of \(P_M\) gives
\[
 |y_{k+1}-y_k|
 \leq L|x_{k+1}-x_k|
 \leq L^2\alpha_k+L|e_k|,
\]
so item~1 yields
\[
 |\nabla_M f(y_{k+1})|
 \leq |\nabla_M f(y_k)|
      +L_{f,k}(L^2\alpha_k+L|e_k|).
\]
The one-sign calculations in
\cite[proof of Lemma~4.1, equations~(4.5)--(4.7)]{lai2025diameter} give the
following estimate in the present notation. If
\(0\leq a\leq b\leq K\) and \(z_a,\ldots,z_b\) have the same weak sign, then
\begin{equation}
\begin{aligned}
\sum_{k=a}^{b-1}\alpha_k|\nabla_M f(y_k)|
&\leq2\bigl(\psi(z_a)-\psi(z_b)\bigr)\\
&\quad+\sum_{k=a}^{b-1}\left(
L^2L_{f,k}\alpha_k^2+LL_{f,k}\alpha_k|e_k|
+\frac{\alpha_k}{\psi'(g_k)}\right).
\end{aligned}
\label{eq:inexact_one_sign}
\end{equation}
For a nonpositive block, the two Lipschitz terms arise from the preceding
gradient-variation estimate; for a nonnegative block, they may be retained
because they are nonnegative. The zero-to-zero updates vanish by the descent
inequality.

If \(z_K\geq0\) or \(z_0\leq0\), applying
\eqref{eq:inexact_one_sign} to \(\llbracket0,K\rrbracket\) gives the desired
bound. Suppose instead that \(z_0>0>z_K\), and set
\(j:=\max\{k\in\llbracket0,K-1\rrbracket:z_k>0\}\). If
\(z_{j+1}<0\), set \(r:=j+1\); otherwise, set
\(r:=\max\{k\in\llbracket j+1,K-1\rrbracket:z_k=0\}\). In the latter case,
\(\nabla_Mf(y_k)=0\) for \(j+1\leq k<r\). Thus the full sum consists of the
two one-sign blocks \(\llbracket0,j\rrbracket\) and
\(\llbracket r,K\rrbracket\), the single update at \(j\), and possibly
zero-to-zero updates. Apply \eqref{eq:inexact_one_sign} to the two blocks and
use
\(\alpha_j|\nabla_Mf(y_j)|\leq
L\max_{k\in\llbracket0,K-1\rrbracket}\alpha_k\).
Since \(\psi(z_r)=\psi(z_{j+1})\) and
\(\psi(z_j)\geq\psi(z_{j+1})\), the two potential drops are bounded by
\(\psi(z_0)-\psi(z_K)\). Consequently,
\begin{align*}
\sum_{k=0}^{K-1}\alpha_k|\nabla_M f(y_k)|
&\leq
2\left(\psi(z_0)-\psi(z_K)\right)
+\sum_{k=0}^{K-1}
\left(
 L^2L_{f,k}\alpha_k^2
 +L L_{f,k}\alpha_k|e_k|
 +\frac{\alpha_k}{\psi'(g_k)}
\right)
+L\max_{k\in\llbracket0,K-1\rrbracket}\alpha_k .
\end{align*}

Finally, set \(\widetilde x_k:=x_k-\alpha_kP_{N_M(y_k)}u_k\). The segment
\([x_k,\widetilde x_k]\) lies in \(B(x_k,\alpha_kL)\subset U\) and in the
affine normal space \(y_k+N_M(y_k)\). Since \(P_M\) is smooth on an open
neighborhood of this segment, the normal-fiber consequence of
\cite[Theorem~(3.13)(b)]{dudek1994nonlinear} gives
\(P_M(\widetilde x_k)=P_M(x_k)=y_k\). Therefore, the Lipschitz continuity
of \(P_M\), the orthogonal decomposition of \(u_k\), and
\eqref{verdier control} give
\[
\begin{aligned}
|y_{k+1}-y_k|
&=|P_M(x_{k+1})-P_M(\widetilde x_k)|\\
&\leq L\bigl|e_k-\alpha_kP_{T_M(y_k)}u_k\bigr|\\
&\leq L|e_k|+L\alpha_kL_{V,k}(d_k+r_k)
 +L\alpha_k|\nabla_M f(y_k)|.
\end{aligned}
\]
Summing this inequality and inserting the preceding bound for
\(\sum_k\alpha_k|\nabla_M f(y_k)|\) proves
\eqref{eq:y_k-error}.
\end{proof}

\subsection{Scale-dependent stratified neighborhoods and descent verification}

We record the conclusions of
\cite[Proposition~3.5(i), (ii), and (iv)]{lai2025diameter} used below; their
proofs are not repeated.
\begin{proposition}[Stratified neighborhoods]\label{proposition:neighborhoods of cells}
Let \( f\colon \mathbb{R}^n \to \mathbb{R} \) be locally Lipschitz definable, let
$X\subset \mathbb{R}^n$ be definable compact, let $L>1$, and let
$\{M_1,\ldots,M_{\overline{T}}\}$ be a stratification of $X$ supplied by
\cite[Proposition~3.2]{lai2025diameter} for some $m\geq3$. Then there exist
$\eta_i\geq1$ for $i\in\llbracket1,\overline{T}\rrbracket$ such that for any
$\beta_i,\gamma_i\geq0$ satisfying
\begin{equation*}
\forall i,j\in \llbracket 1,\overline{T}\rrbracket,\quad    M_j\subset \partial M_i \implies 0\leq \eta_{i}\gamma_j \leq \beta_i,
\end{equation*}
there exist $c_i,c_{1,i},c_{2,i},c_{3,i}>0$ and $\omega_{1,i},\omega_{2,i},\omega_{3,i}\in[0,\beta_i/4]$ for $i\in \llbracket 1,\overline{T}\rrbracket$ such that for any $\alpha\in (0,1]$, the following holds:
\begin{enumerate}
\item $P_{M_i}$ is $L$-Lipschitz continuous on
$\bigcup_{\alpha\in(0,1]}\mathcal{N}_0(i,\alpha)$ and is $C^{2}$ on an open
neighborhood of this set, where
\begin{equation}\label{eq:partition}
    \mathcal{N}_0(i,\alpha):= X\cap   B\left(M_i \setminus \bigcup_{j:M_j \subset \partial M_i} B(M_j, c_j \alpha^{\gamma_j}/2),2c_i\alpha^{\beta_i}\right).
\end{equation}
 \item For any $j\in\llbracket 1,\overline{T}\rrbracket$ such that $\overline{M_i}\cap M_j=M_i\cap\overline{M_j}=\emptyset$, we have
\[\bigcup _{\alpha\in(0,1]}\mathcal{N}_0(i,\alpha)\cap\bigcup_{\alpha\in(0,1]}\mathcal{N}_0(j,\alpha)=\emptyset.\]
\item It holds that   
\begin{align}
|\nabla_{M_i} f(x)-\nabla_{M_i} f(y)|
&\leq \frac{c_{1,i}}{\alpha ^{\omega_{1,i}}} |x - y|,
\quad \forall x,y\in \mathcal{N}_0(i,\alpha)\cap M_i,
\label{eq:Lip_1}\\
|P_{T_{M_i}(y)}(v)-\nabla_{M_i}f(y)|
&\leq \frac{c_{2,i}}{\alpha^{\omega_{2,i}}}|x-y|,
\quad \forall x\in \mathcal{N}_0(i,\alpha),\ y\in \mathcal{N}_0(i,\alpha)\cap M_i,\ v\in\partial f(x),
\label{eq:Lip_2}\\
|DP_{M_i}(x)-DP_{M_i}(y)|
&\leq \frac{c_{3,i}}{\alpha^{\omega_{3,i}}}|x-y|,
\quad \forall x,y\in \mathcal{N}_0(i,\alpha).
\label{eq:Lip_3}
\end{align}
\end{enumerate}
\end{proposition}

\begin{proof}[Proof of Theorem~\ref{thm:inexact_diameter}]
The case \(K=0\) is immediate, so assume that \(K\geq1\).
Choose a compact definable set
\(\widetilde X\subset\mathbb{R}^n\) such that $B(X,2)\subset \widetilde X$ and set $\tau_0 := \min\{1,\tau\}$.

\runinheading{Step~1: Uniform desingularizing function}
Apply \cite[Proposition~3.2]{lai2025diameter} with $m=3$ to choose a $C^3$
stratification \(\{M_1,\ldots,M_{\overline{T}}\}\) of \(\widetilde X\) such
that each restriction \(f|_{M_i}\) is $C^3$.  Proposition~\ref{proposition:neighborhoods of cells},
applied to this stratification with \(L=2\), gives the associated constants
\(\eta_i\geq1\).

Since there are finitely many strata, the signed \L{}ojasiewicz gradient
inequality applied to \(f|_{M_i}\) yields a number \(\bar\theta\in(0,1)\)
such that, for every fixed \(\theta\in(\bar\theta,1)\), there are
\(a_\theta,\varepsilon>0\) for which the signed power function
\(\psi:=a_\theta\varphi_\theta\)
is a desingularizing function of \(f|_{M_i}\) on \([|f|\leq 2\varepsilon]\)
for every \(i\in\llbracket1,\overline{T}\rrbracket\). Fix such a \(\theta\) for the rest
of the proof. We will verify the stratified descent condition (Definition \ref{def:stratified_descent}) in $X \cap[|f|\leq\varepsilon]$.

\runinheading{Step~2: Choice of exponents}
We will choose constants
\(\gamma_i,\beta_i>0\) so that, for all \(i,j\in\llbracket1,\overline{T}\rrbracket\),
\(0<\beta_i<\gamma_i(1-\theta)<\tau_0(1-\theta)\) and \(\beta_i<\xi\), and, whenever \(M_j\subset\partial M_i\), \(\gamma_j<\beta_i\) and \(\eta_i\gamma_j\leq\beta_i\).

Set \(d_*:=\max_i\dim M_i\) and \(\eta_*:=\max_i\eta_i\), and choose
\[
    \delta\in
    \left(0,
    \frac12\min\left\{
        \tau_0,\frac{\xi}{1-\theta},\frac{1-\theta}{\eta_*}
    \right\}
    \right).
\]
For each \(i\), define
\[
    \gamma_i:=\delta^{d_*-\dim M_i+1},
    \qquad
    a_i:=\max\bigl(\{0\}\cup
    \{\eta_i\gamma_j:M_j\subset\partial M_i\}\bigr),
    \qquad
    \beta_i:=\frac12\left(a_i+\gamma_i(1-\theta)\right).
\]
Since \(\delta \leq (1-\theta)/(2\eta_*)<1\), each \(\gamma_i>0\) and
\(\gamma_i\in(0,\delta]\). Hence \(\gamma_i(1-\theta)\leq \delta(1-\theta)< \min\{\tau_0(1-\theta),\xi\}\).
If \(M_j\subset\partial M_i\), then \(\dim M_j<\dim M_i\), and therefore $\gamma_j\leq\delta\gamma_i$. Consequently,
\(\eta_i\gamma_j\leq \eta_*\delta\gamma_i<\frac12\gamma_i(1-\theta)\).
Thus \(0\leq a_i<\frac12\gamma_i(1-\theta)\), and the definition of \(\beta_i\)
gives \(0<\beta_i<\gamma_i(1-\theta)<\xi\), and, whenever \(M_j\subset\partial M_i\), \(\gamma_j\leq\eta_i\gamma_j\leq a_i<\beta_i\).

With these exponents fixed, Proposition~\ref{proposition:neighborhoods of cells}
gives constants
\[
    c_i,c_{1,i},c_{2,i},c_{3,i}>0,
    \qquad
    \omega_{1,i},\omega_{2,i},\omega_{3,i}\in[0,\beta_i/4],
\]
and the corresponding sets
\[
    \mathcal{N}_0(i,\alpha):=
     \widetilde X\cap
    B\left(
        M_i\setminus
        \bigcup_{j:M_j\subset\partial M_i}
        B(M_j,c_j\alpha^{\gamma_j}/2),
        2c_i\alpha^{\beta_i}
    \right),
    \qquad \forall \alpha\in(0,1].
\]
Let
\[
    \omega_i:=\max\{\omega_{1,i},\omega_{2,i},\omega_{3,i}\},
    \qquad
    A:=\max\Bigl\{2,\max_i(c_i+c_{1,i}+c_{2,i}+c_{3,i})\Bigr\}.
\]
Increase \(A\) so that \(f\) is \(A\)-Lipschitz on \(\widetilde X\), and each
\(f|_{M_i}\) is \(A\)-Lipschitz on \(M_i\). Then
\eqref{eq:Lip_1}--\eqref{eq:Lip_3} hold with \(L_{f,k}=L_{V,k}=L_{P,k}:=A\alpha_k^{-\omega_i}\)
whenever the corresponding points lie in \(\mathcal{N}_0(i,\alpha_k)\).

Define
\[
    \beta_0:=\min_i\{\beta_i-\omega_i,1-2\omega_i\}>0,
    \qquad
    \underline\beta:=\min_i\beta_i,
    \qquad
    \beta:=\min\{\beta_0,\underline\beta(1-\theta)\}.
\]

\runinheading{Step~3: Local scale and neighborhoods}
For \(i\in\llbracket1,\overline{T}\rrbracket\), define
\[
    \mathcal{N}(i,\alpha):=
    (X\cap [|f|\leq\varepsilon])\cap
    \left(
        B(M_i,c_i\alpha^{\beta_i})
        \setminus
        \bigcup_{j:M_j\subset\partial M_i}
        B(M_j,c_j\alpha^{\gamma_j})
    \right).
\]
This is exactly the neighborhood form required by
Definition~\ref{def:stratified_descent}, with the ambient set \(X\cap [|f|\leq\varepsilon]\).

Since \(\beta_i<\xi\), \(\beta_i<1\), and
\(\beta_i>\gamma_j\) whenever \(M_j\subset\partial M_i\), choose
\(\hat{\alpha}_{\mathrm{loc}}\in(0,1]\) such that for all
\(\alpha\in(0,\hat{\alpha}_{\mathrm{loc}}]\),
\[
    A\alpha+c_b\alpha^\xi\leq \frac12 c_i\alpha^{\beta_i},
\]
and
\[
    2c_i\alpha^{\beta_i}+A^2\alpha+A\alpha^\tau+c_b\alpha^\xi
    \leq \frac12 c_j\alpha^{\gamma_j}
\]
for all relevant \(i,j\) with \(M_j\subset\partial M_i\). We also impose
\[
    A c_i\alpha^{\beta_i}\leq\varepsilon,
    \qquad
    A\alpha+c_b\alpha^\xi\leq 1.
\]

We record the resulting containments. Let
\(\alpha\in(0,\hat{\alpha}_{\mathrm{loc}}]\) and \(x\in\mathcal{N}(i,\alpha)\).
Choose \(y\in M_i\) with \(|x-y|\leq c_i\alpha^{\beta_i}\). For every
\(j\) with \(M_j\subset\partial M_i\), since
\(x\notin B(M_j,c_j\alpha^{\gamma_j})\), we have
\[
    d(y,M_j)
    \geq d(x,M_j)-|x-y|
    > c_j\alpha^{\gamma_j}-c_i\alpha^{\beta_i}
    \geq \frac34 c_j\alpha^{\gamma_j}
    > \frac12 c_j\alpha^{\gamma_j}.
\]
Thus \(y\in M_i\setminus
\bigcup_{j:M_j\subset\partial M_i}B(M_j,c_j\alpha^{\gamma_j}/2)\),
and \(x\) lies within \(c_i\alpha^{\beta_i}\leq2c_i\alpha^{\beta_i}\) of this
set. Since \(x\in X\subset\widetilde X\), it follows that
\(\mathcal{N}(i,\alpha)\subset \mathcal{N}_0(i,\alpha)\).

Next suppose \(x_k\in\mathcal{N}(i,\alpha_k)\) with
\(\alpha_k\leq\hat{\alpha}_{\mathrm{loc}}\), and set \(\alpha=\alpha_k\). The above
choice of \(y\in M_i\) also satisfies \(d(y,M_j)>c_j\alpha^{\gamma_j}/2\)
for every \(j\) with \(M_j\subset\partial M_i\). If
\(z\in B(x_k,A\alpha)\cup B(x_k,c_b\alpha^\xi)\),
then
\[
    |z-y|
    \leq c_i\alpha^{\beta_i}+\max\{A\alpha,c_b\alpha^\xi\}
    \leq \frac32c_i\alpha^{\beta_i}
    \leq 2c_i\alpha^{\beta_i}.
\]
Moreover, since \(x_k\in X\) and \(|z-x_k|\leq A\alpha+c_b\alpha^\xi\leq1\),
we have $z\in B(X,1)\subset\widetilde X$. Consequently,
\[
    B(x_k,A\alpha_k)\cup B(x_k,c_b\alpha_k^\xi)
    \subset \mathcal{N}_0(i,\alpha_k)
    \subset \bigcup_{\alpha\in(0,1]}\mathcal{N}_0(i,\alpha).
\]

Finally, since \(x_k\in\mathcal{N}_0(i,\alpha_k)\), the projection
\(P_{M_i}(x_k)\) is well-defined. For
\(w\in B(P_{M_i}(x_k),A^2\alpha_k+A|e_k|)\cap M_i\), we have
\begin{align*}
    d(w,M_j)
    &\geq d(x_k,M_j)-|x_k-P_{M_i}(x_k)|-|w-P_{M_i}(x_k)|\\
    &> c_j\alpha_k^{\gamma_j}-c_i\alpha_k^{\beta_i}
       -A^2\alpha_k-A|e_k|\\
    &\geq c_j\alpha_k^{\gamma_j}-c_i\alpha_k^{\beta_i}
       -A^2\alpha_k-A\alpha_k^\tau\\
    &\geq
    \frac12 c_j\alpha_k^{\gamma_j}
    +c_i\alpha_k^{\beta_i}
    +c_b\alpha_k^\xi
    >\frac12 c_j\alpha_k^{\gamma_j}
\end{align*}
for every \(j\) with \(M_j\subset\partial M_i\). Hence \(w\) belongs to the
central stratum set appearing in the definition of \(\mathcal{N}_0(i,\alpha_k)\),
and therefore
\[
    B(P_{M_i}(x_k),A^2\alpha_k+A|e_k|)\cap M_i
    \subset \mathcal{N}_0(i,\alpha_k)\cap M_i.
\]

Since \(\hat\alpha_{\mathrm{loc}}\leq1\), the inclusion
\(\mathcal{N}(i,\alpha)\subset \mathcal{N}_0(i,\alpha)\) for
\(0<\alpha\leq\hat\alpha_{\mathrm{loc}}\), together with
Proposition~\ref{proposition:neighborhoods of cells}, gives the separation
condition in Definition~\ref{def:stratified_descent} with threshold
\(\hat\alpha_{\mathrm{loc}}\).

\runinheading{Step~4: Error functions and small steps}
The estimates in this step apply to any realization
on the block with \(x_{\llbracket0,K\rrbracket}\subset
X\cap[|f|\leq\varepsilon]\) and \(\alpha_k\leq\hat{\alpha}_{\mathrm{loc}}\) for
\(k\in\llbracket0,K-1\rrbracket\). The final upper step-size bound
\(\hat{\alpha}\) in the theorem statement will be fixed in Step~6.

For a constant \(Q>0\) to be fixed below, and for
\(a=(a_m)_{m\in\mathbb N}\in(\mathbb R_+)^{\mathbb N}\) and
\(e=(e_m)_{m\in\mathbb N}\in(\mathbb R^n)^{\mathbb N}\), define
\[
\mathfrak g_1(a,e):=Q\sum_{m=0}^{\infty}(a_m^{1+\beta}+|e_m|),
\qquad
\mathfrak g_2(a,e):=Q\bigl(a_0^{1+\beta}
+a_0\mathfrak g_1(a,e)^\theta+|e_0|\bigr).
\]
For the present finite block, extend
\(\alpha_{\llbracket0,K-1\rrbracket}\) and
\(e_{\llbracket0,K-1\rrbracket}\) by zero at every index \(m\geq K\), and
let \(g_{1,k}\) and \(g_{2,k}\) be the corresponding tail evaluations of
\(\mathfrak g_1\) and \(\mathfrak g_2\). Writing
\(S_k:=\sum_{j=k}^{K-1}\alpha_j^{1+\beta}\) and
\(E_k:=\sum_{j=k}^{K-1}|e_j|\) gives
\begin{align}
g_{1,k}&=Q(S_k+E_k),
&&k\in\llbracket0,K\rrbracket,
\label{eq:g1}\\
g_{2,k}&=Q\bigl(\alpha_k^{1+\beta}
+\alpha_k g_{1,k}^{\theta}+|e_k|\bigr),
&&k\in\llbracket0,K-1\rrbracket.
\label{eq:g2}
\end{align}
Also set \(\widetilde\psi:=2A\psi\), the power function used in Definition~\ref{def:stratified_descent}.

If \(x_k\in X\cap[|f|\leq\varepsilon]\) and
\(\alpha_k\leq\hat{\alpha}_{\mathrm{loc}}\), then
\(c_b\alpha_k^\xi\leq1\) by the choice of \(\hat{\alpha}_{\mathrm{loc}}\), and hence
\[
    B(x_k,c_b\alpha_k^\xi)
    \subset B(X,1)
    \subset \mathring{\widetilde X} .
\]
Since $f$ is $A$-Lipschitz on $\widetilde X$, the Clarke subgradient bound
\cite[Proposition~2.1.2]{clarke1990} gives
\begin{equation}
    \partial f(x_k;c_b\alpha_k^\xi)\subset B(0,A).
\label{eq:inexact_uniform_goldstein_bound}
\end{equation}
Consequently,
\[
    |x_{k+1}-x_k|
    \leq A\alpha_k+|e_k|
    \leq A\alpha_k+\alpha_k^\tau
    \leq (A+1)\alpha_k^{\tau_0}.
\]
Moreover, for \(k\in\llbracket0,K-1\rrbracket\),
\[
    g_{1,k}-g_{1,k+1}
    =
    Q(\alpha_k^{1+\beta}+|e_k|)
    \leq Q(\alpha_k^{1+\beta}+\alpha_k^\tau)
    \leq 2Q\alpha_k^{\tau_0},
\]
because \(\hat{\alpha}_{\mathrm{loc}}\leq1\) and
\(\tau_0=\min\{1,\tau\}\) give
\(\alpha_k^{1+\beta}\leq\alpha_k^{\tau_0}\) and
\(\alpha_k^\tau\leq\alpha_k^{\tau_0}\).
Thus item~(i) of Definition~\ref{def:stratified_descent} holds, after
increasing \(\iota\), for the zero-extended tails defined in Step~4.

\runinheading{Step~5: Projected descent on one stratum}
Fix \(i\in\llbracket1,\overline{T}\rrbracket\), and suppose
\(x_k\in\mathcal{N}(i,\alpha_k)\) for
\(k\in\llbracket0,K\rrbracket\).
Let \(y_k^i:=P_{M_i}(x_k)\) and \(d_k^i:=d(x_k,M_i)\).
Then
\[
    d_k^i\leq c_i\alpha_k^{\beta_i},
    \qquad
    |f(y_k^i)|
    \leq |f(x_k)|+A d_k^i
    \leq 2\varepsilon.
\]
The containments from Step~3 and
\eqref{eq:inexact_uniform_goldstein_bound} verify all assumptions of
Lemma~\ref{lemma:inexact_projected_length} with \(M=M_i\), \(L=A\), with the
relevant neighborhoods \(\mathcal{N}_0(i,\alpha_k)\), and with the
desingularizing function \(\psi\) on \([|f|\leq2\varepsilon]\).

We next check the \(g\)-condition in
Lemma~\ref{lemma:inexact_projected_length}. Using \(L_{f,k}=L_{V,k}=L_{P,k}=A\alpha_k^{-\omega_i}\) and \(d_k^i\leq c_i\alpha_k^{\beta_i}\),
the required lower bound is controlled by a constant multiple of
\[
    \alpha_k^{1+2\beta_i-2\omega_i}
    +\alpha_k^{1+2\xi-2\omega_i}
    +\alpha_k^{1+\beta_i-\omega_i}
    +\alpha_k^{2-2\omega_i}
    +|e_k|.
\]
By the definition of \(\beta\), this is bounded by $ C(\alpha_k^{1+\beta}+|e_k|)$. Hence \(g_k:=g_{1,k}\) satisfies the lemma's \(g\)-condition once \(Q\) is
chosen large enough.

The additional summand in \eqref{eq:y_k-error} caused by the perturbation is
$A^3\alpha_k^{1-\omega_i}|e_k|\leq A^3|e_k|$, because
$\omega_i<1$ and $\alpha_k\leq1$.  Thus it is absorbed by the
$|e_k|$ summand in \eqref{eq:y_k-error}.

Applying Lemma~\ref{lemma:inexact_projected_length} gives, with $ z_k^i:=f(y_k^i)+g_{1,k}$, that
\begin{align*}
\sum_{k=0}^{K-1}|y_{k+1}^i-y_k^i|
&\leq\widetilde\psi(z_0^i)-\widetilde\psi(z_K^i)+
C\sum_{k=0}^{K-1}
\left(
    \alpha_k^{1+\beta}
    +\alpha_k g_{1,k}^{\theta}
    +|e_k|
\right)
+
C\max_{k\in\llbracket0,K-1\rrbracket}\alpha_k .
\end{align*}
Since \(\alpha_k\leq1\) and \(\tau_0\leq1\), the last maximum is bounded by \(C\max_{k\in\llbracket0,K-1\rrbracket}\alpha_k^{\tau_0}\).
Enlarging \(Q\) and \(\iota\), and using \eqref{eq:g2}, we obtain
\[
\sum_{k=0}^{K-1}|y_{k+1}^i-y_k^i|
\leq
\widetilde\psi(z_0^i)-\widetilde\psi(z_K^i)
+\sum_{k=0}^{K-1}g_{2,k}
+\iota\max_{k\in\llbracket0,K-1\rrbracket}\alpha_k^{\tau_0}.
\]
This is precisely item~(ii) of Definition~\ref{def:stratified_descent} for
the zero-extended tails defined in Step~4.

\runinheading{Step~6: Apply Theorem~\ref{thm:diameter}}
Steps~3--5 verify the stratified descent condition needed by
Theorem~\ref{thm:diameter}, with ambient set \(X\cap [|f|\leq\varepsilon]\), neighborhoods
\(\mathcal{N}(i,\alpha)\), power function
\(\widetilde\psi=2A\psi=B_\theta\varphi_\theta\), where
\(B_\theta:=2Aa_\theta>0\),
small-step exponent \(\tau_0\), and the error functions
\(\mathfrak g_1,\mathfrak g_2\) evaluated on the zero extension of the data
after \(K\). Let \(\hat{\alpha}_{\mathrm{diam}},C_{\mathrm{diam}}>0\) be the upper
step-size bound and the constant supplied by Theorem~\ref{thm:diameter} for
these choices. Set \(\hat{\alpha}:=\min\{\hat{\alpha}_{\mathrm{loc}},\hat{\alpha}_{\mathrm{diam}}\}\).
Applying Theorem~\ref{thm:diameter} to
\(x_{\llbracket0,K\rrbracket}\subset X\cap [|f|\leq\varepsilon]\), with the
zero-extended tails defined in Step~4, gives
\begin{align*}
\operatorname{diam}(x_{\llbracket0,K\rrbracket})
&\leq2B_\theta\Big(
\varphi_\theta(f(x_0))-\varphi_\theta(f(x_K))
\Big)  \\
&+
4B_\theta\left(g_{1,0}^{1-\theta}+g_{1,K}^{1-\theta}\right)
+2\sum_{k=0}^{K-1}g_{2,k}
+C_{\mathrm{diam}}\alpha_0^{\underline\beta(1-\theta)} .
\end{align*}
The zero extension gives \(g_{1,K}=0\) and \(g_{1,0}=Q(S_0+E_0)\). By
the subadditivity of \(t\mapsto t^q\) for \(q\in(0,1)\),
\begin{align*}
g_{1,0}^{1-\theta}+g_{1,K}^{1-\theta}
&\leq Q^{1-\theta}\bigl(S_0^{1-\theta}+E_0^{1-\theta}\bigr),\\
\sum_{k=0}^{K-1}g_{2,k}
&\leq Q(S_0+E_0)
+Q^{1+\theta}\sum_{k=0}^{K-1}\alpha_k
\bigl(S_k^\theta+E_k^\theta\bigr).
\end{align*}
Moreover, \(\beta\leq\underline\beta(1-\theta)\) and \(\alpha_0\leq1\)
give \(\alpha_0^{\underline\beta(1-\theta)}\leq\alpha_0^\beta\).
Substituting these estimates above and absorbing the fixed coefficients into
positive constants \(\varsigma_1:=2B_\theta\) and \(\varsigma_2\) yield
\begin{align*}
\operatorname{diam}(x_{\llbracket0,K\rrbracket})
&\leq\varsigma_1\bigl(
\varphi_\theta(f(x_0))-\varphi_\theta(f(x_K))\bigr)\\
&\quad+\varsigma_2\left(
\alpha_0^\beta
+S_0+S_0^{1-\theta}
+\sum_{k=0}^{K-1}\alpha_k S_k^\theta
\right)\\
&\quad+\varsigma_2\left(
E_0+E_0^{1-\theta}
+\sum_{k=0}^{K-1}\alpha_k E_k^\theta
\right).
\end{align*}
The constants are independent of the realization and \(K\), so this is the
claimed estimate.
\end{proof}

\section{Auxiliary estimates for the stochastic subgradient method}
\label{sec:stochastic_auxiliary_estimates}

This appendix collects the auxiliary arguments used in the stochastic
subgradient analyses of Sections~\ref{sec:stochastic} and
\ref{sec:active_stochastic}.  We first prove the polynomial-window estimate
used in the proof of Theorem~\ref{theorem:1overk}.  We then state and prove
the martingale and quadratic-noise tail estimates used in the active analysis
before turning to the sparse-index and exit-interval bounds.

\subsection{Proof of Fact~\ref{fact:st}}
\label{sec:proof_fact_st}
Choose constants \(0<c_1\leq c_2\) such that, for all large \(k\),
\(c_1(k+1)^{-a}\leq \alpha_k\leq c_2(k+1)^{-a}\).
By the definition of \(s_{t+1}\) in \eqref{eq:def_st}, we have
\begin{equation}\label{eq:st-window-size}
    \sum_{i=s_t}^{s_{t+1}-1}\alpha_i
    \leq \alpha_{s_t}^{\zeta}
    <
    \sum_{i=s_t}^{s_{t+1}}\alpha_i .
\end{equation}
\runinheading{Step~1: Window length}
Set \(h_t:=s_{t+1}-s_t\). We first prove
\(h_t=\Theta((s_t+1)^{a(1-\zeta)})\).

For the lower bound, the right inequality in
\eqref{eq:st-window-size} gives
\(\alpha_{s_t}^{\zeta}<\sum_{i=s_t}^{s_{t+1}}\alpha_i
\leq c_2(h_t+1)(s_t+1)^{-a}\). Since
\(\alpha_{s_t}\geq c_1(s_t+1)^{-a}\), it follows that
\(h_t+1\geq (c_1^\zeta/c_2)(s_t+1)^{a(1-\zeta)}\). Thus
\(h_t\geq (c_1^\zeta/(2c_2))(s_t+1)^{a(1-\zeta)}\) for all sufficiently
large \(t\).

For the upper bound, we first show that \(h_t\leq s_t+1\) for all large
\(t\). Suppose otherwise that \(h_t>s_t+1\) along an infinite subsequence.
Then
\(\sum_{i=s_t}^{s_t+s_t}\alpha_i\leq
\sum_{i=s_t}^{s_{t+1}-1}\alpha_i\leq\alpha_{s_t}^{\zeta}\).
However, the left-hand side is bounded from below by
\(c_1 2^{-a}(s_t+1)^{1-a}\), while
\(\alpha_{s_t}^{\zeta}=O((s_t+1)^{-a\zeta})\). This is impossible for large
\(t\), because \(1-a+a\zeta=1-a(1-\zeta)>0\).
Hence \(h_t\leq s_t+1\) eventually. On this eventual tail, for every
\(i\in\llbracket s_t,s_{t+1}-1\rrbracket\), we have
\((i+1)^{-a}\geq 2^{-a}(s_t+1)^{-a}\).
Therefore the left inequality in \eqref{eq:st-window-size} yields
\[
    c_1 2^{-a}h_t(s_t+1)^{-a}
    \leq
    \sum_{i=s_t}^{s_{t+1}-1}\alpha_i
    \leq
    \alpha_{s_t}^{\zeta}
    \leq
    c_2^\zeta(s_t+1)^{-a\zeta}.
\]
Thus \(h_t=O((s_t+1)^{a(1-\zeta)})\). Combining the two bounds proves the
claimed estimate. In particular, since \(a(1-\zeta)<1\), we have
\(h_t/(s_t+1)\to0\).

\runinheading{Step~2: Window endpoints}
Set \(\lambda:=1-a(1-\zeta)>0\).
By the mean value theorem applied to \(r\mapsto r^\lambda\),
\((s_{t+1}+1)^\lambda-(s_t+1)^\lambda=\lambda r_t^{\lambda-1}h_t\)
for some \(r_t\in(s_t+1,s_{t+1}+1)\). Since
\(h_t/(s_t+1)\to0\), we have \(r_t=\Theta(s_t+1)\). Using
\(h_t=\Theta((s_t+1)^{a(1-\zeta)})\), we get
\((s_{t+1}+1)^\lambda-(s_t+1)^\lambda
=\Theta((s_t+1)^{\lambda-1+a(1-\zeta)})=\Theta(1)\). Summing over \(t\)
gives \((s_t+1)^\lambda=\Theta(t+1)\), and hence
\(s_t=\Theta((t+1)^{1/\lambda})\).

\runinheading{Step~3: Aggregated step sizes}
It remains to estimate \(a_t\). By \eqref{eq:def_at} and
\eqref{eq:st-window-size},
\(a_t\leq \alpha_{s_t}^{\zeta}\) and
\(a_t>\alpha_{s_t}^{\zeta}-\alpha_{s_{t+1}}\). Since
\(s_{t+1}/s_t\to1\), we have
\(\alpha_{s_{t+1}}=\Theta(\alpha_{s_t})\).
Because \(\zeta<1\) and \(\alpha_{s_t}\to0\),
\(\alpha_{s_{t+1}}=O(\alpha_{s_t})=o(\alpha_{s_t}^{\zeta})\). Therefore
\(a_t=\Theta(\alpha_{s_t}^{\zeta})\). Combining this with
\(\alpha_{s_t}=\Theta((s_t+1)^{-a})\) and
\(s_t=\Theta((t+1)^{1/\lambda})\), we obtain
\(a_t=\Theta((t+1)^{-a\zeta/\lambda})\). Equivalently, if
\(b:=a\zeta/(1-a(1-\zeta))\), then \(a_t=\Theta((t+1)^{-b})\).
This proves Fact~\ref{fact:st}. \qed

\subsection{Martingale and quadratic-noise tail estimates}
\label{sec:martingale_tail_estimates}

The next lemma isolates the two probabilistic estimates needed for the exit
intervals in Section~\ref{sec:active_stochastic_exit_estimates}.  Its first
conclusion controls martingale tails on the event
\(\sum_k\alpha_k|w_k|^2<\infty\); its second controls the quadratic-noise
tail.

\begin{lemma}[Noise-tail estimates]
\label{lemma:martingale_tail_summability}
Let
\((\Omega,\mathcal{F},(\mathcal{F}_k)_{k\in\mathbb{N}},\mathbb{P})\)
be a filtered probability space.  Let \(\varepsilon_k\) be
\(\mathcal{F}_{k+1}\)-measurable and suppose that, for some \(\sigma>0\),
\(\mathbb{E}[\varepsilon_k\mid\mathcal{F}_k]=0\) and
\(\mathbb{E}[|\varepsilon_k|^2\mid\mathcal{F}_k]\leq \sigma^2\) almost
surely.  Let \((\alpha_k)\) be a deterministic sequence of positive numbers,
and let \((w_k)\) be an \(\mathcal{F}_k\)-measurable sequence with values in
the same Euclidean space as \(\varepsilon_k\).  Define
\(R_k^w:=\sup_{r\geq k}
\left|\sum_{i=k}^{r}\alpha_i\langle w_i,\varepsilon_i\rangle\right|\).
Then the following statements hold.
\begin{enumerate}
    \item Fix \(\theta\in(0,1)\).  If
    \(\alpha_k=\Theta((k+1)^{-a})\) and
    \(a(1+\theta/2)>1\), then
    \begin{equation}
    \mathbb P\left(
      \left\{\sum_{k=0}^{\infty}\alpha_k|w_k|^2<\infty\right\}
      \cap
      \left\{\sum_{k=0}^{\infty}\alpha_k(R_k^w)^\theta=\infty\right\}
    \right)=0.
    \label{eq:martingale_tail_localized}
    \end{equation}
    In particular, if
    \(\sum_k\alpha_k|w_k|^2<\infty\) almost surely, then
    \(\sum_k\alpha_k(R_k^w)^\theta<\infty\) almost surely.

    \item Fix \(\theta\in(0,1)\), suppose that
    \(\alpha_k=\Theta((k+1)^{-a})\) with \(a>1/2\), and define
    \begin{equation}
        D_k:=\sum_{i=k}^{\infty}
        \alpha_i^2(1+|\varepsilon_i|^2).
        \label{eq:quadratic_noise_tail}
    \end{equation}
    Then \(D_0<\infty\) almost surely and \(D_k\downarrow0\) almost surely.
    If, in addition, \(a+\theta(2a-1)>1\), then
    \begin{equation}
        \sum_{k=0}^{\infty}\alpha_kD_k^\theta<\infty
        \qquad\text{almost surely}.
        \label{eq:quadratic_noise_tail_summability}
    \end{equation}
\end{enumerate}
\end{lemma}

\begin{proof}
\runinheading{Item~(1): Martingale tails}
Fix \(m\geq1\), set
\(V_i:=\sum_{j=0}^{i}\alpha_j|w_j|^2\), and stop the coefficient at
\(\tau_m:=\inf\{i:V_i>m\}\) by defining
\(w_i^{(m)}:=w_i\mathbf1_{\{i<\tau_m\}}\).  Then \(w_i^{(m)}\) is
\(\mathcal F_i\)-measurable and
\(\sum_i\alpha_i|w_i^{(m)}|^2\leq m\) almost surely.

For \(k\geq0\), let
\(R_{k,m}:=\sup_{r\geq k}|\sum_{i=k}^{r}
\alpha_i\langle w_i^{(m)},\varepsilon_i\rangle|\).  The martingale
maximal inequality \cite[Theorem~1.1]{burkholder1972integral}, applied with
the convex function \(t\mapsto t^2\), gives a fixed \(C_{\rm M}>0\) such that
\[
 \mathbb E[R_{k,m}^2]
 \leq C_{\rm M}\sigma^2\mathbb E\!\left[
       \sum_{i=k}^{\infty}\alpha_i^2|w_i^{(m)}|^2\right].
\]
The polynomial bounds give a fixed \(C_\alpha\geq1\) such that
\(\alpha_i\leq C_\alpha\alpha_k\) whenever \(i\geq k\), and hence
\(\mathbb E[R_{k,m}^2]\leq C_{\rm M}\sigma^2C_\alpha m\alpha_k\).
Since \(t\mapsto t^{\theta/2}\) is concave, Jensen's inequality gives
\(\mathbb E[R_{k,m}^{\theta}]
\leq\bigl(\mathbb E[R_{k,m}^2]\bigr)^{\theta/2}\).  Tonelli's theorem
\cite[Theorem~2.37]{folland1999real} therefore yields
\[
 \mathbb E\!\left[\sum_{k=0}^{\infty}
       \alpha_kR_{k,m}^{\theta}\right]
 \leq(C_{\rm M}\sigma^2C_\alpha m)^{\theta/2}
       \sum_{k=0}^{\infty}\alpha_k^{1+\theta/2}<\infty.
\]
Here the final series converges because \(a(1+\theta/2)>1\).

Let \(A_m:=\{\sum_i\alpha_i|w_i|^2\leq m\}\).  On \(A_m\), the stopping
time is infinite and \(R_{k,m}=R_k^w\) for every \(k\).  Thus
\[
 \mathbb P\!\left(A_m\cap
 \left\{\sum_{k=0}^{\infty}\alpha_k(R_k^w)^\theta=\infty\right\}
 \right)=0.
\]
Taking the union over \(m\geq1\) proves
\eqref{eq:martingale_tail_localized} and its stated consequence.

\runinheading{Item~(2): Quadratic-noise tails}
The tower property and Tonelli's theorem
\cite[Theorem~2.37]{folland1999real} give
\[
 \mathbb E[D_k]\leq(1+\sigma^2)\sum_{i=k}^{\infty}\alpha_i^2
 \leq C_{\rm Q}(k+1)^{1-2a}
\]
for a fixed \(C_{\rm Q}\geq1\).
In particular, \(D_0<\infty\) almost surely and \(D_k\downarrow0\).
Jensen's inequality and Tonelli's theorem now give
\[
 \mathbb E\!\left[\sum_{k=0}^{\infty}\alpha_kD_k^\theta\right]
 \leq C_{\rm Q}\sum_{k=0}^{\infty}
 (k+1)^{-a-\theta(2a-1)}<\infty,
\]
where the last inequality is equivalent to
\(a+\theta(2a-1)>1\).  This proves
\eqref{eq:quadratic_noise_tail_summability}.
\end{proof}

\subsection{Proof of Lemma~\ref{lemma:active_stochastic_sparse}}
\label{sec:proof_active_stochastic_sparse}

\begin{proof}
\runinheading{Step~1: Select the deterministic indices}
Let
\[
 \mathcal A:=\left\{\sup_k|x_k|<\infty\right\}
 \cap\left\{\liminf_{k\to\infty}|x_k-x^*|=0\right\}
\]
be the event in \(\mathcal F\) that the realization is bounded and has
\(x^*\) as an accumulation point.  For \(N\in\mathbb N\), set
\[
 F_N:=\sup_{k\geq N}|f(x_k)-f(x^*)|,
 \qquad
 E_N:=\sum_{k=N}^{\infty}\alpha_k^2.
\]
Equation~\eqref{eq:stoch_value_convergence} gives \(F_N\to0\) almost surely
on \(\mathcal A\).  Dominated convergence
\cite[Theorem~2.24]{folland1999real} gives
\(\mathbb E[\mathbf1_{\mathcal A}\min\{1,F_N\}]\to0\), while
\eqref{eq:stoch_step_consequences} gives \(E_N\downarrow0\).  We may
therefore choose a deterministic sequence \(K_1<K_2<\cdots\) such that
\[
 \mathbb E\!\left[
   \mathbf1_{\mathcal A}\min\{1,F_{K_j}\}\right]
   +E_{K_j}^{1/4}\leq2^{-j},
 \qquad j\geq1.
\]
This choice uses neither \(\rho\) nor the exit intervals.

\runinheading{Step~2: Control the stochastic tails}
Using the sequence selected in Step~1, define \(p_j,q_j\) by
\eqref{eq:stoch_excursions}, \(\eta_k\) by
\eqref{eq:stoch_excursion_indicator}, and \(n_k\) by
\eqref{eq:stoch_normal_coefficient}.
Since \(n_k\) is \(\mathcal F_k\)-measurable and \(|n_k|\leq\rho\),
Assumption~\ref{assumption:stochastic_noise} gives
\[
 \mathbb E[\alpha_k\langle\varepsilon_k,n_k\rangle
   \mid\mathcal F_k]=0,
 \qquad
 \mathbb E[\alpha_k^2\langle\varepsilon_k,n_k\rangle^2
   \mid\mathcal F_k]
 \leq\sigma^2\rho^2\alpha_k^2.
\]
The martingale maximal inequality
\cite[Theorem~1.1]{burkholder1972integral} and monotone convergence
\cite[Theorem~2.14]{folland1999real} give a fixed \(C_{\rm S}>0\) such that
\[
 \mathbb E\!\left[
   \sup_{m\geq N}\left|
     \sum_{k=N}^{m-1}\alpha_k
       \langle\varepsilon_k,n_k\rangle
   \right|^2\right]
 \leq C_{\rm S}\sigma^2\rho^2E_N.
\]
The triangle inequality and Jensen's inequality for the concave function
\(t\mapsto t^{1/4}\) give
\[
\begin{aligned}
\mathbb E\!\left[\left(
   \sup_{s\geq r\geq N}\left|
     \sum_{k=r}^{s-1}\alpha_k
       \langle\varepsilon_k,n_k\rangle
   \right|\right)^{1/2}\right]
&=\mathbb E\!\left[\left(
   \sup_{s\geq r\geq N}\left|
     \sum_{k=N}^{s-1}\alpha_k\langle\varepsilon_k,n_k\rangle
     -\sum_{k=N}^{r-1}\alpha_k\langle\varepsilon_k,n_k\rangle
   \right|\right)^{1/2}\right]\\
&\leq\sqrt{2}\,\mathbb E\!\left[\left(
   \sup_{m\geq N}\left|
     \sum_{k=N}^{m-1}\alpha_k\langle\varepsilon_k,n_k\rangle
   \right|^2\right)^{1/4}\right]\\
&\leq\sqrt{2}\left(\mathbb E\!\left[
   \sup_{m\geq N}\left|
     \sum_{k=N}^{m-1}\alpha_k\langle\varepsilon_k,n_k\rangle
   \right|^2\right]\right)^{1/4}\\
&\leq \sqrt{2}\,C_{\rm S}^{1/4}(\sigma\rho)^{1/2}E_N^{1/4}.
\end{aligned}
\]
This estimate is uniform over every deterministic choice of
\((K_j)\).

The tower property and Jensen's inequality give
\[
 \mathbb E\!\left[\left(
   \sum_{k=N}^{\infty}\alpha_k^2
     (1+|\varepsilon_k|^2)\right)^{1/2}\right]
 \leq(1+\sigma^2)^{1/2}E_N^{1/2}.
\]
The choice of \((K_j)\) gives
\(E_{K_j}^{1/4}\leq2^{-j}\leq1\), and hence
\(E_{K_j}^{1/2}\leq E_{K_j}^{1/4}\leq2^{-j}\).  Taking \(N=K_j\) in
the preceding two estimates therefore gives
\[
\begin{aligned}
\mathbb E\!\left[\left(
   \sup_{s\geq r\geq K_j}\left|
     \sum_{k=r}^{s-1}\alpha_k
       \langle\varepsilon_k,n_k\rangle
   \right|\right)^{1/2} +\left(
   \sum_{k=K_j}^{\infty}\alpha_k^2
 (1+|\varepsilon_k|^2)\right)^{1/2}\right]
 &\leq \sqrt{2}\,C_{\rm S}^{1/4}(\sigma\rho)^{1/2}E_{K_j}^{1/4} +(1+\sigma^2)^{1/2}E_{K_j}^{1/2}\\
&\leq\bigl(\sqrt{2}\,C_{\rm S}^{1/4}(\sigma\rho)^{1/2}
 +(1+\sigma^2)^{1/2}\bigr)2^{-j}.
\end{aligned}
\]
Summing this estimate and applying Tonelli's theorem
\cite[Theorem~2.37]{folland1999real} to the nonnegative summands shows that
the expectation of the following series is finite.  A nonnegative random
variable with finite expectation is finite almost surely; hence
\[
 \sum_{j=1}^{\infty}\left[\left(
   \sup_{s\geq r\geq K_j}\left|
     \sum_{k=r}^{s-1}\alpha_k
       \langle\varepsilon_k,n_k\rangle
   \right|\right)^{1/2}
   +\left(\sum_{k=K_j}^{\infty}\alpha_k^2
       (1+|\varepsilon_k|^2)\right)^{1/2}\right]<\infty
 \qquad\text{almost surely}.
\]

\runinheading{Step~3: Control the function-value tail}
Likewise,
\[
 \mathbb E\!\left[
  \mathbf1_{\mathcal A}
  \sum_{j=1}^{\infty}\min\{1,F_{K_j}\}\right]
 =\sum_{j=1}^{\infty}\mathbb E\!\left[
   \mathbf1_{\mathcal A}\min\{1,F_{K_j}\}\right]<\infty.
\]
Thus \(\sum_j\min\{1,F_{K_j}\}<\infty\) almost surely on
\(\mathcal A\).  Since \(F_{K_j}\to0\) there, the truncation is eventually
inactive, and \(\sum_jF_{K_j}<\infty\).  Combining the three conclusions
proves \eqref{eq:stoch_sparse_tails}.

\end{proof}

\subsection{Proof of Lemma~\ref{lemma:active_stochastic_normal}}
\label{sec:proof_active_stochastic_normal}

\begin{proof}
Fix a bounded realization in \(\Omega_{\mathrm{act}}\) for which the
conclusion of Lemma~\ref{lemma:active_stochastic_sparse} holds, having
\(x^*\) as an accumulation point and leaving
\(\mathring B(x^*,\rho)\) infinitely often.
Let \(j_0\) and the late exit intervals be as in
Fact~\ref{fact:active_stochastic_exit_localization}.

Increase this \(j_0\), if necessary, so that the late exit intervals satisfy
\(\alpha_k\leq\bar\alpha\), where \(\bar\alpha\) is supplied by
Proposition~\ref{proposition:active_inexact_local_estimates}.  The definition
of \(q_j\) and \eqref{eq:stoch_exit_endpoint_localization} give
\(x_{\llbracket p_j,q_j\rrbracket}\subset
\mathring B(x^*,\varrho)\), as required by the proposition.  Applying
\eqref{eq:active_meta_normal} with \(u_k=v_k\) and \(r_k=0\) therefore gives
\[
 d_{k+1}^2+\mu\alpha_kd_k
 \leq d_k^2-2\alpha_k\langle\varepsilon_k,x_k-y_k\rangle
       +C\alpha_k^2(1+|\varepsilon_k|^2)
\]
for every update index in a late exit interval.  Because
\(n_k=x_k-y_k\) on these intervals, summing this estimate from \(p_j\) to
\(s-1\), where \(p_j\leq s\leq q_j\), yields
\begin{align}
 d_s^2+\mu\sum_{k=p_j}^{s-1}\alpha_kd_k
 &\leq d_{p_j}^2
 -2\sum_{k=p_j}^{s-1}
       \alpha_k\langle\varepsilon_k,n_k\rangle
 +C\sum_{k=p_j}^{s-1}\alpha_k^2(1+|\varepsilon_k|^2) \notag\\
 &\leq C\left(
 4^{-j-1}\rho^2
 +\sup_{s\geq r\geq K_j}\left|
  \sum_{k=r}^{s-1}
   \alpha_k\langle\varepsilon_k,n_k\rangle
 \right|
 +\sum_{k=K_j}^{\infty}\alpha_k^2(1+|\varepsilon_k|^2)
 \right).
\end{align}
Here \eqref{eq:stoch_excursions} gives \(p_j\geq K_j\), while
\eqref{eq:stoch_excursions} and
\eqref{eq:active_projection_distance_bound} give
\(d_{p_j}\leq|x_{p_j}-x^*|\leq2^{-j-1}\rho\).  The two tail terms
bound the scalar martingale block and the quadratic-noise sum, respectively.
Since the right-hand side is independent of \(s\), dropping the nonnegative
sum and taking the supremum bounds
\(\sup_{p_j\leq s\leq q_j}d_s^2\).  Setting \(s=q_j\) and discarding the
nonnegative term \(d_{q_j}^2\) separately bounds the weighted sum.  Adding
these two bounds and enlarging \(C\) gives
\begin{equation}
 \sup_{p_j\leq s\leq q_j}d_s^2
 +\mu\sum_{k=p_j}^{q_j-1}\alpha_kd_k
 \leq C\left(
 4^{-j-1}\rho^2
 +\sup_{s\geq r\geq K_j}\left|
  \sum_{k=r}^{s-1}
   \alpha_k\langle\varepsilon_k,n_k\rangle
 \right|
 +\sum_{k=K_j}^{\infty}\alpha_k^2(1+|\varepsilon_k|^2)
 \right).
\label{eq:stoch_uniform_normal}
\end{equation}

Taking square roots in the bound for
\(\sup_{p_j\leq s\leq q_j}d_s^2\) and summing over \(j\) gives the first
conclusion in \eqref{eq:stoch_normal_summability}.  Retaining the bound for
\(\sum_{k=p_j}^{q_j-1}\alpha_kd_k\) and summing gives the second.  Both
follow from Lemma~\ref{lemma:active_stochastic_sparse}, since the summable
square-root tails in \eqref{eq:stoch_sparse_tails} also have summable
squares.
\end{proof}

\subsection{Proof of Lemma~\ref{lemma:active_stochastic_budget}}
\label{sec:proof_active_stochastic_budget}

\begin{proof}
Fix a bounded realization in \(\Omega_{\mathrm{act}}\) for which the
conclusions of Lemmas~\ref{lemma:active_stochastic_sparse} and
\ref{lemma:active_stochastic_normal} hold, having \(x^*\) as an accumulation
point and leaving \(\mathring B(x^*,\rho)\) infinitely often.
Let \(j_0\) and the late exit intervals be as in
Fact~\ref{fact:active_stochastic_exit_localization}.

\runinheading{Step~1: Projected descent and squared-gradient summability}
Since \(p_j\) and \(q_j\) are stopping times, \(\eta_k\) is
\(\mathcal F_k\)-measurable.  Equations
\eqref{eq:stoch_normal_coefficient} and
\eqref{eq:stoch_tangential_coefficient}, together with
\eqref{eq:active_local_uniform_constants}, show that \(n_k\) and \(t_k\)
are uniformly bounded and \(\mathcal F_k\)-measurable.  At every update index
of a late exit interval, \(n_k=x_k-y_k\) and
\(t_k=DP_A(x_k)^*\nabla_A f(y_k)\).

Increase this \(j_0\), if necessary, so that \(\alpha_k\leq\bar\alpha\) on every
late exit interval.  The definition of \(q_j\) and
\eqref{eq:stoch_exit_endpoint_localization} verify the local-segment hypotheses of
Proposition~\ref{proposition:active_inexact_local_estimates}.  Applying
\eqref{eq:active_meta_projected_descent} with \(u_k=v_k\) and \(r_k=0\), and
denoting its remainder constant by \(C_{\rm D}\), gives
\begin{align*}
 f(y_{k+1})-f(y_k)
 &\leq-\frac12\alpha_k|\nabla_A f(y_k)|^2
      -\alpha_k\langle\varepsilon_k,t_k\rangle
      +C_{\rm D}\bigl(\alpha_kd_k^2+\alpha_k^3
      +\alpha_k^2(|\varepsilon_k|+|\varepsilon_k|^2)\bigr).
\end{align*}
Since \(d_k^2\leq\varrho d_k\),
\(\alpha_k\leq\bar\alpha\leq1\), and
\(|\varepsilon_k|\leq1+|\varepsilon_k|^2\), the remainder is at most
\[
 C_{\rm D}\varrho\alpha_kd_k
 +3C_{\rm D}\alpha_k^2(1+|\varepsilon_k|^2).
\]
Thus \eqref{eq:stoch_projected_descent} holds with the fixed choice
\(C_{\rm T}:=C_{\rm D}\max\{\varrho,3\}\).

For \(J\geq j_0\), summing
\eqref{eq:stoch_projected_descent} over the late exit intervals gives
\begin{align*}
 \frac12\sum_{j=j_0}^J\sum_{k=p_j}^{q_j-1}
       \alpha_k|\nabla_A f(y_k)|^2
 &\leq\sum_{j=j_0}^J
       [f(y_{p_j})-f(y_{q_j})]
       -\sum_{j=j_0}^J\sum_{k=p_j}^{q_j-1}
       \alpha_k\langle\varepsilon_k,t_k\rangle\\
 &\quad+C_{\rm T}\sum_{j=j_0}^J\sum_{k=p_j}^{q_j-1}\alpha_kd_k
       +C_{\rm T}\sum_{j=j_0}^J\sum_{k=p_j}^{q_j-1}
       \alpha_k^2(1+|\varepsilon_k|^2).
\end{align*}

We now control the four terms on the right in their displayed order.  First,
the Lipschitz bound in \eqref{eq:active_local_uniform_constants} gives
\begin{align}
 |f(y_{p_j})-f(y_{q_j})|
 &\leq |f(x_{p_j})-f(x^*)|+|f(x_{q_j})-f(x^*)|
       +L(d_{p_j}+d_{q_j}) \notag\\
 &\leq2\sup_{k\geq K_j}|f(x_k)-f(x^*)|+L(d_{p_j}+d_{q_j}).
 \label{eq:stoch_endpoint_value_bound}
\end{align}
The right-hand side of
\eqref{eq:stoch_endpoint_value_bound} is summable by
Lemmas~\ref{lemma:active_stochastic_sparse} and
\ref{lemma:active_stochastic_normal}; hence the endpoint series is absolutely
convergent.

Second, \eqref{eq:active_local_uniform_constants} gives \(|t_k|\leq L^2\).
Since \(t_k\) is \(\mathcal F_k\)-measurable, the conditional variances satisfy
\[
 \mathbb E\!\left[
  \alpha_k^2\langle\varepsilon_k,t_k\rangle^2
  \mid\mathcal F_k\right]
 \leq\sigma^2L^4\alpha_k^2.
\]
The martingale convergence theorem
\cite[Theorem~2.15]{hall1980martingale} therefore shows that
\(\sum_k\alpha_k\langle\varepsilon_k,t_k\rangle\) converges almost surely.
Moreover, \(t_k=0\) between consecutive exit intervals, so
\[
 \sum_{j=j_0}^{J}\sum_{k=p_j}^{q_j-1}
 \alpha_k\langle\varepsilon_k,t_k\rangle
 =\sum_{k=p_{j_0}}^{q_J-1}
  \alpha_k\langle\varepsilon_k,t_k\rangle.
\]
Thus the martingale term in the summed inequality converges as
\(J\to\infty\).

Finally, \eqref{eq:stoch_normal_summability} makes the third term uniformly
bounded in \(J\), while
\eqref{eq:stoch_quadratic_noise_summability} does the same for the fourth.
The right-hand side is therefore uniformly bounded, and letting
\(J\to\infty\) gives
\begin{equation}
 \sum_{j=j_0}^{\infty}\sum_{k=p_j}^{q_j-1}
 \alpha_k|\nabla_A f(y_k)|^2<\infty.
\label{eq:stoch_projected_gradient_summability}
\end{equation}

\runinheading{Step~2: Coefficient summability and martingale tails}
At every update index of a late exit interval,
\(|n_k|^2=d_k^2\leq\rho d_k\) and
\(|t_k|^2\leq L^2|\nabla_A f(y_k)|^2\).  Therefore
\eqref{eq:stoch_normal_summability} and
\eqref{eq:stoch_projected_gradient_summability} show that both weighted
square sums are finite over all late exit intervals.  The intervals with
\(j<j_0\) contain only finitely
many indices on the fixed sample path, and \(n_k=t_k=0\) outside the exit
intervals.  Consequently, \(\sum_k\alpha_k|n_k|^2<\infty\) and
\(\sum_k\alpha_k|t_k|^2<\infty\).

For either \(w_k=n_k\) or \(w_k=t_k\), recall that
Lemma~\ref{lemma:martingale_tail_summability} defines
\[
 R_i^w:=\sup_{\ell\geq i}\left|
   \sum_{k=i}^{\ell}\alpha_k\langle\varepsilon_k,w_k\rangle
 \right|.
\]
Every block sum beginning after \(i\) is the difference of two partial sums
beginning at \(i\).  Hence
\[
 \sup_{s\geq r\geq i}\left|
   \sum_{k=r}^{s-1}\alpha_k\langle\varepsilon_k,w_k\rangle
 \right|\leq2R_i^w.
\]
On the event
\[
 \left\{\sum_k\alpha_k|n_k|^2<\infty,\quad
        \sum_k\alpha_k|t_k|^2<\infty\right\},
\]
the inequality \(a(1+\theta/2)>1\) in
\eqref{eq:stoch_theta_choice} allows
Lemma~\ref{lemma:martingale_tail_summability}(1) to be applied once with
\(w_k=n_k\) and once with \(w_k=t_k\).  The preceding bound and
\((u+v)^\theta\leq u^\theta+v^\theta\) for \(u,v\geq0\) then give, for
\(\mathcal T_i\) defined in
\eqref{eq:stoch_combined_martingale_tail},
\(\sum_i\alpha_i\mathcal T_i^\theta<\infty\) almost surely.  Since
\((\mathcal T_i)\) is nonincreasing and
\(\sum_i\alpha_i=\infty\) by \eqref{eq:stoch_step_consequences}, this
summability also forces \(\mathcal T_i\to0\).  Thus
\begin{equation}
 \mathcal T_i\to0,
 \qquad
 \sum_{i=0}^{\infty}\alpha_i\mathcal T_i^\theta<\infty.
 \label{eq:stoch_T_summability}
\end{equation}
The inequality \(a+\theta(2a-1)>1\) in
\eqref{eq:stoch_theta_choice} and
Lemma~\ref{lemma:martingale_tail_summability}(2) show that the
quadratic-noise tail satisfies
\begin{equation}
 \sum_{i=0}^{\infty}\alpha_i
 \left(\sum_{k=i}^{\infty}\alpha_k^2(1+|\varepsilon_k|^2)\right)^\theta
 <\infty
 \qquad\text{almost surely}.
\label{eq:stoch_quadratic_tail_summability}
\end{equation}

\runinheading{Step~3: Summability of the combined errors}
It remains to prove the two conclusions in
\eqref{eq:stoch_Q_summability}.  For the summability conclusion, set
\(D_i:=\sum_{k=i}^{\infty}
\alpha_k^2(1+|\varepsilon_k|^2)\).  By
\eqref{eq:stoch_Q_definition} and the subadditivity of
\(u\mapsto u^\theta\), its left-hand side satisfies
\begin{align*}
 \sum_{j=j_0}^{\infty}\sum_{i=p_j}^{q_j-1}
   \alpha_i\bigl(Q_i^\theta+Q_{i+1}^\theta\bigr)
 &\leq \sum_{j=j_0}^{\infty}\sum_{i=p_j}^{q_j-1}
   \alpha_i\bigl(d_i^{2\theta}+d_{i+1}^{2\theta}\bigr)\\
 &\quad+\sum_{j=j_0}^{\infty}\sum_{i=p_j}^{q_j-1}
   \alpha_i\bigl(D_i^\theta+D_{i+1}^\theta\bigr)\\
 &\quad+\sum_{j=j_0}^{\infty}\sum_{i=p_j}^{q_j-1}
   \alpha_i\bigl(\mathcal T_i^\theta+\mathcal T_{i+1}^\theta\bigr).
\end{align*}
We verify that the distance, quadratic-noise, and martingale
contributions on the right are finite, in the order displayed.

First, \(2\theta>1\) and \(d_i<\rho\) at every update index of a late
exit interval.  Hence
\[
 \sum_{j=j_0}^{\infty}\sum_{i=p_j}^{q_j-1}\alpha_i d_i^{2\theta}
 \leq \rho^{2\theta-1}
 \sum_{j=j_0}^{\infty}\sum_{i=p_j}^{q_j-1}\alpha_i d_i<\infty
\]
by \eqref{eq:stoch_normal_summability}.  To treat the shifted distance,
use that the distance function is \(1\)-Lipschitz and,
by \eqref{eq:active_local_uniform_constants}, the subgradient norms \(|v_i|\leq L\) on every late
exit interval.  Thus the update
gives
\[
 d_{i+1}\leq d_i+|x_{i+1}-x_i|
 \leq d_i+L\alpha_i(1+|\varepsilon_i|)
\]
at every update index of a late exit interval.  Increase \(j_0\) once more,
if necessary, so that
\(\alpha_i(1+|\varepsilon_i|)\leq1\) on all late exit intervals.  This is
possible because \(\alpha_i\to0\) by \eqref{eq:stoch_step_consequences} and
\(\alpha_i\varepsilon_i\to0\) by
\eqref{eq:stoch_vector_martingale_convergence}.  Since \(2\theta>1\), the
distances \(d_i\) are bounded there, and \eqref{eq:stoch_excursions} shows
that the update-index sets of the exit intervals are disjoint.  In particular,
\(\alpha_i[\alpha_i(1+|\varepsilon_i|)]^{2\theta}
\leq\alpha_i^2(1+|\varepsilon_i|)\), and hence
\begin{align*}
 \sum_{j=j_0}^{\infty}\sum_{i=p_j}^{q_j-1}
       \alpha_id_{i+1}^{2\theta}
 &\leq 2^{2\theta-1}\rho^{2\theta-1}
       \sum_{j=j_0}^{\infty}
       \sum_{i=p_j}^{q_j-1}\alpha_i d_i
       +2^{2\theta-1}L^{2\theta}\sum_{i=0}^{\infty}
          \alpha_i^2(1+|\varepsilon_i|)<\infty.
\end{align*}
because \(1+|\varepsilon_i|\leq2(1+|\varepsilon_i|^2)\),
\eqref{eq:stoch_normal_summability}, and
\eqref{eq:stoch_quadratic_noise_summability}.  Thus the distance
contribution is finite.

Second, \((D_i)\) is nonincreasing, and the update-index sets of the exit
intervals are disjoint.  Therefore
\[
 \sum_{j=j_0}^{\infty}\sum_{i=p_j}^{q_j-1}
   \alpha_i\bigl(D_i^\theta+D_{i+1}^\theta\bigr)
 \leq 2\sum_{i=0}^{\infty}\alpha_iD_i^\theta<\infty
\]
by \eqref{eq:stoch_quadratic_tail_summability}.  Thus the quadratic-noise
contribution is finite.  Finally,
\((\mathcal T_i)\) is nonincreasing by
\eqref{eq:stoch_combined_martingale_tail}, so the same disjointness gives
\[
 \sum_{j=j_0}^{\infty}\sum_{i=p_j}^{q_j-1}
   \alpha_i\bigl(\mathcal T_i^\theta+\mathcal T_{i+1}^\theta\bigr)
 \leq 2\sum_{i=0}^{\infty}\alpha_i\mathcal T_i^\theta<\infty
\]
by \eqref{eq:stoch_T_summability}.  Thus the martingale contribution is
finite, and the expanded bound proves the required summability.

It remains to consider the entrance indices.  By
\eqref{eq:stoch_Q_definition},
\[
 Q_{p_j}=d_{p_j}^2+D_{p_j}+\mathcal T_{p_j}.
\]
By \eqref{eq:stoch_excursions} and
\eqref{eq:active_projection_distance_bound},
\(d_{p_j}\leq2^{-j-1}\rho\to0\) and \(p_j\geq K_j\to\infty\).
Moreover, \(D_{p_j}\to0\) by
\eqref{eq:stoch_quadratic_noise_summability}, while
\(\mathcal T_{p_j}\to0\) by \eqref{eq:stoch_T_summability}.  Consequently,
\begin{equation}
 Q_{p_j}\to0,\qquad
 \sum_{j=j_0}^{\infty}\sum_{i=p_j}^{q_j-1}
 \alpha_i(Q_i^\theta+Q_{i+1}^\theta)<\infty.
\end{equation}
\end{proof}

\bibliographystyle{alpha}
\bibliography{mybib}

@inproceedings{GowerSebbouhLoizou2021,
  title={Sgd for structured nonconvex functions: Learning rates, minibatching and interpolation},
  author={Gower, Robert and Sebbouh, Othmane and Loizou, Nicolas},
  booktitle={International Conference on Artificial Intelligence and Statistics},
  pages={1315--1323},
  year={2021},
  organization={PMLR}
}

@inproceedings{GowerLoizouQianEtAl2019,
  title={SGD: General analysis and improved rates},
  author={Gower, Robert Mansel and Loizou, Nicolas and Qian, Xun and Sailanbayev, Alibek and Shulgin, Egor and Richt{\'a}rik, Peter},
  booktitle={International conference on machine learning},
  pages={5200--5209},
  year={2019},
  organization={PMLR}
}

@article{BottouCurtisNocedal2018,
  title={Optimization methods for large-scale machine learning},
  author={Bottou, L{\'e}on and Curtis, Frank E and Nocedal, Jorge},
  journal={SIAM review},
  volume={60},
  number={2},
  pages={223--311},
  year={2018},
  publisher={SIAM}
}

@article{Ochs2016,
  title={Local convergence of the heavy-ball method and iPiano for non-convex optimization},
  author={Ochs, Peter},
  journal={Journal of Optimization Theory and Applications},
  volume={177},
  number={1},
  pages={153--180},
  year={2018},
  publisher={Springer}
}

@incollection{combettes2008fejer,
  title={Fej{\'e}r monotonicity in convex optimization},
  author={Combettes, Patrick L},
  booktitle={Encyclopedia of optimization},
  pages={1016--1024},
  year={2008},
  publisher={Springer}
}

@article{bauschke2001weak,
  title={A weak-to-strong convergence principle for Fej{\'e}r-monotone methods in Hilbert spaces},
  author={Bauschke, Heinz H and Combettes, Patrick L},
  journal={Mathematics of operations research},
  volume={26},
  number={2},
  pages={248--264},
  year={2001},
  publisher={INFORMS}
}

@incollection{fejer1922lage,
  title={{\"U}ber die Lage der Nullstellen von Polynomen, die aus Minimumforderungen gewisser Art entspringen},
  author={Fej{\'e}r, Leopold},
  booktitle={Festschrift David Hilbert zu Seinem Sechzigsten Geburtstag am 23. Januar 1922},
  pages={41--48},
  year={1922},
  publisher={Springer}
}

@article{xiao2023stochastic,
  title={Stochastic subgradient methods with guaranteed global stability in nonsmooth nonconvex optimization},
  author={Xiao, Nachuan and Hu, Xiaoyin and Toh, Kim-Chuan},
  journal={Mathematical Programming},
  year={2026},
  month=aug,
  doi={10.1007/s10107-026-02409-2},
  url={https://doi.org/10.1007/s10107-026-02409-2}
}

@article{beck2009fast,
  title={A fast iterative shrinkage-thresholding algorithm for linear inverse problems},
  author={Beck, Amir and Teboulle, Marc},
  journal={SIAM journal on imaging sciences},
  volume={2},
  number={1},
  pages={183--202},
  year={2009},
  publisher={SIAM}
}

@article{tadic2015convergence,
  title={Convergence and convergence rate of stochastic gradient search in the case of multiple and non-isolated extrema},
  author={Tadi{\'c}, Vladislav B},
  journal={Stochastic Processes and their Applications},
  volume={125},
  number={5},
  pages={1715--1755},
  year={2015},
  publisher={Elsevier},
  doi={10.1016/j.spa.2014.11.001},
  url={https://doi.org/10.1016/j.spa.2014.11.001}
}

@article{goldstein1977optimization,
  title={Optimization of Lipschitz continuous functions},
  author={Goldstein, Allen A},
  journal={Mathematical Programming},
  volume={13},
  number={1},
  pages={14--22},
  year={1977},
  publisher={Springer}
}

@article{milzarek2023convergence,
  title={A Normal Map-Based Proximal Stochastic Gradient Method: Convergence and Identification Properties},
  author={Qiu, Junwen and Jiang, Li and Milzarek, Andre},
  journal={SIAM Journal on Optimization},
  volume={36},
  number={3},
  pages={1773--1804},
  year={2026},
  publisher={SIAM},
  doi={10.1137/25M1757332},
  url={https://doi.org/10.1137/25M1757332}
}

@book{borkar2008stochastic,
  title={Stochastic approximation: a dynamical systems viewpoint},
  author={Borkar, Vivek S},
  volume={100},
  year={2008},
  publisher={Springer}
}

@article{qiu2024convergence,
  title={Convergence of {SGD} with Momentum in the Nonconvex Case: A Time Window-Based Analysis},
  author={Qiu, Junwen and Ma, Bohao and Milzarek, Andre},
  journal={arXiv preprint arXiv:2405.16954},
  year={2024},
  note={arXiv:2405.16954v3},
  url={https://arxiv.org/abs/2405.16954}
}

@article{dereich2024convergence,
  title={Convergence of Stochastic Gradient Descent Schemes for {\L}ojasiewicz-Landscapes},
  author={Dereich, Steffen and Kassing, Sebastian},
  journal={Journal of Machine Learning},
  volume={3},
  number={3},
  pages={245--281},
  year={2024},
  doi={10.4208/jml.240109},
  url={https://doi.org/10.4208/jml.240109}
}

@article{trotman2020stratification,
  title={Stratification theory},
  author={Trotman, David},
  journal={Handbook of geometry and topology of singularities I},
  pages={243--273},
  year={2020},
  publisher={Springer}
}

@article{lojasiewicz1963propriete,
  title={Une propri{\'e}t{\'e} topologique des sous-ensembles analytiques r{\'e}els},
  author={\L{}ojasiewicz, Stanislaw},
  journal={Les {\'e}quations aux d{\'e}riv{\'e}es partielles},
  volume={117},
  number={87-89},
  year={1963}
}

@article{parusinski1994lipschitz,
  title={Lipschitz stratification of subanalytic sets},
  author={Parusi{\'n}ski, Adam},
  journal={Annales scientifiques de l'\'{E}cole normale sup{\'e}rieure},
  volume={27},
  number={6},
  pages={661--696},
  year={1994}
}

@article{davis2025active,
  title={Active manifolds, stratifications, and convergence to local minima in nonsmooth optimization},
  author={Davis, Damek and Drusvyatskiy, Dmitriy and Jiang, Liwei},
  journal={Foundations of Computational Mathematics},
  volume={26},
  pages={779--861},
  year={2026},
  publisher={Springer},
  doi={10.1007/s10208-025-09691-0}
}

@book{coste2000introduction,
  title={An introduction to o-minimal geometry},
  author={Coste, Michel},
  year={2000},
  publisher={Istituti editoriali e poligrafici internazionali Pisa}
}

@article{fischer2007minimal,
  title={{O-minimal $\Lambda^m$-regular stratification}},
  author={Fischer, Andreas},
  journal={Annals of Pure and Applied Logic},
  volume={147},
  number={1-2},
  pages={101--112},
  year={2007},
  publisher={Elsevier}
}

@article{nguyen2016lipschitz,
  title={Lipschitz stratifications in o-minimal structures},
  author={Nguyen, Nhan and Valette, Guillaume},
  journal={Annales Scientifiques de l'{\'E}cole Normale Sup{\'e}rieure},
  volume={49},
  number={2},
  pages={399--421},
  year={2016}
}

@article{alber1998projected,
  title={On the projected subgradient method for nonsmooth convex optimization in a Hilbert space},
  author={Alber, Ya I and Iusem, Alfredo N and Solodov, Mikhail V},
  journal={Mathematical Programming},
  volume={81},
  pages={23--35},
  year={1998},
  publisher={Springer}
}

@article{kushner1977general,
  title={General convergence results for stochastic approximations via weak convergence theory},
  author={Kushner, Harold J},
  journal={Journal of mathematical analysis and applications},
  volume={61},
  number={2},
  pages={490--503},
  year={1977},
  publisher={Elsevier}
}

@article{rios2022examples,
  title={Examples of pathological dynamics of the subgradient method for Lipschitz path-differentiable functions},
  author={Rios-Zertuche, Rodolfo},
  journal={Mathematics of Operations Research},
  volume={47},
  number={4},
  pages={3184--3206},
  year={2022},
  publisher={INFORMS}
}

@article{lewis2002active,
  title={Active sets, nonsmoothness, and sensitivity},
  author={Lewis, Adrian S},
  journal={SIAM Journal on Optimization},
  volume={13},
  number={3},
  pages={702--725},
  year={2002},
  publisher={SIAM},
  doi={10.1137/S1052623401387623},
  url={https://doi.org/10.1137/S1052623401387623}
}

@article{vaiter2017degrees,
  title={The Degrees of Freedom of Partly Smooth Regularizers},
  author={Vaiter, Samuel and Deledalle, Charles-Alban and Fadili, Jalal M. and
          Peyr{\'e}, Gabriel and Dossal, Charles},
  journal={Annals of the Institute of Statistical Mathematics},
  volume={69},
  number={4},
  pages={791--832},
  year={2017},
  doi={10.1007/s10463-016-0563-z},
  url={https://doi.org/10.1007/s10463-016-0563-z}
}

@article{bianchi2023stochastic,
  title={Stochastic subgradient descent escapes active strict saddles on weakly convex functions},
  author={Bianchi, Pascal and Hachem, Walid and Schechtman, Sholom},
  journal={Mathematics of Operations Research},
  volume={49},
  number={3},
  pages={1761--1790},
  year={2024},
  publisher={INFORMS},
  doi={10.1287/moor.2021.0194},
  url={https://doi.org/10.1287/moor.2021.0194}
}

@article{josz2023convergence,
  title={Convergence of the momentum method for semialgebraic functions with locally Lipschitz gradients},
  author={Josz, C{\'e}dric and Lai, Lexiao and Li, Xiaopeng},
  journal={SIAM Journal on Optimization},
  volume={33},
  number={4},
  pages={3012--3037},
  year={2023},
  publisher={SIAM},
  doi={10.1137/23M1545720},
  url={https://doi.org/10.1137/23M1545720}
}

@article{bolte2025inexact,
  title={Inexact subgradient methods for semialgebraic functions},
  author={Bolte, J{\'e}r{\^o}me and Le, Tam and Moulines, Eric and Pauwels, Edouard},
  journal={Mathematical Programming},
  pages={1--27},
  year={2025},
  publisher={Springer}
}

@article{ochs2014ipiano,
  title={{iPiano: Inertial proximal algorithm for nonconvex optimization}},
  author={Ochs, Peter and Chen, Yunjin and Brox, Thomas and Pock, Thomas},
  journal={SIAM Journal on Imaging Sciences},
  volume={7},
  number={2},
  pages={1388--1419},
  year={2014},
  publisher={SIAM}
}

@article{zavriev1993heavy,
  title={Heavy-ball method in nonconvex optimization problems},
  author={Zavriev, SK and Kostyuk, FV},
  journal={Computational Mathematics and Modeling},
  volume={4},
  number={4},
  pages={336--341},
  year={1993},
  publisher={Springer}
}

@article{attouch2009convergence,
  title={On the convergence of the proximal algorithm for nonsmooth functions involving analytic features},
  author={Attouch, Hedy and Bolte, J{\'e}r{\^o}me},
  journal={Mathematical Programming},
  volume={116},
  pages={5--16},
  year={2009},
  publisher={Springer}
}

@article{josz2023global,
  title={Global convergence of the gradient method for functions definable in o-minimal structures},
  author={Josz, C{\'e}dric},
  journal={Mathematical Programming},
  volume={202},
  number={1--2},
  pages={355--383},
  year={2023},
  publisher={Springer},
  doi={10.1007/s10107-023-01937-5},
  url={https://doi.org/10.1007/s10107-023-01937-5}
}

@inproceedings{sutskever2013importance,
  title={On the importance of initialization and momentum in deep learning},
  author={Sutskever, Ilya and Martens, James and Dahl, George and Hinton, Geoffrey},
  booktitle={International conference on machine learning},
  pages={1139--1147},
  year={2013},
  organization={PMLR}
}

@book{bertsekas2009convex,
  title={Convex optimization theory},
  author={Bertsekas, Dimitri P.},
  year={2009},
  publisher={Athena Scientific}
}

@inproceedings{mai2020convergence,
  title={Convergence of a stochastic gradient method with momentum for non-smooth non-convex optimization},
  author={Mai, Vien and Johansson, Mikael},
  booktitle={International Conference on Machine Learning},
  pages={6630--6639},
  year={2020},
  organization={PMLR}
}

@book{van1998tame,
  title={Tame topology and o-minimal structures},
  author={Van den Dries, Lou},
  volume={248},
  year={1998},
  publisher={Cambridge university press}
}

@book{bochnak2013real,
  title={Real algebraic geometry},
  author={Bochnak, Jacek and Coste, Michel and Roy, Marie-Fran{\c{c}}oise},
  volume={36},
  year={2013},
  publisher={Springer Science \& Business Media}
}

@article{van1996geometric,
  title={Geometric categories and o-minimal structures},
  author={Van den Dries, Lou and Miller, Chris},
  journal={Duke Mathematical Journal},
  volume={84},
  number={2},
  pages={497--540},
  year={1996},
  publisher={Duke University Press},
  doi={10.1215/S0012-7094-96-08416-1}
}

@article{attouch2013convergence,
  title={{Convergence of descent methods for semi-algebraic and tame problems: proximal algorithms, forward--backward splitting, and regularized Gauss--Seidel methods}},
  author={Attouch, Hedy and Bolte, J{\'e}r{\^o}me and Svaiter, Benar Fux},
  journal={Mathematical Programming},
  volume={137},
  number={1},
  pages={91--129},
  year={2013},
  publisher={Springer}
}

@article{li2023convergence,
  title={Convergence of random reshuffling under the kurdyka--{\l}ojasiewicz inequality},
  author={Li, Xiao and Milzarek, Andre and Qiu, Junwen},
  journal={SIAM Journal on Optimization},
  volume={33},
  number={2},
  pages={1092--1120},
  year={2023},
  publisher={SIAM}
}

@inproceedings{liapounoff1907probleme,
  title={Probl{\`e}me g{\'e}n{\'e}ral de la stabilit{\'e} du mouvement},
  author={Liapounoff, Alexandre},
  booktitle={Annales de la Facult{\'e} des sciences de Toulouse: Math{\'e}matiques},
  volume={9},
  pages={203--474},
  year={1907}
}

@article{ljung1977analysis,
  title={Analysis of recursive stochastic algorithms},
  author={Ljung, Lennart},
  journal={IEEE transactions on automatic control},
  volume={22},
  number={4},
  pages={551--575},
  year={1977},
  publisher={IEEE}
}

@article{benaim2005stochastic,
  title={Stochastic approximations and differential inclusions},
  author={Bena{\"\i}m, Michel and Hofbauer, Josef and Sorin, Sylvain},
  journal={SIAM Journal on Control and Optimization},
  volume={44},
  number={1},
  pages={328--348},
  year={2005},
  publisher={SIAM}
}

@article{kurdyka1994wf,
  title={wf-stratification of subanalytic functions and the \L{}ojasiewicz inequality},
  author={Kurdyka, Krzysztof and Parusinski, Adam},
  journal={Comptes rendus de l'Acad{\'e}mie des sciences. S{\'e}rie 1, Math{\'e}matique},
  volume={318},
  number={2},
  pages={129--133},
  year={1994}
}

@article{shor1962application,
  title={Application of the Gradient Method for the Solution of Network Transportation Problems. fJotes, Scientific seminar on theory and applications of cybernetics and operations research},
  author={Shor, NZ},
  journal={Kiev: Academy of Sciences USSR},
  year={1962}
}

@article{lojasiewicz1958,
  title={Division d'une distribution par une fonction analytique de variables réelles},
  author={S. \L{}ojasiewicz},
  journal={Comptes rendus hebdomadaires des séances de l'Académie des sciences. Paris},
  issue={246},
  pages={683-686},
  year={1958}
}

@article{bolte2020conservative,
  title={Conservative set valued fields, automatic differentiation, stochastic gradient methods and deep learning},
  author={Bolte, J{\'e}r{\^o}me and Pauwels, Edouard},
  journal={Mathematical Programming},
  pages={1--33},
  year={2020},
  publisher={Springer}
}

@article{dontchev1989error,
  title={Error estimates for discretized differential inclusions},
  author={Dontchev, Asen L and Farkhi, Elza M},
  journal={Computing},
  volume={41},
  number={4},
  pages={349--358},
  year={1989},
  publisher={Springer}
}

@article{davis2019active,
  title={Active strict saddles in nonsmooth optimization},
  author={Davis, Damek and Drusvyatskiy, Dmitriy},
  journal={arXiv preprint arXiv:1912.07146},
  year={2019}
}

@book{aubin1984differential,
  title={Differential inclusions: set-valued maps and viability theory},
  author={Aubin, J-P and Cellina, Arrigo},
  volume={264},
  year={1984},
  publisher={Springer-Verlag}
}

@article{bolte2007clarke,
  title={Clarke subgradients of stratifiable functions},
  author={Bolte, J{\'e}r{\^o}me and Daniilidis, Aris and Lewis, Adrian and Shiota, Masahiro},
  journal={SIAM Journal on Optimization},
  volume={18},
  number={2},
  pages={556--572},
  year={2007},
  publisher={SIAM},
  doi={10.1137/060670080}
}

@article{wang2005subdifferentiability,
  title={Subdifferentiability of real functions},
  author={Wang, Xianfu},
  journal={Real Analysis Exchange},
  volume={30},
  number={1},
  pages={137--171},
  year={2005},
  publisher={JSTOR}
}

@article{daniilidis2020pathological,
  title={Pathological subgradient dynamics},
  author={Daniilidis, Aris and Drusvyatskiy, Dmitriy},
  journal={SIAM Journal on Optimization},
  volume={30},
  number={2},
  pages={1327--1338},
  year={2020},
  publisher={SIAM}
}

@article{drusvyatskiy2015curves,
  title={Curves of descent},
  author={Drusvyatskiy, Dmitriy and Ioffe, Alexander D and Lewis, Adrian S},
  journal={SIAM Journal on Control and Optimization},
  volume={53},
  number={1},
  pages={114--138},
  year={2015},
  publisher={SIAM}
}

@article{absil2005convergence,
  title={Convergence of the iterates of descent methods for analytic cost functions},
  author={Absil, Pierre-Antoine and Mahony, Robert and Andrews, Benjamin},
  journal={SIAM Journal on Optimization},
  volume={16},
  number={2},
  pages={531--547},
  year={2005},
  publisher={SIAM}
}

@article{dontchev1992difference,
  title={Difference methods for differential inclusions: A survey},
  author={Dontchev, Asen and Lempio, Frank},
  journal={SIAM review},
  volume={34},
  number={2},
  pages={263--294},
  year={1992},
  publisher={SIAM}
}

@book{tarski1951decision,
  title={A decision method for elementary algebra and geometry: Prepared for publication with the assistance of JCC McKinsey},
  author={Tarski, Alfred},
  year={1951},
  publisher={Rand Corporation}
}

@article{seidenberg1954new,
  title={A new decision method for elementary algebra},
  author={Seidenberg, Abraham},
  journal={Annals of Mathematics},
  pages={365--374},
  year={1954},
  publisher={JSTOR}
}

@article{polyak1964some,
  title={Some methods of speeding up the convergence of iteration methods},
  author={Polyak, Boris T},
  journal={USSR Computational Mathematics and Mathematical Physics},
  volume={4},
  number={5},
  pages={1--17},
  year={1964},
  publisher={Elsevier}
}

@inproceedings{nesterov1983method,
  title={{A method for solving the convex programming problem with convergence rate $O(1/k^2)$}},
  author={Nesterov, Yurii E},
  booktitle={Dokl. akad. nauk Sssr},
  volume={269},
  pages={543--547},
  year={1983}
}

@article{robbins1951stochastic,
  title={A stochastic approximation method},
  author={Robbins, Herbert and Monro, Sutton},
  journal={The annals of mathematical statistics},
  pages={400--407},
  year={1951},
  publisher={JSTOR}
}

@book{rockafellar1998variational,
  title={Variational Analysis},
  author={Rockafellar, R. Tyrrell and Wets, Roger J.-B.},
  series={Grundlehren der mathematischen Wissenschaften},
  volume={317},
  year={1998},
  publisher={Springer},
  address={Berlin},
  doi={10.1007/978-3-642-02431-3},
  url={https://doi.org/10.1007/978-3-642-02431-3}
}

@article{burke2005robust,
  title={A robust gradient sampling algorithm for nonsmooth, nonconvex optimization},
  author={Burke, James V and Lewis, Adrian S and Overton, Michael L},
  journal={SIAM Journal on Optimization},
  volume={15},
  number={3},
  pages={751--779},
  year={2005},
  publisher={SIAM}
}

@article{duchi2018stochastic,
  title={Stochastic methods for composite and weakly convex optimization problems},
  author={Duchi, John C and Ruan, Feng},
  journal={SIAM Journal on Optimization},
  volume={28},
  number={4},
  pages={3229--3259},
  year={2018},
  publisher={SIAM}
}

@article{davis2020stochastic,
  title={Stochastic subgradient method converges on tame functions},
  author={Davis, Damek and Drusvyatskiy, Dmitriy and Kakade, Sham and Lee, Jason D},
  journal={Foundations of Computational Mathematics},
  volume={20},
  number={1},
  pages={119--154},
  year={2020},
  publisher={Springer},
  doi={10.1007/s10208-018-09409-5}
}

@article{candes2010power,
  title={The power of convex relaxation: Near-optimal matrix completion},
  author={Cand{\`e}s, Emmanuel J and Tao, Terence},
  journal={IEEE Transactions on Information Theory},
  volume={56},
  number={5},
  pages={2053--2080},
  year={2010},
  publisher={IEEE}
}

@article{recht2011simpler,
  title={A simpler approach to matrix completion},
  author={Recht, Benjamin},
  journal={Journal of Machine Learning Research},
  volume={12},
  number={Dec},
  pages={3413--3430},
  year={2011}
}

@book{clarke1990,
author={F. H. Clarke},
title={Optimization and Nonsmooth Analysis},
series={Classics in Applied Mathematics},
volume={5},
publisher={Society for Industrial and Applied Mathematics},
year={1990},
doi={10.1137/1.9781611971309}
}

@article{clarke1975,
author = {F. H. Clarke},
title = {{Generalized gradients and applications}},
journal = {Transactions of the American Mathematical Society},
year = {1975}
}

@article{dudek1994nonlinear,
  title={Nonlinear orthogonal projection},
  author={Dudek, Ewa and Holly, Konstanty},
  journal={Annales Polonici Mathematici},
  volume={59},
  number={1},
  pages={1--31},
  year={1994},
  doi={10.4064/ap-59-1-1-31}
}

@article{bolte2022long,
  title={Long term dynamics of the subgradient method for Lipschitz path differentiable functions},
  author={Bolte, J{\'e}r{\^o}me and Pauwels, Edouard and R{\'\i}os-Zertuche, Rodolfo},
  journal={Journal of the European Mathematical Society},
  year={2022}
}

@article{bierstone1988semianalytic,
  title={Semianalytic and subanalytic sets},
  author={Bierstone, Edward and Milman, Pierre D},
  journal={Publications Math{\'e}matiques de l'IH{\'E}S},
  volume={67},
  pages={5--42},
  year={1988}
}

@book{folland1999real,
  title={Real Analysis: Modern Techniques and Their Applications},
  author={Folland, Gerald B.},
  edition={2},
  year={1999},
  publisher={John Wiley \& Sons}
}

@article{lai2025diameter,
  title={On the diameter of subgradient sequences in o-minimal structures},
  author={Lai, Lexiao and Song, Mingzhi},
  journal={Submitted},
  note={arXiv:2511.06868v3},
  year={2025}
}

@inproceedings{deng2021minibatch,
  author={Deng, Qi and Gao, Wenzhi},
  title={Minibatch and Momentum Model-Based Methods for Stochastic Weakly Convex Optimization},
  booktitle={Advances in Neural Information Processing Systems},
  volume={34},
  pages={23115--23127},
  year={2021},
  url={https://proceedings.neurips.cc/paper_files/paper/2021/hash/c2c701fe341a7756ca7fd4eaa83ff63f-Abstract.html}
}

@book{hall1980martingale,
  author={Hall, Peter and Heyde, C. C.},
  title={Martingale Limit Theory and Its Application},
  year={1980},
  publisher={Academic Press},
  address={New York},
  series={Probability and Mathematical Statistics},
  isbn={9780123193506},
  doi={10.1016/C2013-0-10818-5}
}

@inproceedings{burkholder1972integral,
  author={Burkholder, Donald L. and Davis, Burgess J. and Gundy, Richard F.},
  title={Integral Inequalities for Convex Functions of Operators on Martingales},
  booktitle={Proceedings of the Sixth Berkeley Symposium on Mathematical Statistics and Probability},
  volume={2},
  pages={223--240},
  year={1972},
  publisher={University of California Press},
  url={https://digicoll.lib.berkeley.edu/record/113227}
}

@article{chzhen2026convergence,
  title={On convergence rates of subgradient descent on semialgebraic functions},
  author={Chzhen, Evgenii and Schechtman, Sholom},
  journal={arXiv preprint arXiv:2604.17060},
  year={2026}
}

@article{le2024nonsmooth,
  author={Le, Tam},
  title={Nonsmooth Nonconvex Stochastic Heavy Ball},
  journal={Journal of Optimization Theory and Applications},
  volume={201},
  number={2},
  pages={699--719},
  year={2024},
  publisher={Springer},
  doi={10.1007/s10957-024-02408-3},
  url={https://doi.org/10.1007/s10957-024-02408-3}
}

@inproceedings{sebbouh2021almost,
  author={Sebbouh, Othmane and Gower, Robert M. and Defazio, Aaron},
  title={Almost Sure Convergence Rates for Stochastic Gradient Descent and
         Stochastic Heavy Ball},
  booktitle={Proceedings of the Thirty-Fourth Conference on Learning Theory},
  editor={Belkin, Mikhail and Kpotufe, Samory},
  series={Proceedings of Machine Learning Research},
  volume={134},
  pages={3935--3971},
  year={2021},
  publisher={PMLR},
  url={https://proceedings.mlr.press/v134/sebbouh21a.html}
}

@inproceedings{simsekli2019tail,
  author={Simsekli, Umut and Sagun, Levent and Gurbuzbalaban, Mert},
  title={A Tail-Index Analysis of Stochastic Gradient Noise in Deep Neural
         Networks},
  booktitle={Proceedings of the 36th International Conference on Machine
             Learning},
  editor={Chaudhuri, Kamalika and Salakhutdinov, Ruslan},
  series={Proceedings of Machine Learning Research},
  volume={97},
  pages={5827--5837},
  year={2019},
  publisher={PMLR},
  url={https://proceedings.mlr.press/v97/simsekli19a.html}
}

@inproceedings{garg2021proximal,
  author={Garg, Saurabh and Zhanson, Joshua and Parisotto, Emilio and Prasad,
          Adarsh and Kolter, J. Zico and Lipton, Zachary C. and Balakrishnan,
          Sivaraman and Salakhutdinov, Ruslan and Ravikumar, Pradeep},
  title={On Proximal Policy Optimization's Heavy-Tailed Gradients},
  booktitle={Proceedings of the 38th International Conference on Machine
             Learning},
  editor={Meila, Marina and Zhang, Tong},
  series={Proceedings of Machine Learning Research},
  volume={139},
  pages={3610--3619},
  year={2021},
  publisher={PMLR},
  url={https://proceedings.mlr.press/v139/garg21b.html}
}

@inproceedings{zhang2020adaptive,
  author={Zhang, Jingzhao and Karimireddy, Sai Praneeth and Veit, Andreas and
          Kim, Seungyeon and Reddi, Sashank and Kumar, Sanjiv and Sra, Suvrit},
  title={Why Are Adaptive Methods Good for Attention Models?},
  booktitle={Advances in Neural Information Processing Systems},
  editor={Larochelle, Hugo and Ranzato, Marc'Aurelio and Hadsell, Raia and
          Balcan, Maria-Florina and Lin, Hsuan-Tien},
  volume={33},
  pages={15383--15393},
  year={2020},
  publisher={Curran Associates, Inc.},
  url={https://proceedings.neurips.cc/paper_files/paper/2020/file/b05b57f6add810d3b7490866d74c0053-Paper.pdf}
}

@inproceedings{hubler2024gradient,
  author={H{\"u}bler, Florian and Fatkhullin, Ilyas and He, Niao},
  title={From Gradient Clipping to Normalization for Heavy-Tailed {SGD}},
  booktitle={Proceedings of the 28th International Conference on Artificial
             Intelligence and Statistics},
  editor={Li, Yingzhen and Mandt, Stephan and Agrawal, Shipra and Khan, Emtiyaz},
  series={Proceedings of Machine Learning Research},
  volume={258},
  pages={2413--2421},
  year={2025},
  publisher={PMLR},
  url={https://proceedings.mlr.press/v258/hubler25a.html}
}

@article{yuan2006model,
  author={Yuan, Ming and Lin, Yi},
  title={Model Selection and Estimation in Regression with Grouped Variables},
  journal={Journal of the Royal Statistical Society: Series B (Statistical
           Methodology)},
  volume={68},
  number={1},
  pages={49--67},
  year={2006},
  doi={10.1111/j.1467-9868.2005.00532.x},
  url={https://doi.org/10.1111/j.1467-9868.2005.00532.x}
}

@article{tibshirani1996regression,
  author={Tibshirani, Robert},
  title={Regression Shrinkage and Selection Via the Lasso},
  journal={Journal of the Royal Statistical Society: Series B
           (Methodological)},
  volume={58},
  number={1},
  pages={267--288},
  year={1996},
  doi={10.1111/j.2517-6161.1996.tb02080.x},
  url={https://doi.org/10.1111/j.2517-6161.1996.tb02080.x}
}

@article{burke1991exact,
  author={Burke, James V.},
  title={An Exact Penalization Viewpoint of Constrained Optimization},
  journal={SIAM Journal on Control and Optimization},
  volume={29},
  number={4},
  pages={968--998},
  year={1991},
  doi={10.1137/0329054},
  url={https://doi.org/10.1137/0329054}
}

@article{meier2008group,
  author={Meier, Lukas and {van de Geer}, Sara and B{\"u}hlmann, Peter},
  title={The Group Lasso for Logistic Regression},
  journal={Journal of the Royal Statistical Society: Series B (Statistical
           Methodology)},
  volume={70},
  number={1},
  pages={53--71},
  year={2008},
  doi={10.1111/j.1467-9868.2007.00627.x},
  url={https://doi.org/10.1111/j.1467-9868.2007.00627.x}
}

@article{han1979exact,
  author={Han, S.-P. and Mangasarian, O. L.},
  title={Exact Penalty Functions in Nonlinear Programming},
  journal={Mathematical Programming},
  volume={17},
  pages={251--269},
  year={1979},
  doi={10.1007/BF01588250},
  url={https://doi.org/10.1007/BF01588250}
}
\end{document}